\documentclass[12pt]{amsart}

\usepackage{mathtools}
\mathtoolsset{showonlyrefs}

\usepackage{graphicx}
\usepackage{amsmath,amssymb,amsthm}
\usepackage[english]{babel}

\usepackage[latin1]{inputenc}
\usepackage{amsfonts}
\usepackage{amsmath}
\usepackage{latexsym}
\usepackage{amscd}
\usepackage{color}
\usepackage{shadow}
\usepackage{a4wide}
\usepackage{endnotes}
\usepackage{amsopn}
\usepackage{url}
\usepackage{dsfont}

\numberwithin{figure}{section}
\theoremstyle{plain}
\newtheorem{thm}{Theorem}[section]
\newtheorem{theoreme}{Theorem}[section]
\newtheorem{prop}[thm]{Proposition}
\newtheorem{proposition}[thm]{Proposition}

\newtheorem{lemma}[thm]{Lemma}
\numberwithin{equation}{section}

\theoremstyle{remark}

\newtheorem{rmq}[thm]{Remark}
\newtheorem{rem}[thm]{Remark}

\newtheorem{dfn}[thm]{Definition}

\newcommand{\eps}{\varepsilon}

\renewcommand\Re{\mathrm{Re}\,} \renewcommand\Im{\mathrm{Im}\,}
\newcommand\R{{\mathbb R}} \newcommand\N{{\mathbb N}}
 
\newcommand\C{{\mathbb C}} 
 
\renewcommand\P{{\mathbf P}}

\newcommand\Ai{{\mathrm Ai}}

 \newcommand{\cqfd}{\mbox{ }
\hfill$\Box$} 

\renewcommand{\Im}{  \text{Im}   }
\renewcommand{\Re}{  \text{Re}   }

 \def\imp{\Longrightarrow}

 \def\cdotv{\raise 2pt\hbox{,}}

\renewcommand\Re{\mathrm{Re}\,} \renewcommand\Im{\mathrm{Im}\,}

\makeatletter
\def\@tvsp{\mathchoice{{}\mkern-4.5mu}{{}\mkern-4.5mu}{{}\mkern-2.5mu}{}}
\def\ltrivert{\left|\@tvsp\left|\@tvsp\left|}
\def\rtrivert{\right|\@tvsp\right|\@tvsp\right|}
\makeatother

 \def\imp{\Longrightarrow}

 \def\cdotv{\raise 2pt\hbox{,}}

\newcommand{\anh}{{a,N,h}}
\newcommand{\anb}{{a,N,\hbar}}
\newcommand{\phiN}{\phi_{a,N,\hbar}}
\newcommand{\anl}{{a,N,\lambda}}

\begin{document}

\title[Dispersion for the wave equation inside strictly convex domains]{Dispersion for the wave equation inside strictly
  convex domains I: the Friedlander model case}
  \author{Oana Ivanovici}
  \address{Laboratoire J. A. Dieudonn\'e, UMR CNRS 7351\\
    Universit\'e de Nice Sophia-Antipolis\\
    Parc Valrose\\
    06108 Nice Cedex 02\\
    France} \email{oana.ivanovici@unice.fr}

  \author{Gilles Lebeau}
  \address{Laboratoire J. A. Dieudonn\'e, UMR CNRS 7351\\
    Universit\'e de Nice Sophia-Antipolis\\
    Parc Valrose\\
    06108 Nice Cedex 02\\
    France and Institut universitaire de France} \email{gilles.lebeau@unice.fr}

  \author{Fabrice Planchon}
  \address{Laboratoire J. A. Dieudonn\'e, UMR CNRS 7351\\
    Universit\'e de Nice Sophia-Antipolis\\
    Parc Valrose\\
    06108 Nice Cedex 02\\
    France and Institut universitaire de France}
  \email{fabrice.planchon@unice.fr} \thanks{The third author was
    partially supported by A.N.R. grant SWAP} 
\date{\today}

\maketitle

\begin{abstract}
We consider a model case for a strictly convex domain
$\Omega\subset\mathbb{R}^d$ of dimension $d\geq 2$ with smooth
boundary $\partial\Omega\neq\emptyset$ and we describe dispersion for
the wave equation with Dirichlet boundary conditions. More
specifically, we obtain the optimal fixed time decay rate for the
smoothed out Green function: a $t^{1/4}$
loss occurs with respect to the boundary less case, due to repeated
occurrences of swallowtail
type singularities in the wave front set.
\end{abstract}

\section{Introduction}
Let us consider solutions of the linear wave
equation on a manifold $(\Omega,g)$, with (possibly empty) boundary $\partial
\Omega$: 
\begin{equation} \label{WE} 
\left\{ \begin{array}{l}
   (\partial^2_t-
 \Delta_g) u(t, x)=0,  \;\; x\in \Omega
\\ u(0, x) = u_0(x), \; \partial_t u(0,x)=u_1(x),
\\ 
 u(t,x)=0,\quad x \in \partial \Omega\,, \end{array} \right.
 \end{equation}
where $\Delta_g$ denotes the Laplace-Beltrami operator on $\Omega$.

When dealing with the Cauchy problem for nonlinear wave equations, one
starts with perturbative techniques and faces the difficulty of controlling the size of
solutions to the linear equation in terms of the size of the initial
data. Of course, one has to quantify this notion of size by specifying
a suitable (space-time) norm.  It turns out that, especially at low
regularities, mixed norms of type $L^p _t L^q _x$ are particularly
useful. Moreover, the arguments leading to such estimates turn out to
be useful when considering spectral cluster estimates, which are of
independent interest (see \cite{smso06}).

On any smooth
Riemannian manifold without boundary, the following set of so-called
Strichartz estimates
holds for solutions of the wave
equation \eqref{WE} (for $T<\infty$)
\begin{equation}\label{SE}
\|u\|_{L^q(0,T) L^r(\Omega)}\leq
C_T \bigl(\,||u_0||_{\dot{H}^{\beta}(\Omega)} +
||u_1||_{\dot{H}^{\beta-1}(\Omega)} \bigr)\,,
\end{equation}
where, if $d$ denotes the dimension of the manifold, we have
$\beta=d(\frac 12 -\frac 1r)-\frac 1q$ (which is consistent with
scaling) and where the pair $(q,r)$ is wave-admissible, i.e.
\begin{equation}\label{admw}
q\geq 2\,,\quad \frac 2q +\frac{d-1}{r}\leq \frac{d-1}{2}\, \quad (
q>2 \text{ if } d=3
 \text{ and } q\geq 4 \text{ if } d=2).
\end{equation}
When equality holds in \eqref{admw} we say that the pair $(q,r)$ is sharp wave-admissible.
Here $\dot H^{\beta}$ denotes the (homogeneous) $L^2$ Sobolev space over $\Omega$. Such inequalities were long ago established  for
Minkowski space, where they hold globally in time ($T=+\infty$). Their
local in time version may be generalized to any $(\Omega,
g)$ where $g$ is smooth (thanks to the finite speed of propagation),
while global in time estimates require stronger geometric requirements
of global nature on the metric.

The canonical path leading to such Strichartz estimates is to obtain a
stronger, fixed time, dispersion estimate, which is then combined with
energy conservation, interpolation and $TT^\star$ arguments to obtain
\eqref{SE}. Let us denote by $e^{\pm it\sqrt{-\Delta_{\mathbb{R}^d}}}$ the
half-wave propagators in flat space, and $\psi\in C_{0}^\infty
(]0,\infty[)$.  The following dispersion inequality holds:
\begin{equation}\label{disprd}
\|\psi(-h^{2}\Delta_{\mathbb{R}^d})e^{\pm it\sqrt{-\Delta_{\mathbb{R}^d}}}\|_{L^1(\mathbb{R}^d)\rightarrow L^{\infty}(\mathbb{R}^d)}\leq C(d)h^{-d}\min\{1,(h/|t|)^{\frac{d-1}{2}}\}.
\end{equation}
Our aim in the present paper is to obtain these estimates inside domains. In fact, \cite{gle06} outlines a roadmap to
prove such a  dispersion estimate, on a finite time interval, for
solutions of \eqref{WE} inside a strictly convex domain $(\Omega,g)$
of dimension $d\geq 2$. A complete description of the geometry of the
(semi-classical) wave front set is provided for the solution to \eqref{WE} with
initial data $(u_0, u_1)=(\delta_a, 0)$, where $a\in\Omega$ is a point
sufficiently close to the boundary (depending on the scale $h$). This
wave front set has caustics developing in arbitrarily small times and
this induces a loss of $1/4$ in \eqref{disprd} for the $h/|t|$ factor.

In the present work, we aim at completing the roadmap by constructing
a suitable parametrix for such a solution and then proving dispersion
for the approximated solution. It should be noted that parametrices
have been available for the boundary value problem for a long time,
see \cite{meta85,mesj78,esk77}, as a crucial tool
to establish propagation of singularities for the wave equation on
domains. However, while efficient at proving that singularities travel
along the (generalized) bi-characteristic flow, they do not seem
strong enough to obtain dispersion, at least in the presence of
gliding rays. In the outside of a strictly convex obstacle (no gliding rays), the
Melrose-Taylor parametrix was utilized in \cite{smso95} to
prove Strichartz estimates hold as in the $\mathbb{R}^d$ case. All
other positive results (\cite{blsmso08} and references therein) rely instead on reflecting the
metric across the boundary and considering a boundary less manifold
with a Lipschitz metric across an interface, and then using the
machinery originally developed for low regularity metrics
\cite{sm98,tat02} and spectral cluster estimates \cite{smso06}. Such constructions do away with multiply
reflected rays by suitable microlocalizations: one ends up working on
a possibly very small time interval, depending on the incidence of the
wave packet under consideration, such that all corresponding rays are
only reflected once. Summing these intervals induces (scale-invariant)
losses, which get worse with dimension; while Strichartz estimates are
obtained in a more direct way in \cite{blsmso08}, one can observe that the corresponding
dispersion estimate would have at most $1/t$ decay for $d\geq 4$, as
the argument is blind to the full dispersion which should occur in
tangential directions. On the other hand, negative results were
obtained in \cite{doi,doi2}, where a special solution is
constructed, propagating a cusp across multiple reflections and
providing a counterexample to the sharp Strichartz estimates
\eqref{SE}, for $r>4$. This special solution is constructed via a
microlocal parametrix which utilizes the Melrose one, and our present
construction generalizes this special example while retaining most of
its useful features.\\

Before stating our main result, we briefly introduce the
Friedlander's model domain of the half-space
$\Omega_d=\{(x,y)|x>0, y\in\mathbb{R}^{d-1}\}$ with Laplace operator given by
$$\Delta_g=\partial^{2}_{x}+(1+x)\Delta_{y}.$$ 

By rotational symmetry, we
will eventually reduce to the two-dimensional case $\Omega_2$.
\begin{rmq}
For the metric $g=dx^2+(1+x)^{-1}dy^2$, the Laplace-Beltrami operator is
$\triangle_{g,0}=(1+x)^{1/2}\partial_{x}(1+x)^{-1/2}\partial_{x}+(1+x)\Delta_{y}$
which is self adjoint with the volume form $\sqrt{\text{det}
  g}\,dxdy=(1+x)^{-1/2}dxdy$. The Friedlander's model uses instead
the Laplace operator associated to the Dirichlet form $\int \vert
\nabla_{g}u\vert^2\,dxdy
=\int (\vert \partial_{x}u\vert^2+(1+x)\vert\partial_{y}u\vert^2)dxdy$ and is self adjoint with  volume form $dxdy$. As a model, the Friedlander operator $\Delta_{g}$ is better than the Laplacian
$\Delta_{g,0}$ since it allows explicit computations. Clearly,
manifold $({\Omega_d},g)$ is a strictly convex domain : In fact, on the geodesic flow
starting at $x=0,y=y_{0}, \xi_{0}^2+\eta_{0}^2=1,\xi_{0}\in ]0,1[$, one has
$x(s)=2s\xi_{0}-s^2\eta_{0}^2$. Moreover, $(\Omega_2,g)$ may be seen
as a simplified model for the disk
$D(0,1)$ with polar coordinates $(r,\theta)$, where $r=1-x/2$ and
$\theta=y$.  Multiply reflected light rays become periodic curves in
the $y$ variable, as illustrated on Figure \ref{rays}.
\end{rmq}
\begin{rmq}
We will always work with the Dirichlet boundary condition. The Neumann boundary condition
can be handled exactly in the same way, providing the same results:
one simply modifies the reflexion coefficient in our parametrix construction and replaces zeros of the Airy
function $Ai$ by zeros of its derivative $Ai'$.   
\end{rmq}
\begin{figure}
	\begin{center}
		\includegraphics[height=3.5cm]{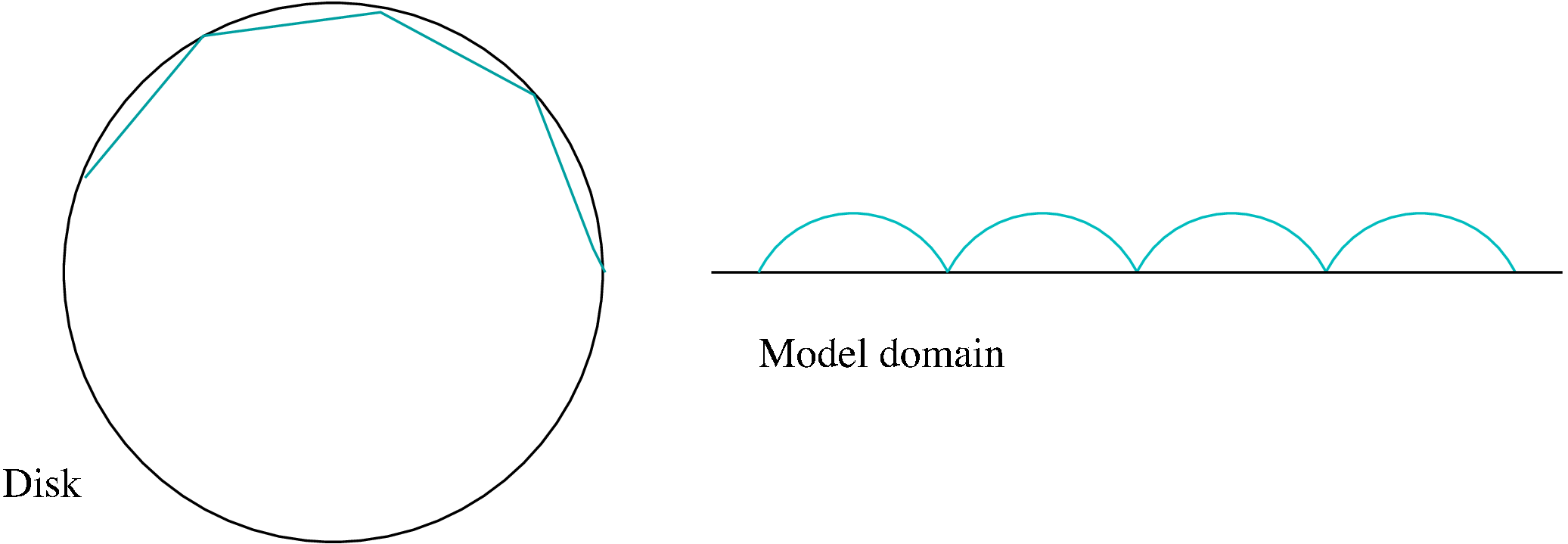}
	\end{center}
	\caption{Light rays in the Friedlander model}
        \label{rays}
\end{figure}
We are now in a position to state our main result.
\begin{thm}\label{disper}
Let $d\geq 2$. There exists 
$C>0,T_{0}>0$ such that for every $a\in ]0,1]$, $h\in (0,1]$ and $t\in (0,T_{0}]$, the solution $u_a$   to \eqref{WE} with data
$(u_0,u_1)=(\delta_a,0)$  where   $\delta_a$ is the Dirac mass at point $(a,0,\cdots,0)\in \Omega_d$, satisfies   
\begin{equation}\label{dispco}
|\psi(-h^2 \triangle_{g})u_{a}(t,x)|\leq C h^{-d}\min(1,(h/t)^{\frac{d-2}{2}+\frac{1}{4}}).
\end{equation}
\end{thm}
The dispersion estimate \eqref{dispco} may be compared to
\eqref{disprd}: we notice a $1/4$ loss in the $h/t$ exponent, which
we may informally relate to the presence of caustics in arbitrarily
small times if $a$ is small. Such caustics occur because optical rays are no longer
diverging from each other in the normal direction, where less
dispersion occurs as compared to the $\mathbb{R}^d$ case. 
We will prove in fact a slightly better estimate than \eqref{dispco}: the 
$(h/t)^{1/4}$ factor may be replace by $h^{1/4}+(h/t)^{1/3}$ for $a\leq h^{1/2}$
( Proposition \ref{propadisp})
and by $(h/t)^{1/2}+a^{1/8}h^{1/4}$ for $a\geq h^{4/7 -\varepsilon}$ ( Theorem \ref{thL1}). In fact, we
can track the caustics and therefore our estimate is optimal for $a\geq h^{4/7 -\varepsilon}$.
\begin{thm}\label{disperoptimal}
Let $d\geq 2$ and $u_a$ be the solution to \eqref{WE} with data
$(u_0,u_1)=(\delta_a,0)$. Let $h\in (0,1]$ and $a\geq h^{4/7
  -\varepsilon}$. There exists a constant $C>0$ and a finite sequence
$(t_n)_n$, $1\leq n\leq \min(a^{-1/2}, a^{1/2}h^{-1/3})$ with $t_n\sim
4n \sqrt a$, such that
\begin{equation}\label{dispcooptimal}
h^{-d}(h/t_n)^{\frac{d-2}{2}}n^{-1/4} a^\frac 1 8 h^{1/4}\sim a^\frac
1 4 h^{-d}(h/t_n)^{\frac{d-2}{2}+\frac 1 4}\lesssim
 |\psi(-h^2 \triangle_{g})u_{a}(t_n,a)|\,.
\end{equation}
\end{thm}
As a byproduct, we get that even for $t\in ]0,T_{0}]$ with $T_{0}$ small,
the $1/4$ loss is unavoidable for $a$ comparatively small to $T_0$ and
independent of $h$. We will see soon that this optimal loss is due to swallowtail type singularities in the wave
front set of $u_a$. 
\begin{rmq}
  Note that when $a=h^\gamma$, where we gain from the
  factor $a^{1/8}$, the loss in \eqref{dispcooptimal}
  is still greater than the usual dispersive estimate in the flat
  case: this requires $\gamma >2/3$ whereas we have $\gamma
  >4/7-\varepsilon$. Moreover, in this range, the loss is also greater
  than the loss which would occur if we had only cusp singularities.
\end{rmq}
\begin{rmq}
It follows from our proof that Theorem \ref{disper} and \ref{disperoptimal} holds true 
if one replaces $\psi(-h^2 \triangle_{g})u_{a}(t,x)$ by 
$\psi(hD_{t})e^{\pm it\sqrt{-\triangle_{g}}}\delta_{x=a,y=0}$ with $\psi\in C_{0}^\infty(\R^*)$.
\end{rmq}

As a consequence of \eqref{dispco} and classical arguments, we
obtain the following set of Strichartz estimates.
\begin{thm}\label{thStri}
  Let $u$ be a solution of \eqref{WE} on the model domain $\Omega_d$,
  $d\geq 2$. Then there exists $T$ such that
\begin{equation}\label{SEF}
\|u\|_{L^q(0,T) L^r(\Omega)}\leq
C_T \bigl(\,||u_0||_{\dot{H}^{\beta}(\Omega)} +
||u_1||_{\dot{H}^{\beta-1}(\Omega)} \bigr)\,,
\end{equation}
for $(d,q,r)$ satisfying 
\[
\frac{1}{q}\leq(\frac{d-2}{2}+\frac 14)(\frac{1}{2}-\frac{1}{r})\,,
\]
and $\beta$ is dictated by scaling.
\end{thm}
In dimension $d=2$ the known range of admissible indices for which
sharp Strichartz hold is in fact slightly larger, see
\cite{blsmso08}. However, in larger dimensions $d\geq 3$, Theorem
\ref{thStri} improves the range of indices for which sharp Strichartz
do hold, and does so in a uniform way with respect to dimension, in
contrast to \cite{blsmso08}. On the other hand, our results are, for now,
restricted to a model case of strictly convex domain, while \cite{blsmso08}
applies to any domain. One may use
the model case analysis to extend estimates to any smooth strictly
convex domain, as in the counterexample situation
\cite{doi2}. This issue will be addressed elsewhere.
\begin{rmq}
  One conjectures that the loss in Strichartz estimates in
  \cite{doi} are optimal. This would heuristically match a $1/6$
  loss in the dispersion estimate. We plan to address this issue in
  future work, by proving that the worst time-space
  points $(t_n,a)$ may be suitably averaged over.
\end{rmq}
One may then make good use of such Strichartz estimates for the local
(and global) Cauchy theory of nonlinear wave equations. We provide one simple example.
\begin{thm}
The energy critical wave equation $\Box_g u+|u|^{\frac{4}{d-2}} u=0$ with data $(u_0,u_1)\in H^1_0(\Omega_d)\times L^2(\Omega_d)$ has unique global in time
    solutions for $3\leq d\leq 6$.
\end{thm}
In the small data case, the result follows directly from the
previously obtained set of Strichartz estimates. Appendix \ref{nonlin}
provides details on how to combine these new estimates with arguments
from \cite{BLP} to obtain the large data case.
\subsection{Light propagation, heuristics and degenerate oscillatory integrals}
In \cite{gle06} the second author sketched the main steps of a proof
of \eqref{dispco} and
gave a full description of the geometry of the wave front set. In this
work, we provide a complete construction of a suitable parametrix for
the wave equation, which we then utilize to obtain decay estimates by
(degenerate) stationary phases.

Recall that, at time $t>0$, one expects the wave propagating from the source of light
to be highly concentrated around the sphere of radius $t$. For a
variable coefficients metric, one can make good of this heuristic as
long as two different light rays emanating from the source do not
cross: in other words, as long as $t$ is smaller than the injectivity
radius. One may then construct parametrices using oscillatory
integrals, where the phase encodes the geometry of the wave front.

In our situation, the geometry of the wave front becomes singular in
arbitrarily small times, depending on the frequency of the source and its distance
to the boundary. In fact, a caustic appears right
between the first and the second reflexion of the wave front, as
illustrated on Figure \ref{figpropS} and Figure \ref{figqueueS} (which
is a zoomed version at the relevant time scale). Therefore, we are to
investigate concentration phenomena (``caustics'') that may occur near
the boundary.  Geometrically, caustics are defined as envelopes of
light rays coming from our source of light. Each ray is tangent to the
caustic at a given point. If one assigns a direction on the caustic,
it induces a direction on each ray. Each point outside the caustic
(and in the sunny side of the caustic)
lies  on a ray which has left the caustic and also lies on a ray
approaching the caustic. Each curve of constant phase has a cusp where
it meets the caustic.

At the caustic point we expect light to be singularly
intense. Analytically, caustics can
be characterized as points were usual bounds on oscillatory integrals
are no longer valid. Oscillatory integrals with caustics have enjoyed
much attention: their asymptotic behavior is known to be driven by the
number and the order of those of their critical points which
are real. Let us consider an
oscillatory integral

\begin{equation}\label{oscint}
u_{h}(z)=\frac{1}{(2\pi h)^{1/2}}\int_{\zeta}e^{\frac ih\Phi(z,\zeta)}g(z,\zeta,h) d \zeta,\quad z\in \mathbb{R}^d,\quad \zeta\in\mathbb{R},\quad h\in (0,1].
\end{equation}
We assume that $\Phi$ is smooth and that $g(.,h)$ is compactly
supported in $z$ and in $\zeta$. If there are no critical points of
the map $\zeta\rightarrow \Phi(z,\zeta)$, so that
$\partial_{\zeta}\Phi\neq 0$ everywhere in an open neighborhood of the
support of $g(.,h)$, then repeated integration by parts (i.e. non
stationary phase) yields that
$|u_h(z)|=O(h^N)$, for any $N>0$.

If there are non-degenerate critical points, where
$\partial_{\zeta}\Phi=0$ but $\det (\partial^2_{\zeta_j
  \zeta_k}\Phi)\neq 0$, then the method of stationary phase applies
and yields $\|u_h(z)\|_{L^{\infty}}=O(1)$. The corresponding canonical
form is a Gaussian phase.

If there are degenerate critical points, we define them to be
caustics, as $\|u_h(z)\|_{L^{\infty}}$ is no longer uniformly
bounded. The order of a caustic $\kappa$ is defined as the infimum of
$\kappa'$ such that $\|u_h(z)\|_{L^{\infty}}=O(h^{-\kappa'})$. 

The most simple degenerate phase beyond the Gaussian is
$\Phi_F(z,\zeta)=\frac{\zeta^3}{3}+z_1\zeta+z_2$, which corresponds to a fold
with order $\kappa=\frac 16$. A typical example is the Airy function. The caustic is given by
$z_{1}=0$ and the illuminated side is $z_{1}<0$. The next canonical form  is given by a phase function which is a polynomial of
degree $4$, namely $\Phi_C(z,\zeta)=\frac{\zeta^4}{4}+z_1\frac{\zeta^2}{2}+z_2\zeta+z_3$
whose order is $\kappa=\frac 14$; its associated integral is called
Pearcey's function and it produces a cusp singularity on the caustic which is parametrized by
$z_{1}=-3s^2, z_{2}=2s^3$.

 Finally, we
conclude this brief overview with the swallowtail integral (which is
an oscillatory integral with four coalescing saddle points) whose
canonical form is given by a polynomial of degree $5$,
$\Phi_S(z,\zeta)=\frac{\zeta^5}{5}+z_1\frac{\zeta^3}{3}+z_2\frac{\zeta^2}{2}+z_3\zeta+z_4$:
the caustic surface of the swallowtail is defined by the condition
that two or more \textit{real} saddle points are equal: it is pictured
on Figure \ref{picswallow}. In the event
that two simple saddle points undergo confluence when $z\rightarrow
z_0$, then the uniform asymptotic behavior of \eqref{oscint} contains
terms involving the Airy function and its derivatives multiplied by
powers of $h^{-\frac 12+\frac 13}$; the caustic surface is smooth
($z_1<0$ on the figure). 
\begin{figure}
	\begin{center}
		\includegraphics[height=6cm]{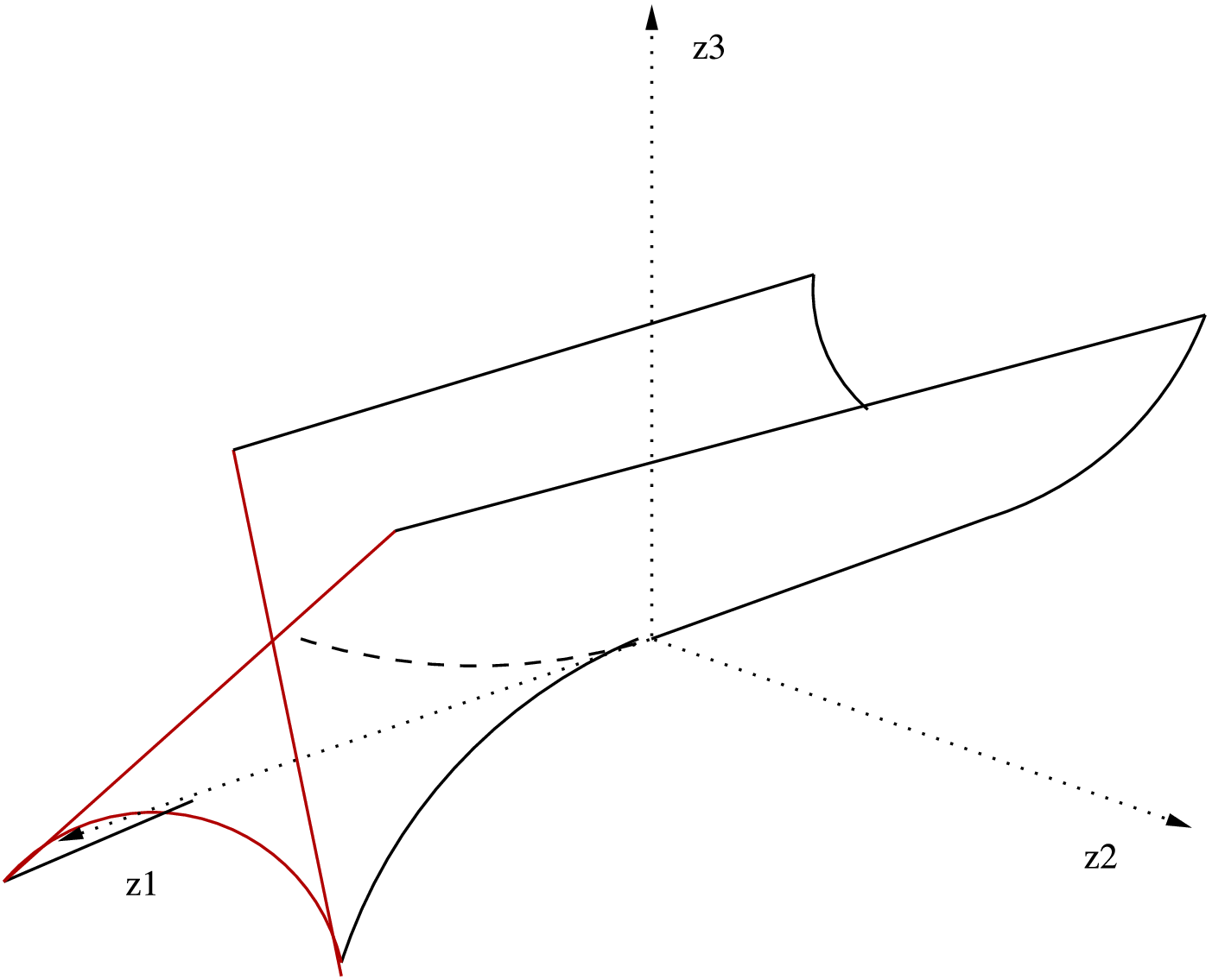}
	\end{center}
	\caption{The caustic for the swallowtail catastrophe.}
\label{picswallow}
\end{figure}
If three simple saddles coalesce
as $z\rightarrow z_0$, then the uniform asymptotic behavior of
\eqref{oscint} can be described by terms containing the Pearcey
function and its first-order derivatives, each multiplied by a power
of $h^{-\frac 12+\frac 14}$; the caustic surface has cusps (two of
them in the $z_1>0$ region on the figure). The swallowtail enters the picture when
four simple saddle points of \eqref{oscint} undergo confluence as
$z\rightarrow z_0$ (which is $z_0=0$ on the figure). We refer to
\cite{Berry} for a very nice presentation, both from the mathematical
and the physical point of vue, of degenerate oscillatory integrals and
their relation to Thom's theory of catastrophes.

 Such integrals will play a crucial role in the
proof of Theorem \ref{disper}. Between two consecutive reflections of
the wave propagating along the boundary, we shall construct a
parametrix of the form
$$
u(z,h)=\frac 1 {h^{2}} \int_{\R^2} e^{\frac i h \Phi(z,\zeta,\eta)}
g(z,\eta,\zeta,h)\,d\zeta d\eta,
$$
 where the phase is essentially $\Phi(z,\zeta,\eta)\approx\eta\Phi_C(z,\zeta)$, with $z=(t,x,y)$,
$\eta/h$ is the Fourier variable associated to the tangential variable
$y$ and
$\zeta=\xi/\eta$ where $\xi/h$ is the Fourier variable associated to
the normal variable $x$. Note we may restrict to $\eta\in (1/2,2)$ which
corresponds precisely to waves propagating along the boundary and
explains the $(\eta,\zeta)$ parametrization for the oscillatory integral. For a particular value $z_S$
of $z=(t,x,y)$, this phase will have a saddle point of order $4$; it corresponds to $
\partial_\eta \Phi=0=\Phi_C(z_S,\zeta)$ and $\partial_\zeta
\Phi_C=\partial^2_\zeta \Phi_C=\partial^3_\zeta\Phi_C=0$: the
geometric picture is that of a swallowtail singularity, but the decay
loss is that of the Pearcey's integral, i.e. $h^{1/4}$. For $z\neq
z_S$, our oscillatory integral will have
only critical points of order at most $3$, corresponding to $\partial_\eta \Phi=0=\Phi_C(z_S,\zeta)$ and $\partial_\zeta
\Phi_C=\partial^2_\zeta \Phi_C=0$: the picture is, at worst, that of cusps and
the loss is that of the Airy function, i.e. $h^{1/6}$. Finally, we
notice Figures \ref{picswallow} and \ref{figqueueS} picture the same
singularity formation: in \ref{picswallow}, up to translations,
$z_1=t$, $z_2=-x$ and $z_3=y$; $z_1<0$ corresponds to the (smooth)
refocusing wave front in  the left part of 
\ref{figqueueS} while two cusps form on the right part after the
swallowtail singularity.

\subsection{An outline of the proof}
Let us mention the main ideas of the proof of Theorem \ref{disper}. 
First, we may reduce to the two dimensional case, as the tangential
directions will produce the usual decay factor when we integrate them
out, see Section \ref{dim3}.

 Let $h\in (0,1]$ be a small parameter
($1/h$ will later be the spectral frequency) and $0< a\ll1$ the
distance of the source to the boundary. We assume $a$ to be small as
we are interested in highly reflected waves, which we do not observe
if the waves do not have time to reach the boundary.

From the spectral analysis which will be recalled in Section \ref{wgm}, we have an
explicit representation for the Green function associated to the
half-wave initial value problem with a Dirac at $(a,b)$ as initial
condition at time $s$:
\begin{equation}
  \label{eq:green}
  G((x,y,t),(a,b,s))=\sum_{k\geq 1} \int_{\R} e^{\pm i(t-s)\sqrt
 { \lambda_k(\eta)}} e^{i(y-b)\eta} e_k(x,\eta)e_k(a,\eta) \,d\eta
\end{equation}
where $\lambda_k(\eta)=\eta^2+\eta^{4/3}\omega_{k}$, with $-\omega_{k}$ a zero of the Airy function  and the $e_k(x,\eta)$ are explicit, real-valued functions which
are defined in Section \ref{wgm}. We now record several remarks that will
be of help later and relate to various phase space localizations.
\begin{rmq}\label{remspectal}
  We may perform a spectral localization at
  $\lambda_k(\eta)\sim h^{-2}$, which corresponds to inserting a
  smooth, compactly supported away from zero $\psi_{2}(h \sqrt{\lambda_k(\eta)})$; on the flow,
this is nothing but $\psi_{2}(h D_t)$ and this smoothes out the Green
function. Then we are dealing with a semi-classical boundary value problem with small parameter 
$h$. With the usual notations $\tau={h\over i}\partial_{t}$, $\eta={h\over i}\partial_{y}$, 
$\xi={h\over i}\partial_{x}$, the characteristic set of our operator is given by
$$\tau^2 =\xi^2+(1+x)\eta^2$$
The hyperbolic (resp. elliptic) subset of the cotangent bundle of the boundary $x=0$
is $\vert \tau\vert > \vert \eta\vert$, (resp. $\vert \tau\vert < \vert \eta\vert$) and
the gliding subset is  $\vert \tau\vert=\vert \eta\vert$. From 
$\tau^2=(hD_{t})^2=h^2\lambda_k(D_{y})$, one gets at the symbolic level on the micro-support
of any gallery mode associated to $\omega_{k}$ (see Section \ref{wgm}
for a definition of gallery modes)
\begin{equation}\label{gbz}
\eta^{4/3}h^{2/3}\omega_{k}=\xi^2+x\eta^2\,.
\end{equation}
\end{rmq}
\begin{rmq}\label{remspectraly}
We may also localize with $\psi_{1}(h D_{y})$, with $\psi_{1}\in C_{0}^\infty(]0,\infty[)$, which correspond to
a Fourier localization along the tangential (i.e. $y$) direction (notice
such a truncation is easily seen to commute with the equation, hence
the flow). Since we are not interested with waves transverse to the boundary, 
we may and will assume that on the support of 
$\psi_{1}(h\eta)\psi_{2}(h \sqrt{\lambda_k(\eta)})$ one has $k\leq \varepsilon h^{-1}$
with $\epsilon$ small.
This is compatible with \eqref{gbz} since $\omega_{k}\simeq k^{2/3}$ and 
$k\leq \varepsilon h^{-1}$ is equivalent to $\vert \xi \vert \lesssim \varepsilon^{2/3}$.
This fact  will later have its importance when $a\leq h^{1/2}$.
\end{rmq}
\begin{rmq}\label{remtransverse}
  Irrespective of the position of $a$ relative to $h$,
the remaining part of the Green function,  will be essentially
transverse
and see at most one reflexion for $t\in [0,T_{0}]$, with $T_{0}$ small 
(depending on the above choice of $\varepsilon$). Hence, it can be dealt with as in
\cite{bss08} to get the free space decay and we will ignore it in the
upcoming analysis. 
\end{rmq}
\begin{rmq}\label{symax}
Finally, the symmetry of $G$ (or its suitable spectral truncations) with respect to $x$ and $a$ will
be of great importance: it allows us to
restrict the computation of the $L^\infty$ norm to the region $0\leq x\leq a$.
\end{rmq}

Now, we consider initial data $u_0(x,y)=\psi_{2}(h\sqrt{-\triangle_{g}})\psi_{1}(hD_y)\delta_{x=a,y=0}$ where the $\psi_{j}$ are those of Remark \ref{remspectraly}.
 We will use different arguments depending on the respective position
 of $a$ and $h$.

The first case is $a\gg h^{4/7}$: there, we follow ideas of \cite{doi} and write a parametrix for the wave equation as a
superposition of localized waves for which we can compute the wave front set and hence the
singularities that appear at different times and locations. The
construction of \cite{doi} has to be significantly altered to allow
for the range $h^{4/7}\ll a\leq h^{1/2}$, with a phase which is less
explicit but prevents amplifying factors at each reflexion that induced the $a>h^{1/2}$ restriction in \cite{doi}.

The second case corresponds to data for which the distance $a$
to the boundary is such that $0<a\lesssim h^{1/2}$: we write the
contribution of our data which is localized in a $h^{1/4}$ cone of
tangential
directions as the $L^2(\Omega)$ orthogonal sum of whispering gallery
modes and prove that after a time $t$ the corresponding wave remains
frequency localized in the same cone of directions of size $h^{1/4}$, at
least up to smooth remainders. While not quite as strong as a
microlocal propagation of singularities result, this allows for the
use of Sobolev embedding theorem to recover the ``dispersion'' by
using the size of the Fourier support. The contribution of data
corresponding to directions with angles with the boundary greater than $h^{1/4}$ may be dealt with separately,
using a crude parametrix construction, as they involve only cusp-type singularities.

Notice that there is an overlap between the two regions: in fact the
parametrix construction obviously provides better bounds in the
overlap region, both in size (we gain an $a^{1/8}$ factor in the worst
case) and position (the swallowtail occurs exactly once in between two
consecutive reflexions). Had we reproduced the parametrix construction from
\cite{doi}, we would have an epsilon loss in the dispersion estimate
because of the $a\sim h^{1/2}$ region. We thought it was of
independent interest to quantify how ``far'' below $h^{1/2}$ the
construction could be pushed while retaining the most interesting
features of \cite{doi}.
\begin{rmq}
Figure \ref{figpropS} illustrates the propagation of (part of) the wavefront set
of the Dirac data; the second picture is a zoomed version of the first
one and shows in detail the formation of the swallowtail singularity
for the part of the wave front moving along directions which are
initially tangent to the boundary.
\begin{figure}
	\begin{center}
		\includegraphics[width=13cm, height=3cm]{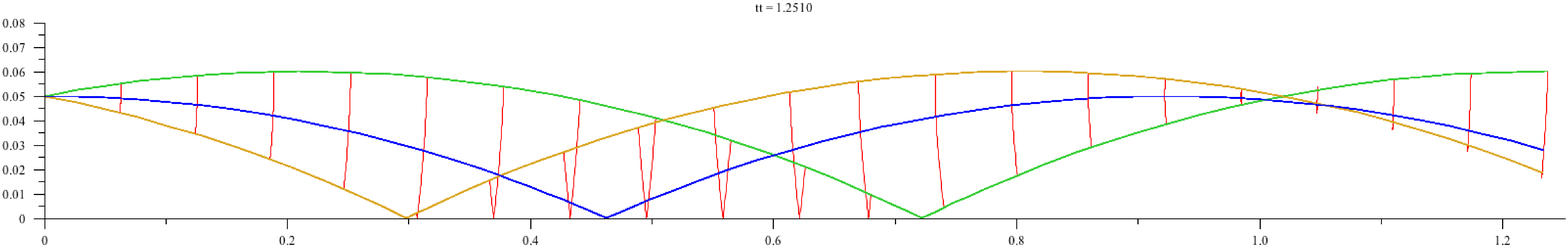}
	\end{center}
	\caption{Propagation of the wavefront}
\label{figpropS}
\end{figure}

\begin{figure}
	\begin{center}
		\includegraphics[width=15cm]{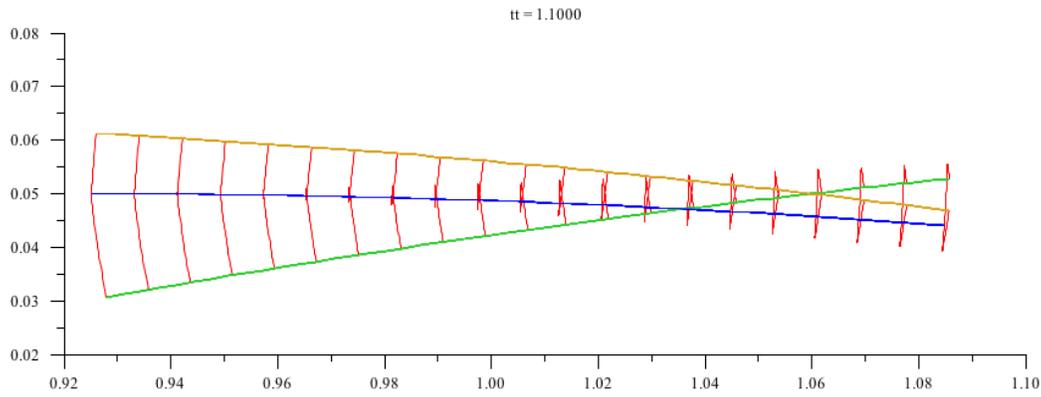}
	\end{center}
	\caption{The formation of a swallowtail singularity just after
          the first reflection (zoomed image)}
\label{figqueueS}
\end{figure}
\end{rmq}
\noindent
Finally, Theorem \ref{disper} is obtained for $a\geq h^{4/7 -\varepsilon}$ in Section
\ref{secttg1}, Theorem \ref{thL1}, and for $a\leq h^{1/2}$ in Section \ref{secgal},
Proposition \ref{propadisp}.
Theorem \ref{disperoptimal} is obtained in section
\ref{secttg1}, as a remark at the end of the proof of Proposition \ref{propL2}.

\section{Parametrix for $a> h^{\frac{4}{7}}$}\label{secttg1}

 \par \noindent
 
 This section is devoted to the construction, modulo $O(h^\infty)$,
  of the Green function in the case
 $a\geq h^{\frac{4}{7}-\varepsilon}$. The Green function is represented 
 in Proposition \ref{p1} as a superposition
 of $O(a^{-1/2})$ reflected waves. We give a precise analysis  of the Lagrangian in the phase space
 associated to each
 reflected wave. This geometric analysis allows us to track the degeneracy of the phases
 when we apply phase stationary arguments. Our main dispersive
 estimate will be  Theorem \ref{thL1}.
   
Let us set $\hbar=h/\eta$ and $\P=(-i
\hbar\partial_x)^2+1+x-(-i\hbar \partial_t)^2$. For $a\geq 0$, we denote by $\Lambda_a\subset T^*\R$
the Lagrangian
\begin{equation}
  \Lambda_a=\{ (t',\tau') \ \exists \theta\in \R \ \text{s.t.}\ \  t'=-2\theta\sqrt{1+a+\theta^2},\,\tau'=\sqrt{1+a+\theta^2}\}\,.
\end{equation}
 The set $\Lambda_a$
may be parametrized by $t'$. Let $\psi_a(t')$ be the unique function 
 such that $\psi_a(0)=0$ and
$\Lambda_a=\{(t',\psi_a'(t')\}$.
Let us set $\rho=1+a$ and $\theta=\sqrt \rho z$, then
$(t')^2=4\rho^2(z^2+z^4)$, from which we get
$$
2z^2+1=\sqrt{1+(t')^2/\rho^2} \imp \psi_a'^2=\rho(1+z^2)=\frac{\rho}{2}
(1+\sqrt{1+(t')^2/\rho^2})
$$
and as $\psi'_a=\tau'>0$,
\[
 \psi_a'=\sqrt\rho(1+t'^2/(8\rho^2)+O(t'^4))\,;
\]
finally, by integration, as $\psi_a(0)=0$,
\begin{equation}
  \label{eq:1}
   \psi_a(t')=\sqrt\rho(t'+\frac{t'^3}{24\rho^2}+O(t'^5))
\end{equation}
\subsection{A singular integral representation for the data}
We start by a suitable decomposition of the smoothed Dirac as an
inverse Fourier transform of a superposition of Airy functions.
\begin{lemma}
\label{l1}
Let $\chi_1\in C^\infty_0((-\theta_0,\theta_0))$ with small
$\theta_0$. There exists a symbol $\sigma_0(t',\hbar)$ of degree $0$ with an
asymptotic expansion in $\hbar$, i.e.
$$
\forall N,\,k \,,\, \exists\, C_N \, \text{ s.t. }\,\,
\sup_{t'}\bigl\vert \partial^k_{t'} \bigl(\sigma_0(t',\hbar)-\sum_{0\leq
  j\leq N} \sigma_{0,j}(t')\hbar^j\bigr) \bigr\vert\leq
C_N\hbar^{N+1}\,,
$$
which is supported in a neighborhood of $t'=0$ and with the following properties: let
\begin{equation}
  \label{eq:u0}
  \tilde u_0(t,x,h,\hbar)=\frac 1 {2\pi \hbar}\int e^{\frac i \hbar
    (\zeta(t-t')+s(x+1-\zeta^2)+\frac{s^3}3)}e^{\frac i \hbar \psi_a(t')} 
  \sigma_0(t',\hbar)\frac{dt'dsd\zeta}{(2\pi h)^2} 
\end{equation}
then $\tilde u_0$ is such that, for $x>-1$,
\begin{enumerate}
\item The wave front set of $\tilde u_0$ is included in $\tau>0$. In
  fact,
$$
\textrm{WF}_h(\tilde u_0)\subset \{\tau\in [\sqrt{1+a},\tau_0]\},
$$
where $\tau_0$ is related only to the size of the support of
$\sigma_0$ in $t'$. Moreover,
$$
\P \tilde u_0=0.
$$
\item The initial data $\tilde u_0(0,x,h,\hbar)$ is a smoothed out
  Dirac, that is
$$
\tilde u_0(0,x,h,\hbar)=\frac 1 {(2\pi h)^2} \int e^{\frac i \hbar
  (x-a)\theta} \chi_1(\theta)\,d\theta+O_{C^\infty}(h^\infty)\,.
$$
\end{enumerate}
\end{lemma}

\begin{proof}
  Consider the time Fourier transform of $\tilde u_0$,
  \begin{align*}
    \hat{\tilde u}_0(\tau/\hbar,x,h,\hbar) & =\int e^{-it\tau/\hbar} \tilde
u_0(t,x,h,\hbar) \,dt \\
 & = \hbar^{\frac 1 3} \int \Ai (\hbar^{-\frac 2 3}(x+1-\tau^2))
 e^{\frac i \hbar (\psi_a(t')-\tau t')} \sigma_0(t',\hbar)
 \frac{dt'}{(2\pi h)^2}\,.
  \end{align*}
Therefore, $\hat{\tilde u}_0$ is an average (with compact support in $t'$) of
solutions to the equation
$$
(\frac \hbar i \partial_x)^2 f+(1+x-\tau^2)f =0\,.
$$
From $\partial_{t'} (\psi_a(t')-\tau t')=0$, we get $\tau=\psi_a'(t')$
and therefore there exists $\theta$ such that
$\tau=\sqrt{1+a+\theta^2}$, which proves the claim on
$\mathrm{WF}_h(\tilde u_0)$.

We proceed with the second part of the statement, regarding the
initial data,
$$
\tilde u_0(0,x,h,\hbar)=\frac{1}{(2\pi h)^2 2\pi \hbar} \int e^{\frac i \hbar
    (s(x+1-\zeta^2)+s^3/3-t'\zeta+\psi_a(t'))} \sigma_0 \,dt'dsd\zeta.
$$
Let $\phi(t',x,s,\zeta)=s(x+1-\zeta^2)+s^3/3-t'\zeta$ and denote by
$\mathcal{C}_\phi$ the set
$$
\mathcal{C}_\phi=\{ (t',x,s,\zeta) \mathrm{\ \ s.t.\ \ } \partial_s
\phi=\partial_\zeta \phi=0 \}.
$$
The equations defining $\mathcal{C}_\phi$ read $x+1+s^2=\zeta^2$ and
$2s\zeta+t'=0$. From the first equation, we get $\zeta\neq 0$ on
$\mathcal{C}_\phi$ (recall $x>-1$). Now,
$$
 \mathrm{Hess}_{s,\zeta} \phi=\begin{pmatrix} 2s & -2\zeta \\ -2\zeta &
  -2s \end{pmatrix},
$$
and $\mathrm{det} ( \mathrm{Hess}_{s,\zeta} \phi)\neq 0$ on
$\mathcal{C}_\phi$. Therefore $\mathcal{C}_\phi$ is a smooth manifold.

Denote by $\pi_1$ the projection from $\mathcal{C}_\phi$ to
$T^*\R_{x}$, that is
$$
\pi_1((t',x,s,\zeta)\in
\mathcal{C}_\phi)=(x,\partial_{x}\phi)=(x,s).
$$
and by $\pi_2$ the projection from $\mathcal{C}_\phi$ to
$T^*\R_{t'}$, that is
$$
\pi_2((t',x,s,\zeta)\in
\mathcal{C}_\phi)=(t',-\partial_{t'}\phi)=(t',\zeta)\,.
$$
For $\tau'\neq 0$, we have
$$
\pi_2^{-1}(t',\tau')=(t',x=-1+\tau'^2-t'^2/(4\zeta^2),s=-t'/2\zeta,\zeta=\tau').
$$
Therefore $\mathcal{C}_\phi$ induces a canonical transformation from
$T^*\R_{t'}\setminus\{\tau'=0\}$ to $T^*\R_x$ defined by
$$
\chi(t',\tau')=(x=-1+\tau'^2-t'^2/4\tau'^2,\xi=-t'/2\tau').
$$
Notice that
$$
\chi(\Lambda_a)=(x=a,\xi=\theta)=T^*_{x=a},
$$
and $\chi$ is a symplectic isomorphism from a neighborhood of
$(t',\tau')=(0,\sqrt{1+a})$ onto a neighborhood of $(x,\xi)=(a,0)$.

The remaining part of the argument is standard: denote by $G(t',x)=\phi(t',x,s_c,\zeta_c)$ where $(s_c,\zeta_c>0)$ is
the unique solution of $1+x+s_c^2=\zeta_c^2$ and $t'+2s_c\zeta_c=0$,
then $G(0,a)=0$, as $s_c(0,a)=0$ and $\zeta_c(0,a)=\sqrt{1+a}$.
By stationary phase in
$(s,\zeta)$ we get 
$$
\tilde u_0(0,x,h,\hbar)=\frac 1 {(2\pi h)^2} \int e^{\frac i \hbar
  (G(t',x)+\psi_a(t'))} A_0(t',x,\hbar)\sigma_0(t',\hbar) \,dt'
$$
where $A_0(t',x,\hbar)$ is an elliptic symbol of order $0$. From $\partial_{t'}
G(t',a)+\psi'_a(t')=0$ and $G(0,a)=0$, we get $G(t',a)=-\psi_a(t')$,
and therefore
$$
G(t',x)+\psi_a(t')=(x-a)H_a(t',x) \mathrm{\ \ with \ \ }\partial_{t'}H_a(0,a)\neq 0,
$$
and by change of variables $\Theta=H_a(t',x)$ and using that for all
$F$ there exists $G$ such that
$$
\int e^{\frac i \hbar (x-a)\Theta } F(\Theta,x,\hbar)\,d\Theta=\int
e^{\frac i \hbar (x-a)\Theta } G(\Theta,\hbar)\,d\Theta+O(\hbar^\infty)
$$
we obtain the desired conclusion, since by the above canonical transformation
the map $\sigma_{0}(t',\hbar) \mapsto G(\Theta,\hbar)$ is elliptic of
degree $0$.

\end{proof}

Set $g_0(t',\hbar)=e^{\frac i \hbar\psi_a(t')}\sigma_0(t',\hbar)$. We
proceed with
\begin{lemma}
\label{l2}
  Let $c>0$, $\eps>0$, then, with $\rho=1+a$,
  \begin{equation}
    \label{eq:star}
    \sup_{\tau\leq \sqrt \rho-c\hbar^{2/3-\eps}} |\hat
      g_0(\tau/\hbar,\hbar)| \in O(\hbar^\infty).
  \end{equation}
\end{lemma}
\begin{proof}
  Notice that $\hat g_0(\tau/\hbar)$ behaves like an Airy function
  from the geometry of $\Lambda_a$ so the estimate on $\hat g_0$ is
  really the classical estimate on $\Ai$. We provide a direct
  argument: for $\tau<\sqrt \rho$,
$$
\hat g_0(\tau/\hbar,\hbar)=\int e^{\frac i \hbar (\psi_a(t)-t\tau)}
\sigma_0(t,\hbar) \,dt,
$$
and we may integrate by parts using $L=(\psi'_a(t)-\tau)^{-1} \frac
\hbar i \partial_t$ (recall $\psi_a'(t)=\sqrt
\rho(1+t^2/8\rho^2+O(t^4))$). Notice that
$$
({}^{t}\!{L})^N(\sigma_0)=\hbar^N \sum_{j=0}^N \frac{\alpha_{j,N}(t)}{(\psi_a'(t)-\tau)^{2N-j}}
$$
where $\alpha_{j,N}(t)=t^{\lfloor N-2j\rfloor_+} \beta_{j,N}(t)$ (by
induction, as $\psi''_a(t)=O(t)$ and $\lfloor N-2j\rfloor_++1\geq
\lfloor N-2j+1\rfloor_+$ as well as $\lfloor \lfloor
N-2(j-1)\rfloor_+-1\rfloor_+ \geq
\lfloor N-2j+1\rfloor_+$). As such, it remains to check that for $t\in [-1,+1]$ and $\alpha\in ]0,1]$
$$
\hbar^N \frac{t^{\lfloor N - 2j\rfloor_+}}{(\alpha+t^2)^{2N-j} } \leq
C_{N,j} \frac{\hbar^N}{\alpha^{3N/2}},
$$
which is trivial if $j\geq N/2$ and follows from setting $t=\sqrt
\alpha s$ if $j<N/2$.
\end{proof}
\begin{rem}
  One may also prove that there exists $\tau_0>0$ (related to the
  support of $\sigma_0$) such that
  \begin{equation}
    \label{eq:starstar}
    \sup_{\tau\geq \tau_0} |\hat
      g_0(\tau/\hbar,\hbar)| \in O(\hbar^\infty).
  \end{equation}
\end{rem}
\subsection{Digression on Airy functions}\label{appairy}
We recall a few well-known facts about Airy functions: let $z>0$, the
$C^\infty$ function $\Ai$ may be defined as
$$
\Ai(-z)=\frac 1 {2\pi} \int_{\R} e^{i(s^3/3-sz)}\,ds\,,
$$
and is easily seen to satisfy the Airy equation $\Ai''(z)-zAi(z)=0$
which we denote by $(A)$.
\begin{rmq}
  Notice that the defining integral is only an oscillatory
  integral; it may be seen as the inverse Fourier transform of a
  tempered distribution and subsequently proved to be
  $C^\infty$. Alternatively, one may proceed as in \cite{Hormander},
  7.6.16: let $\eta>0$, $\xi=s+i\eta$ and define \mbox{$\Ai(z)=(2\pi)^{-1} \int_{\Im{\xi}=\eta}
    e^{i(\xi^3/3+\xi z)}\,d\xi$}, which is absolutely convergent. One then proves the definition to be
      independent of $\eta$ and for $\eta\rightarrow 0$ we recover the
      previous definition.
\end{rmq}
Let $\omega$ be a cubic root of unity: $\omega^3=1$. Obviously,
$z\mapsto \Ai(\omega z)$ is a solution to $(A)$. Any two of these
three solutions yield a basis of solutions to $(A)$, and the linear
relation between them is $\sum_{\omega^3=1} \omega \Ai(\omega z)=0$,
see \cite{Hormander}, 7.6.18. Then, if we set $\omega=e^{2i\pi/3}$,
$\Ai(z)=-\omega Ai(\omega z)-\bar \omega \Ai(\bar \omega z)$, which we
rewrite as
$$
\Ai(-z)=e^{-i\pi/3} Ai(e^{-i\pi/3} z)+e^{i\pi/3} Ai(e^{i\pi/3} z)=A_+(z)+A_-(z)\,.
$$
if we define $A_\pm(z)=\mp\omega Ai(\mp \omega z)$ (our definition differs slightly
from the usual one which does not include the front factor $\mp\omega$). Notice that
$A_-(z)=\bar A_+(\bar z)$. We also have asymptotic expansions (e.g. \cite{Olver}):
$$
A_-(z)=\frac 1
  {2\sqrt\pi z^{\frac 1 4}} e^{i\pi/4} e^{-\frac 2 3 i z^\frac 3 2}
  \exp \Upsilon(z^{3/2})=\frac 1 {z^{\frac 1 4}} e^{i\pi/4} e^{-\frac 2 3 i z^\frac 3 2}\Psi_-(z)
$$
with  $\exp \Upsilon(z^{3/2})\sim (1+\sum_{l\geq 1} c_l z^{-\frac{3l}2})\sim
{2\sqrt \pi}\Psi_-(z)$ as 
$z\rightarrow +\infty$ and the corresponding expansion for $A_+$,
where we define $\Psi_+(z)=\bar \Psi_-(\bar z)$. Moreover, we
have
$$
\frac{A_-(z)}{A_+(z)} = i e^{-\frac 4 3 i z^{3/2}} e^{i B(z^\frac 3
  2)} \text{\ with \ } iB=\Upsilon-\bar \Upsilon.
$$
Notice that for $u\in \R_+$, $B(u)\in \R$ and $B(u)\sim \sum_{j\geq 1}
b_j u^{-j}$ for $u\rightarrow +\infty$.
\subsection{The parametrix construction}
Let $F(\zeta,\hbar)$ be a function with compact support in $\zeta\in
[1+c h^{\frac 2 3-\eps}, \zeta_0]$. Define
$$
u(t,x,\hbar)=\frac 1 {2\pi \hbar} \int e^{\frac i \hbar(\zeta
  t+s(x+1-\zeta^2) +s^3/3)} F(\zeta,\hbar)\,ds d\zeta.
$$
One easily checks that $\P u=0$ and the trace on $x=0$ is
$$
u(t,0,\hbar)=\hbar^{-\frac 2 3}  \int e^{\frac i \hbar t\zeta}
 (A_++A_-)(\hbar^{-2/3}(\zeta^2-1)) F(\zeta,\hbar)\,d\zeta\,.
$$
Define $f$ by $F(\zeta,\hbar)=\int \exp(-it'\zeta/\hbar) f(t',\hbar)
dt'$, then
$$
u(t,0,\hbar)=J_+(f)+J_-(f),
$$
where $J_\pm$ are Fourier integral operators corresponding to canonical transformations $j_\pm$ on
$T^*\R_t\cap \{\tau\geq 1\}$,
$$
j_\pm(t',\tau')=(t=t'\mp2\tau\sqrt{\tau^2-1},\tau=\tau').
$$

We now set up a few notations:
\begin{itemize}
\item let $\chi_0(\eta)\in C^\infty_0((1/2,5/2))$ be a cut-off function such
  that $\chi_0=1$ on $[1,2]$;
\item recall $\chi_1\in C^\infty_0((-\theta_0,\theta_0))$ with small
  $\theta_0$;
\item let $a\in [h^{\frac 2 3-\eps},a_0]$, with $a_0$ small;
\item let $\beta>0$ be such that $\sqrt{1+a}-\sqrt{1+a\beta}\geq c a$,
  for all $a\in [0,a_0]$;
\item let $ \chi_2(u) \in C^\infty$ with $\chi_2(u)=0$ for $u\leq
  \beta/2$ and $\chi_2(u)=1$ for $u\geq \beta$;
\item let $ \chi_3(\zeta) \in C^\infty$ with $\chi_3(\zeta)=1$ for
  $3/4\leq \zeta\leq \zeta_0$ and $\chi_3(\zeta)=0$ for $\zeta\geq
  \zeta_1$ or $\zeta\leq 1/2$ (with $\zeta_{1}>\zeta_{0}$, $\zeta_0-1>0$ and small and
  $\zeta_1-1$ small).
\end{itemize}
Define
\begin{align*}
  v_N(t,x,y,h)  = & \frac 1 {(2\pi h)^2} \int e^{i\frac \eta h y}
    u_N(t,x,h/\eta) \eta \chi_0(\eta) \,d\eta\,,\\
  u_N(t,x,\hbar)= & \frac{(-i)^N}{2\pi \hbar} \int e^{\frac i
    \hbar(t\zeta+s(x+1-\zeta^2)+s^3/3-\frac 4 3 N(\zeta^2-1)^\frac 3 2+\hbar
    NB((\zeta^2-1)^\frac 3 2/\hbar))}\\
  & \hspace{4cm} {}\times \chi_2((\zeta^2-1)/a)
  \chi_3(\zeta) \hat g_0(\zeta/\hbar,\hbar) \,dsd\zeta\,,\\
  v(t,x,y,h)= &\sum_{0\leq N\leq C_0/\sqrt a} v_N(t,x,y,h).
\end{align*}
and let $P=\partial_t^2-(\partial^2_x+(1+x)\partial_y^2)$.
\begin{proposition}
\label{p1}
  There exists $C_0$ such that the following holds true:
  \begin{enumerate}
  \item $v$ is a solution to $Pv=0$ for $x>-1$;
\item its trace on the boundary, $v(t\in [0,1],x=0)$ is
  $O_{C^\infty}(h^\infty)$;
\item at time $t=0$, we have
  \begin{equation*}
    v(0,x,y,h)-(2\pi h)^{-2} \int e^{\frac i h(\eta y+(x-a)\xi)}
    \chi_0(\eta) \chi_1(\xi/\eta) \,d\eta d\xi = O_{C^\infty}(h^\infty).
  \end{equation*}
  \end{enumerate}
\end{proposition}
\begin{rem}
  Here and thereafter, $f(z,h)\in O_{C^\infty}(h^\infty)$ for $z\in {\Gamma}$ if,
    uniformly in $a\in [h^{\frac 2 3-\eps},1]$, 
$$
\forall \alpha,N,\,\,\,\exists C_{\alpha,N} \text{ \ s.t. \ }
\sup_{z\in \Gamma} |\partial^\alpha_z f(z,h)|\leq C_{\alpha,N} h^N.
$$
\end{rem}
\begin{proof}
  Obviously $v$ is defined by a finite sum and each $v_N$ is a
  solution to $Pv_N=0$. We postpone the rest of the proof to
  subsection \ref{proofp1}.
\end{proof}
\begin{rem}
  We may also define $v$ by a sum from $-C_0/\sqrt a$ to $C_0/\sqrt
  a$, and replace $t\in [0,1]$ by $t\in [-1,1]$. The equation enjoys
  time symmetry and therefore the two points of view are equivalent.
\end{rem}
We start by studying $u_N$; from there, we may obtain information on
$v_N$ by integration over $\eta$. This, however, is a non trivial
matter, as $\hbar=h/\eta$ and integration over $\eta$ has an effect on
$\exp(iNB((\zeta^2-1)^\frac 3 2/\hbar))$.

Let
\begin{equation*}
\phi_{a,N,\hbar}(t,x,t',s,\zeta)=  (t-t')\zeta+s(x+1-\zeta^2)+s^3/3-\frac 4 3 N(\zeta^2-1)^\frac 3 2+\hbar
    NB((\zeta^2-1)^\frac 3 2/\hbar) +\psi_a(t'),
\end{equation*}
so that
\begin{equation*}
  u_N(t,x,\hbar)=\frac{(-i)^N}{2\pi \hbar} \int e^{\frac i
    \hbar\phi_{a,N,\hbar}} \chi_2((\zeta^2-1)/a)
  \chi_3(\zeta) \sigma_0(t',\hbar) \, dt'dsd\zeta.
\end{equation*}
Notice that
\begin{itemize}
\item $t'$ takes values in a compact set close to $t'=0$;
\item $\zeta$ takes values in a compact set close to $\zeta=1$;
\item the $s$ integral is oscillatory, and as the symbol is
  independent of $s$, this yields an Airy function (something we will
  use only to check the trace condition in Proposition \ref{p1}).
\end{itemize}

Let us set
$$
\mathcal{C}_{a,N,\hbar}=\{ (t,x,t',s,\zeta) \text{\, s.t.\
} \partial_{t'}\phiN=\partial_{s}\phiN=\partial_\zeta \phiN=0\}
$$
we therefore get a system of three equations defining
$\mathcal{C}_{a,N,\hbar}$,
\begin{align*}
  \zeta & =\psi'_a(t')\\
x & = \zeta^2-1-s^2\\
t & = t'+2s\zeta+4N\zeta(\zeta^2-1)^{1/2}(1-\frac 3 4
B'((\zeta^2-1)^\frac 3 2/\hbar)).
\end{align*}
Notice that on the support of the symbol in the definition of $u_N$,
we have $\zeta\in [\sqrt{1+a\beta/2},\zeta_1]$ with $\zeta_1\sim
1$. We can thus localize further the symbol with $\chi_4(s)\in
C^\infty_0$, $\chi_4=1$ for $s\in [-\zeta_1,\zeta_1]$, as for
$|s|>|\zeta|$, we will have $x=\zeta^2-1-s^2<-1$ and as such the
contribution of $1-\chi_4$ will be $O_{C^\infty}(h^\infty)$ (by
integration by parts in $s$) in the $x\geq -1$ region.
\begin{rem}
  In fact, one may localize closer to $s=0$: if $\chi_4(s)=1$ on
$[-\sqrt{\zeta_1^2-1},\sqrt{\zeta_1^2-1}]$, the same argument provides
a remainder term for $x\geq -\eps_0$. Hence, localizing $s$ close to
$0$ implies $\zeta_1$ close to $1$ and therefore $\theta_0$ smaller
and smaller and the same for $a_0$.
\end{rem}
We may parametrize $\mathcal{C}_{a,N,\hbar}$ by $(s,\theta)$ when they
are close to the origin:
\begin{align*}
  x & = a+\theta^2-s^2\\
t & = 2\sqrt{1+a+\theta^2} (s-\theta+2N\sqrt{a+\theta^2}
(1-\frac 3 4 B'((a+\theta^2)^\frac 3 2 /\hbar)))\\
t' &=-2\theta \sqrt{1+a+\theta^2}\\
s & = s\\
\zeta & =  \sqrt{1+a+\theta^2}\,.
\end{align*}
Notice that $(s,\theta)\rightarrow (s,t'=-2\theta
\sqrt{1+a+\theta^2})$ is a local diffeomorphism in a neighborhood of
$(0,0)$, which ensures that $\mathcal{C}_{a,N,\hbar}$ is a smooth
2D manifold.

Let us denote by $\Lambda_{a,N,\hbar}$ the image of
$\mathcal{C}_{a,N,\hbar}$ by the map
$$
(t,x,t',s,\zeta)\rightarrow (x,t,\xi=\partial_x
\phi_{a,N,\hbar},\tau=\partial_t \phi_{a,N,\hbar})\,;
$$
then $\Lambda_{a,N,\hbar}$ is a Lagrangian submanifold which is
parametrized by $(s,\theta)$:
\begin{align*}
  x & = a+\theta^2-s^2\\
t & = 2\sqrt{1+a+\theta^2} (s-\theta+2N\sqrt{a+\theta^2}
(1-\frac 3 4 B'((a+\theta^2)^\frac 3 2 /\hbar)))\\
\xi & = s\\
\tau & =  \sqrt{1+a+\theta^2}\,.
\end{align*}
\begin{lemma}
\label{l3}
  The Lagrangian submanifold $\Lambda_{a,N,\hbar}$ is smooth and its
  parametrization by $(s,\theta)$ is one to one.
\end{lemma}
\begin{proof}
 One has first to verify that at each point $(s,\theta)$,
  the differential of the map $(s,\theta)\mapsto (x,t,\xi,\tau)$ is injective.
  But this is obvious since $({\partial \xi \over \partial s},{\partial \xi \over \partial \theta})=(1,0)$, 
  ${\partial \tau \over \partial \theta}=\theta \tau^{-1/2}$, and if $\theta =0$,
  then ${\partial t \over \partial \theta} (s,0)=-2\sqrt{1+a}$. The map is clearly one to one as $t(s,\theta)\neq t(s,-\theta)$ for
  $\theta\neq 0$.
\end{proof}
We digress for a while and explain how to add the $y$ variable. In the
definition of $v_N(t,x,y,h)$, we have a phase function
$$
\Psi_\anh(t,x,y,t',s,\zeta,\eta)=\eta y+\eta
\phi_{a,N,h/\eta}(t,x,t',s,\zeta),
$$
from which we get $\partial_y \Psi_\anh=\eta$, and
$$
\partial_\eta
\Psi_\anh=y+\psi_a(t')+\zeta(t-t')+s(x+1-\zeta^2)+s^3/3+N(\zeta^2-1)^{3/2}
(-4/3+B'((\zeta^2-1)^\frac 3 2 \eta/h))
$$
and the full Lagrangian ${\bf \Lambda} _\anh \subset T^*\R^3$ is the set of
points $(x,y,t,\xi,\eta,\tau)$ such that there exist
$(s,\theta,\eta)$ solution to
\begin{align*}
  x = & a+\theta^2-s^2\\
y= & -\psi_a(-2\theta(1+a+\theta^2)^\frac 1 2)-2s(1+a+\theta^2)+\frac
2 3 s^3\\
 & {}-N(a+\theta^2)^\frac 1 2(3+2a+2\theta^2)(\frac 4
 3-B'((\zeta^2-1)^\frac 3 2 \eta/h))\\
 t = & 2 \sqrt{1+a+\theta^2}(s-\theta+2N(a+\theta^2)^\frac 1 2(1-\frac 3 4
 B'((\zeta^2-1)^\frac 3 2 \eta/h))\\
\xi = & \eta s\\
 \eta = & \eta\\
 \tau = & \eta\sqrt{1+a+\theta^2}\,.
\end{align*}
\begin{rem}
  Notice that for $N=0$, having $t=0$ in ${\bf \Lambda}_{a,0,h}$ is
  equivalent to having $s=\theta$. This implies $x=a$ and then $y=0$ is a consequence of
  \begin{equation}
    \label{eq:1'}
    \psi_a(-2\theta(1+a+\theta^2)^\frac 1 2)=-2\theta(1+a+\theta^2)+\frac
2 3 \theta^3\,,
  \end{equation}
and observe that \eqref{eq:1'} holds true since $\psi'_a(t')=(1+a+\theta^2)^\frac 1 2$ for
$t'=-2\theta(1+a+\theta^2)^\frac 1 2$ and $\psi_{a}(0)=0$. Therefore we can explicitly
compute $\psi_a(t')$, as $\theta=-t'(1+a+((1+a)^2+t'^2)^\frac 1
2)^{-\frac 1 2}/\sqrt 2$.
\end{rem}
\subsection{A suitable change of coordinates}\label{changescale}
We now perform a rescaling of our coordinates that provides some
useful reductions.

Set
$$
t=\sqrt a T,\, x=aX,\, y=-t\sqrt{1+a} +a^\frac 3 2 Y,
$$
and define $\gamma_N(T,X,Y,h)$, $w_N(T,X,\hbar)$ as
\begin{align*}
  v_N(t,x,y,h) & =\gamma_N(t/\sqrt a,x/a,y+ta^{-3/2}\sqrt{1+a},h)\\
u_N(t,x,\hbar) & = a^{-\frac 1 2} e^{\frac i \hbar t \sqrt{1+a}}
w_N(t/\sqrt a,x/a,\hbar).
\end{align*}
We define
$$
\tilde \psi_a(T')=a^{-\frac 3 2} (\psi_a(\sqrt a T')-\sqrt a
\sqrt{1+a} T')\,.
$$
Notice that $\tilde \psi_a$ is $C^\infty$ in $(T',a)$, with support in
$\sqrt a |T'|\lesssim 1$ and
\begin{equation}
  \label{eq:2}
  \tilde \psi_a (T')=\frac {T'^3}{24 \rho^\frac 3 2}(1+O(aT'^2))
  \text{\ \ \ (recall\ }\rho=1+a)\,.
\end{equation}
Set
$$
\zeta^2-1=az,\,s=\sqrt a \sigma,\, t'=\sqrt a T'
$$
and
\begin{equation}\label{eq:2bis}
\zeta-\sqrt \rho=a\gamma_a(z)=a\frac{z-1}{\sqrt{1+a}+\sqrt{1+az}}\,
\end{equation}
so that
$$
(t-t')\zeta+\psi_a(t')=t\sqrt \rho+a^\frac 3
2((T-T')\gamma_a(z)+\tilde \psi_a(T'))
$$
and therefore
$$
\phi_\anb=t\sqrt \rho+a^\frac 3 2 \varphi_{a,N,\lambda}
$$
with
\begin{multline}
  \label{eq:3}
  \varphi_\anl(T,X,T',\sigma,z)=\gamma_a(z)(T-T')+\tilde\psi_a(T')+\sigma(X-z)+\sigma^3/3\\
{}+N(-\frac
  4 3 z^\frac 3 2+\frac 1 \lambda B(\lambda z^\frac 3 2)),
\end{multline}
where $\lambda=a^\frac 3 2/\hbar$ will be our large parameter. 
One may remark that $\varphi_\anl$ is $C^\infty$ in $a$ and
\begin{equation}
  \label{eq:7}
  \varphi_{0,N,\lambda}=\frac{z-1}2 (T-T')+\frac {T'^3} {24}
  +\sigma(X-z)+\frac 1 3 \sigma^3+N(-\frac 4 3 z^\frac 3 2+\frac 1 \lambda B(\lambda
  z^\frac 3 2))
\end{equation}
and we have $z\geq \beta /2>0$ on the support of the symbols in our integrals.

We have now
\begin{equation}
  \label{eq:4}
  \gamma_N(T,X,Y,h)= \frac{\sqrt a}{(2\pi h)^2} \int e^{i {a^3/2\over
      h } \eta    Y} w_N(T,X,\hbar) \eta \chi_0(\eta) d\eta\,,
\end{equation}
and
\begin{equation}
w_N(T,X,\hbar)= \frac{(-i)^N\lambda}{2\pi} \int e^{i \lambda \varphi_\anl}
\chi_2(z)\frac{\chi_3(\sqrt{1+az})}{2\sqrt{1+az}} \chi_4(\sqrt a
\sigma) \sigma_0(\sqrt a T',\hbar) \,dT'd\sigma dz
\end{equation}
where we used $dt'dsd\zeta=\frac{a^2}{2\sqrt{1+az}} dT'd\sigma dz$ and
$a^2/(2\pi\hbar)=\sqrt a \lambda/(2\pi)$. Finally, we set
$\theta=\sqrt a \mu$ and $\tilde \lambda=\lambda/\eta=a^\frac 3 2/h$. By our change of variables, the differential operator $P$ becomes
$$
P=a^{-2}Q_a \text{\ with \ }
Q_a=-\partial^2_X-(X-1)\partial^2_Y+2\sqrt{1+a}\partial_T \partial_Y+a\partial^2_T,
$$
and 
$$
Q_a(e^{i\lambda Y}  f(T,X))=e^{i\lambda Y}\lambda^2 \mathcal{Q}_a f
$$
with 
$$
\mathcal{Q}_a = (\frac 1 {i\lambda} \partial_X)^2+(X-1)-2\sqrt \rho
\frac 1 {i\lambda}\partial_T-a(\frac 1 {i\lambda}\partial_T)^2\,.
$$
Our initial data at $T=0$ is now
$$
\frac{\sqrt a}{(2\pi h)^2} \int e^{i {a^{3/2}\over h} (\eta Y+(X-1)\Xi)}
  \chi_0(\eta) \chi_1(\sqrt a \Xi/\eta)\,d\eta d\Xi \,,
$$
and is concentrated at $Y=0,X=1$. The new operator $\mathcal{Q}_a$ has symbol
$$
\sigma(\mathcal{Q}_a)=\Xi^2+X-1-2\sqrt \rho \tau-a\tau^2,
$$
and the positive root in $\tau$ of $\sigma(\mathcal{Q}_a)$ at $X=1$ is
$$
\tau=\frac{\Xi^2}{\sqrt\rho+\sqrt{\rho+a\Xi^2}}\,.
$$
Notice that, as $\tau$ is preserved by the flow, the bouncing angles
at $X=0$ are such that $\Xi^2_{\text{bounce}}=1+2\sqrt \rho
\tau+a\tau^2=1+\Xi^2_0\geq 1$; we are now facing only transverse
reflexions, however we aim at studying the flow for very large times.
\begin{rem}
  Assuming the worst terms are $|\xi|\lesssim \sqrt a$, this
  translates into $|\Xi|\lesssim 1$ which implies $\tau$ bounded, and
  for small $a$, $\mathcal{Q}_a$ degenerates to a Schr\"odinger operator.
\end{rem}

From $\partial_{T'}=\sqrt a \partial_{t'}$, $\partial_s=\sqrt
a \partial_\sigma$ and $\partial_z=\frac a 2 (1+az)^{-\frac 1
  2}\partial_\zeta$, we have
\begin{align*}
\label{eq:8}
  \partial_{T'}\varphi_\anl & =\tilde\psi'_a(T')+\frac{1-z}{\sqrt
    \rho+\sqrt{1+az}}\\
  \partial_{\sigma}\varphi_\anl & =X-z+\sigma^2\\
  \partial_{z}\varphi_\anl &
  =\frac{1}{2\sqrt{1+az}}\bigl(T-T'-2\sigma\sqrt{1+az}-4N\sqrt
  z\sqrt{1+az}(1-\frac 3 4 B'(\lambda z^\frac 3 2))\bigr)\,.
\end{align*}
The Lagrangian of $\tilde\psi_a$ is parametrized by
\begin{equation*}
  \label{eq:9}
  T'=-2\mu\sqrt{1+a+a\mu^2},\,\,\,
  \tilde\psi'_a(T')=\frac{\mu^2}{\sqrt \rho+\sqrt{1+a+a\mu^2}}
\end{equation*}
and as $1+az=\zeta^2$, and $\zeta=(1+a+\theta^2)^{1/2}$ on $\mathcal{C}_{a,N,\hbar}$, we have 
$$1+\mu^2=z \ \ \text{on} \ \ 
\mathcal{C}_{a,N,\hbar}\,.$$
In our new set of coordinates, the projection of ${\bf \Lambda}_{a,N,h}$
onto $\R^3$ is, with $z=1+\mu^2$,
\begin{equation}
\label{eq:11}
\begin{aligned}
  X & = 1+\mu^2-\sigma^2\\
  Y &=2\mu^2(\mu-\sigma) H_1(a,\mu)+\frac 2 3(\sigma^3-\mu^3)+4N(1-\frac
 3 4B'(\lambda z^\frac 3 2)) H_2(a,\mu)\\
 T & =2\sqrt{\rho+a\mu^2}\bigl(\sigma-\mu+2N\sqrt{1+\mu^2}(1-\frac 3 4B'(\lambda z^\frac 3 2))\bigr)
\end{aligned}
\end{equation}
where $H_{1}$ and $H_2$ are defined as
\begin{equation}
  \label{eq:11'}
  H_1(a,\mu)
  =\frac{\sqrt{\rho+a\mu^2}}{\sqrt\rho+\sqrt{\rho+a\mu^2}}\,,\,\,H_2(a,\mu)  = \sqrt{1+\mu^2} \frac{\frac 2 3+\frac
  {5a}{9}+\mu^2(-\frac 1 3+\frac a 9)-\frac 4 9 a \mu^4}{\sqrt
  \rho\sqrt{\rho+a\mu^2}+1+\frac 2 3 a(1+\mu^2)}\,.
\end{equation}
\begin{rem}
  Notice that the parameters are $\mu,\sigma$ and $\eta$ through the
  $\lambda$ factor in the $X,Y,T$ parametrization of ${\bf\Lambda}_\anh$.
\end{rem}
\subsection{Proof of Proposition \ref{p1}}\label{proofp1}
We already dealt with the first item. We now address the remaining
two, which deal respectively with the boundary condition and the
initial data.
\subsubsection{Proof of (2) in Proposition \ref{p1}: the boundary condition.}
Set
$$
F_N(\zeta,\hbar)=(-i)^N e^{\frac i \hbar (-\frac 4 3 N(\zeta^2-1)^\frac
  3 2+\hbar NB((\zeta^2-1)^\frac 3 2/\hbar))}\chi_2(\frac{\zeta^2-1} a)
\chi_3(\zeta)\hat g_0(\zeta/\hbar,\hbar)
$$
and recall
$$
v_{|x=0}=(2\pi h)^{-2} \int e^{i\eta y/h}\eta
\chi_0(\eta)\hbar^{-2/3} e^{it\zeta/\hbar}\sum_{0\leq N\leq C_0/\sqrt
  a} (A_++A_-)(\hbar^{-2/3}(\zeta^2-1)) F_N
\,d\zeta d\eta\,;
$$
recall as well that we constructed $F_N$ so that
$$
F_N=(-1)^N \left(\frac{A_-}{A_+}\right)^N F_0,
$$
 which allows to cancel all middle terms in the sum to get
$$
v_{|x=0}=(2\pi h)^{-2} \int e^{i\eta y/h}\eta
\chi_0(\eta)\hbar^{-2/3} e^{it\zeta/\hbar}(A_+(\cdots) F_0
+A_-(\cdots) F_{N_{\mathrm{max}}})\,d\zeta d\eta\,.
$$
Let us define 
\begin{align*}
  I_0(t,\hbar) & = \int e^{i(t-t')\frac \zeta \hbar}
  A_+(\hbar^{-2/3}(\zeta^2-1)) e^{\frac i \hbar \psi_a(t')} \chi_2
  \chi_3 \sigma_0 \,dt'd\zeta\\
  I_{N_{\text{max}}}(t,\hbar) & = \int e^{i(t-t')\frac \zeta \hbar}
  A_-(\hbar^{-2/3}(\zeta^2-1)) e^{\frac i \hbar (-\frac 4 3
    N(\zeta^2-1)^\frac 3 2 +\hbar N B((\zeta^2-1)^\frac 3 2/\hbar))} e^{\frac i \hbar \psi_a(t')} \chi_2
  \chi_3 \sigma_0 \,dt'd\zeta \,.
\end{align*}
Hence, it is enough to prove that
\begin{itemize}
\item $I_0 \in O_{C^\infty}(h^\infty)$ for $t\geq 0$, uniformly in $a$;
\item $I_{N_{\text{max}}} \in O_{C^\infty}(h^\infty)$ for $t\leq 1$,
  uniformly in $a, N_{\text{max}}$.
\end{itemize}
Start with $I_0$. Using our change of scales,
$I_0(t,\hbar)=J_0(a^{-\frac 1 2 } t,\lambda)$ and
$$
J_0= \hbar ^\frac 1 6 a^\frac 3 2 a^{-\frac 1 4} \int
e^{i\lambda(\gamma_a(z)(T-T')+\tilde \psi_a(T')+\frac 2 3 z^{3/2})} \chi_2(z)\chi_3(\sqrt{1+az}) \frac{m_0(\lambda
z^{3/2})}{z^{1/4}} \sigma_0(\sqrt a T',\hbar) \frac{dT' dz}{2\sqrt{1+az}}
$$
where $m_0$ is a symbol of order $0$. As we have $\partial_{t}=a^{-1/2}\partial_{T}$ 
and $\lambda=a^\frac 3 2
/\hbar \geq \eta h^{-\frac 3 2 \eps}$ and $a\geq h^{\frac 2 3-\eps}$,
we are left to prove the following:
$$
J_0(T,\lambda) \in O_{C^\infty} (\lambda^{-\infty}) \text{\ for \ }
T\geq 0\,, \text{\ uniformly in \ } a.
$$
We already computed the derivatives of the phase of $J_0$. Recall that $T'$ and $\mu$
are related by $T'=-2\mu\sqrt{1+a +a\mu^2}$, and  $a\mu^2$ is bounded.
\begin{align*}
  \partial_{T'}(\text{phase of\ } J_0) & =
  \frac{\mu^2+1-z}{\sqrt{1+az}+\sqrt{1+a+a\mu^2}} \\
  \partial_{z}(\text{phase of\ } J_0) & =
  \frac{T-T'+2\sqrt z \sqrt{1+az}}{2\sqrt{1+az}}
\end{align*}
where for the first derivative one uses $\partial_{T'}=\sqrt
a \partial_{t'}$ and the identity
$$
\gamma_a(z)(T-T')+\tilde \psi_a(T')+\frac 2 3 z^\frac 3 2=a^{-\frac 3
  2} ((t-t')\zeta+\frac 2 3(\zeta^2-1)^\frac 3
2+\psi_a(t')-t\sqrt{1+a})\,.
$$
The first derivative vanishes if $z=1+\mu^2$, and the second one
vanishes if
$$
T=-2\sqrt{1+a+a\mu^2} (\mu+\sqrt{1+\mu^2}) <0.
$$
As such, the phase has no critical points for $T\geq 0$. One has to be
careful as the domain of integration of the $(T',z)$ variables is very
large with small $a$, as it is like $(-c/\sqrt a,c/\sqrt a)\times
(1/2,\xi_1/a)$.

We turn to the details: for $z\leq z_0$, $z$ is bounded. For large
$|T'|$, we get $|\partial_{T'}(\text{phase})|\approx \mu^2\approx
T'^2$, therefore by integration by parts in $T'$ we get decay. If
$|T'|$ is bounded there is no critical points in $(z,T')$.

We are left with $z\geq z_0$, where $z_0$ is large. We first perform the integration in $T'$. There will be two
critical points in $T'$, given by $\mu_\pm=\pm \sqrt{z-1}$. Denote by
$J_{a,\pm}(z)$ the critical values of the phase
$-\gamma_a(z)T'+\tilde\psi_a(T')$. We have
$$
J_{a,\pm}(z)=(-\gamma_a(z)T'+\tilde\psi_a(T'))_{|T'=-2(\pm)\sqrt{z-1}\sqrt{1+az}}
$$
and 
$$
\frac{d}{dz} J_{a,\pm}(z)=\pm2 \sqrt{z-1}\sqrt{1+az} \gamma'_a(z)=\pm
\sqrt{z-1}.
$$
We are left with the $z$ integral,
$$
\int_{z_0}^{+\infty} e^{i\lambda(\gamma_a(z)T+\frac 2 3 z^\frac 3 2
  +J_{a,\pm}(z))} g_a(z,\lambda)\,dz,
$$
where $g_a$ is a symbol of order $-1/4$, uniformly in $a$:
$|\partial^l_z(z^\frac 1 4 g_a(z,\lambda))|\leq C_l z^{-l}$ with $C_l$
independent of $a,z\geq z_0$, and $g_a$ is supported in $[z_0,c/a]$
with small $c$ (notice we used that the $m_0$ and $\chi_3$ terms in
$J_0$ are symbols or order $0$, uniformly in $a$). We have
$$
\partial_z(\gamma_a(z)T+\frac 2 3 z^{\frac 3 2} +J_{a,\pm}(z))=\sqrt
z\pm \sqrt{z-1}+T\gamma'_a(z),
$$
and for the $+$ case we may integrate by parts in $z$ without
difficulties. For the $-$ case, set 
$$
\mathcal{J}=\partial_z(\gamma_a(z)T+\frac 2 3 z^{\frac 3 2}
+J_{a,-}(z))=\sqrt z-\sqrt{z-1} +\frac{T}{2\sqrt{1+az}}.
$$
For $T\geq 0$, $z\in \mathbb{C}$ with $|\Im z|\leq \delta |\Re z|$ and
$z_0\leq \Re z\leq c/a$, we have
$$
|\mathcal{J}|\geq \Re \mathcal{J} \geq \frac C {\sqrt z},
$$
with a constant $C$ which does not depend on $a$ or $T\geq 0$. Hence
by the Cauchy formula, $|\partial_z^l \mathcal{J}^{-1}|\leq C_l z^{\frac 1
  2-l}$ and $\mathcal{J}^{-1}$ is a symbol or order $1/2$, uniformly
in $a,T\geq 0$. We may then conclude by integration by parts in $z$
with the operator $g \mapsto \lambda^{-1} \partial_z(\mathcal{J}^{-1}g)$
(if $g$ is a symbol of order $m$, $\partial_z(\mathcal{J}^{-1} g)$ is
a symbol of order $m-\frac 1 2$).

The remaining integral $I_{N_{\text{max}}}$ may be dealt with in a
similar way. In fact, the situation is easier: on $\Lambda_\anb \cap
\{x=0\}$ we have $s=\pm \sqrt{a+\theta^2}$ and therefore
$$
t=2\sqrt{1+a+\theta^2}(\pm\sqrt{a+\theta^2}-\theta+2N_{\text{max}}\sqrt{a+\theta^2}
(1-\frac 3 4 B')) \geq 3C_0.
$$
\subsubsection{Proof of (3) in Proposition \ref{p1}: the initial data.}
Taking into account Lemmas \ref{l1} and \ref{l2}, we are left to
prove that $V_N(0,x,y,h)\in O_{C^\infty}(h^\infty)$ uniformly in
$1\leq N\leq C_0/\sqrt a$ for $x\geq 0$. Recall $x=\sqrt a X$ and from
\eqref{eq:11} we have
$$
T=0 \Longleftrightarrow \sigma=\mu-2N\sqrt{1+\mu^2}\alpha,\text{ \ with \ }
\alpha=1-\frac 3 4 B'(\lambda z^\frac 3 2)=1+O(\frac 1 {\lambda^2
  z^3})\,,
$$
from which we get
$$
X=1+\mu^2-\sigma^2=1-4N(1+\mu^2)\alpha(N(1+\mu^2)\alpha-\mu)
$$
 and, as $N\geq 1$ and $\mu^2(\alpha^2-1)=\mu^2 O((\lambda z^\frac 3
 2)^2)=O((z-1)/(\lambda^2z^3)) \in O(\lambda^{-2})$,
$$
4N\sqrt{1+\mu^2} (N\sqrt{1+\mu^2} \alpha -\mu)\geq 2(1-O(\frac 1
{\lambda^2}))\,.
$$
For $\lambda\geq \lambda_0$, $\lambda_0$ large, we get, uniformly in
$N$, $X\leq -1/2$ on the projection of $\Lambda_\anh$, which is what
we need, $w_N(0,X,\hbar)\in O_{C^\infty}(h^\infty)$ for $X\geq 0$,
uniformly in $N$. We turn to the details. As before, we will proceed
by integration by parts. We have
$$
w_N(0,X,\hbar)= \frac{\lambda}{2\pi} \int e^{i\lambda \psi} g
\,dT'd\sigma dz,
$$
where $g$ is a symbol in $\sigma,T',z$ and
$$
\psi=-\gamma_a(z)T'+\tilde
\psi_a(T')+\sigma(X-z)+\sigma^3/3+N(-\frac 43 z^\frac 3 2+\frac
{B(\lambda z^\frac 3 2)}{\lambda})\,.
$$
For $z\leq z_0$, we may localize $\sigma$ to a compact region as
$\partial_\sigma \psi=X-z+\sigma^2\geq \sigma^2-z_0$ and large $|T'|$
will not be a problem. Then for $T',\sigma,z$ in a compact set we get
decay from the geometrical observation on the Lagrangian.

For $z\geq z_0$, we may again eliminate $T'$ and obtain two
contributions,
$$
 \int e^{i\lambda \psi_\pm} g_\pm\,d\sigma dz,
$$
where
$$
\psi_\pm=\pm \frac 2 3 (z-1)^\frac 3
2+\sigma(X-z)+\frac{\sigma^3}3+N(-\frac 4 3 z^\frac 32+\frac{ B (\lambda
z^\frac 3 2)}\lambda)\,.
$$
By integration in $\sigma$, the case $z<X-1$ provide decay, while for
$z\geq X+1$ we have again two contributions $\pm 2/3 (z-X)^\frac 3
2$. The associated phases in $z$ are
$$
\pm \frac 2 3 (z-1)^\frac 3
2\pm \frac 2 3 (z-X)^\frac 3 2+N(-\frac 4 3 z^\frac 32+\frac{ B (\lambda
z^\frac 3 2)}\lambda),
$$
for which we readily observe that they are non stationary: not only
derivatives never vanish, but they increase in value with $N$ (for $N=1$ and the two plus signs,  we deal with large values of $z$ as we have done
just above, for the boundary condition).

We are left with the contributions of $X-1\leq z\leq X+1$, for which
again we may reduce to compact $\sigma$ as $\partial_\sigma \psi_\pm=\sigma^2+X-z$,
and we conclude by the geometric observation on the Lagrangian.

\subsection{Decay for the parametrix}
This section is devoted to the proof of the following result.
\begin{theoreme}\label{thL1}
Let $\alpha<4/7$. There exists $C$ such that for all $h\in ]0,h_{0}]$, all $a\in [h^{\alpha},a_{0}]$, all $X\in [0,1]$, all $T\in ]0, a^{-1/2}]$
and all $Y\in \R$, the following holds true
\begin{equation}\label{gl1}
\vert \sum_{0\leq N\leq C_{0}/\sqrt a}\gamma_{N}(T,X,Y,h)\vert \leq C (2\pi h)^{-2}(({h\over a^{1/2}T})^{1/2}
+a^{1/8}h^{1/4})
\end{equation}
\end{theoreme}
\begin{rmq}
Note that, in the given range of parameters $a,h$, the above theorem
immediately implies our main result, Theorem \ref{disper}, after undoing the rescaling from
Section \ref{changescale}.
\end{rmq}
We first observe that $\gamma_{0}(T,X,Y,h)=v_{0}(t,x,y,h)$ where $v_{0}$ is a solution of $Pv_{0}=0$ 
in $x>-1$ with $WF_{h}v_{0}\subset \{\tau >0\}$. By Proposition \ref{p1}
(and its proof), the associated data at time $t=0$ is a smoothed out
Dirac at $x=a, y=0$. Thus  $v_{0}$ satisfies the classical dispersive
estimate for the wave equation in two space dimensions   and since $t=a^{1/2}T$, this implies 
$$
 \vert \gamma_{0}(T,X,Y,h)\vert \leq C(2\pi h)^{-2} ({h\over
   a^{1/2}T})^{1/2}\,.
$$
Thus we may assume in the proof of \eqref{gl1} that the summation is taken over $1\leq N\leq C_{0}a^{-1/2}$. Recall 
\begin{equation}
  \label{gl2}
  \gamma_N(T,X,Y,h)= \frac{\sqrt a}{(2\pi h)^2} \int e^{i {a^{3/2}\over h } \eta    Y} w_N(T,X,\hbar) \eta \chi_0(\eta) d\eta
  \end{equation}
  where the $w_{N}$ are defined by
  \begin{equation}\label{gl3}
w_N(T,X,\hbar)= \frac{(-i)^N\lambda}{2\pi} \int e^{i \lambda \varphi_\anl}
\chi_2(z)\frac{\chi_3(\sqrt{1+az})}{2\sqrt{1+az}} \chi_4(\sqrt a
\sigma) \sigma_0(\sqrt a T',\hbar) \,dT'd\sigma dz\,.
\end{equation}
We  split each $w_{N}$ in two pieces, $w_{N}=w_{N,1}+w_{N,2}$; $w_{N,2}$
is defined  by introducing an extra  cutoff $\chi_{5}(z)\in C_{0}^\infty(]0,z_{0}[)$ in the integral \eqref{gl3},
 with  $z_{0}>1$, close to $1$, and $\chi_{5}(z)=1$ on $[\beta/2,(1+z_{0})/2]$.
 $w_{N,1}$ is then defined by introducing the cutoff $1-\chi_{5}(z)$ in the integral \eqref{gl3}. We denote
 by $\gamma_{N}=\gamma_{N,1}+\gamma_{N,2}$ the corresponding splitting using formula \eqref{gl2}.
 The following propositions obviously imply  Theorem \ref{thL1}.
 \begin{proposition}\label{propL1}
 There exists $C$ such that for all $h\in ]0,h_{0}]$, all $a\in [h^{\alpha},a_{0}]$, all $X\in [0,1]$, all $T\in ]0, a^{-1/2}]$
and all $Y\in \R$, the following holds true
\begin{equation}\label{gl1bis}
\vert \sum_{2\leq N\leq C_{0}/\sqrt a}\gamma_{N,1}(T,X,Y,h)\vert \leq C (2\pi h)^{-2}h^{1/3}\,.
\end{equation}
\end{proposition}
\begin{proposition}\label{propL2}
 There exists $C$ such that for all $h\in ]0,h_{0}]$, all $a\in [h^{\alpha},a_{0}]$, all $X\in [0,1]$, all $T\in ]0, a^{-1/2}]$
and all $Y\in \R$, the following holds true
\begin{equation}\label{gl1ter}
\vert \sum_{2\leq N\leq C_{0}/\sqrt a}\gamma_{N,2}(T,X,Y,h)\vert \leq C (2\pi h)^{-2}a^{1/8}h^{1/4}\,.
\end{equation}
\end{proposition}
\begin{proposition}\label{propL3}
 There exists $C$ such that for all $h\in ]0,h_{0}]$, all $a\in [h^{\alpha},a_{0}]$, all $X\in [0,1]$, all $T\in ]0, a^{-1/2}]$
and all $Y\in \R$, the following holds true
\begin{equation}\label{gl1quad}
\vert \gamma_{1}(T,X,Y,h)\vert \leq C (2\pi h)^{-2}(({h\over a^{1/2}T})^{1/2}+a^{1/8}h^{1/4})\,.
\end{equation}
\end{proposition}

The remaining part of this section is devoted to the proof of these 3 propositions.
 We also include a separate section which contains useful geometric estimates,
 and a section where we recall useful known estimates on phase integrals.
 \begin{rem} The hypothesis $a\in [h^{\alpha},a_{0}]$ with $\alpha<4/7$ in Theorem \ref{thL1} will be used to see
 that only a few $\gamma_{N}$ overlap with each others. We will address estimate \eqref{gl1}
 in the full range $a\in [h^{2/3-\epsilon},a_{0}]$ in a forthcoming paper. 
 \end{rem} 
 
 \subsubsection{Geometric estimates}

 We denote in this section by  $f(a,a\mu^2)$ various analytic functions defined for
 $a$ and $a\mu^2$ small, with $f(a,b)\in \R$ for $(a,b)\in \R^2$.
 Recall that the projection of ${\bf \Lambda}_{a,N,h}$ onto $\R^3$ is given by 
 \begin{equation}\label{gl4}
 \begin{aligned}
X & = 1+\mu^2-\sigma^2\\
Y &=2\mu^2(\mu-\sigma) H_1+\frac 2 3(\sigma^3-\mu^3)+4N(1-\frac
 3 4B'(\lambda z^\frac 3 2)) H_2\\
T & =2\sqrt{\rho+a\mu^2}(\sigma-\mu+2N\sqrt{1+\mu^2}(1-\frac 3 4B'(\lambda z^\frac 3 2))
\end{aligned}
\end{equation}
with $z=1+\mu^2$ and $H_{1}$, $H_2$ of the form (see \eqref{eq:11'})
\begin{equation}\label{gl5}
 \begin{aligned}
H_1 & =f_{0}(a,a\mu^2), \ f_{0}(0,0)=1/2\\
H_2(1+\mu^2)^{-1/2} & = f_{1}(a,a\mu^2)+\mu^2f_{2}(a,a\mu^2), \ f_{1}(0,0)=1/3,  \ f_{2}(0,0)=-1/6\,.
\end{aligned}
\end{equation}
Let us rewrite the system of equation \eqref{gl4} in the following form
\begin{equation}\label{gl6}
 \begin{aligned}
X & = 1+\mu^2-\sigma^2\\
Y &=2\mu^2(\mu-\sigma) H_1+\frac 2 3(\sigma^3-\mu^3)+ 2H_2(1+\mu^2)^{-1/2}
({T\over 2\sqrt{\rho+a\mu^2}}-\sigma+\mu)
\end{aligned}
\end{equation}
and
\begin{equation}\label{gl7}
2N(1-\frac
 3 4B'(\lambda z^\frac 3 2))=(1+\mu^2)^{-1/2}
({T\over 2\sqrt{\rho+a\mu^2}}-\sigma+\mu)\,.
\end{equation}
Then \eqref{gl6} and \eqref{gl7} is obviously equivalent to \eqref{gl4}. For  a given $a$ and a given point $(X,Y,T)\in\R^3$,  \eqref{gl6} is a system of two
equations for the unknown $(\mu,\sigma)$, and we will use the fact that \eqref{gl7} gives
an equation for $N$.
Recall that we are looking at solutions of
\eqref{gl6} in the range 
$$ a\in [h^\alpha,a_{0}], \alpha<4/7,  \ a\vert\mu\vert^2\leq \epsilon_{0}, \ 0< T\leq a^{-1/2},
\ X\in [0,1]$$
with $a_{0},\epsilon_{0}$ small. Let us denote $R=2(1-3Y/T)$.
\begin{lemma}\label{lemgl1}
Let $T\geq T_{0}>0$, $X\in [-2,2]$, $Y\in \R$ . There exists $\mu_{j}(X,Y,T,a)\in \C$, $j=1,2,3,4$
such that
$$ \{\mu \in \C, a\vert\mu\vert^2\leq \epsilon_{0},\ \exists \sigma\in\C, \  (\mu,\sigma)\ \text{is a solution of \eqref{gl6}} \ \}\subset\{\mu_{1}, \mu_{2},\mu_{3},\mu_{4}\}\,.$$
 Moreover, there exists a function $f_{*}(a,a\mu^2)$ with $f_{*}(0,0)=1$, and  constants 
 $C_{0}, C_{1}, C_{2}>0$, $R_{0},M_{0}>0$  such that 
the following holds true: \\
(a) If $\vert R\vert\geq R_{0}$, two of the $\mu_{j}f_{*}(a,a\mu_{j}^2)$ are in the complex disk 
$D(\sqrt R, A)$,
the two others in the complex disk $D(-\sqrt R, A)$ with 
$A= C_{0}(1/T+{a(1+\vert R \vert)\over \sqrt{\vert R\vert}})$. Moreover, one has
$\sqrt{\vert R\vert}\geq 2A$. \\ 
(b) If $\vert R\vert\leq R_{0}$ and $\vert R \vert T\geq M_{0}$, two of the 
$\mu_{j}f_{*}(a,a\mu_{j}^2)$ are in the complex disk $D(\sqrt R, A)$,
the two others in the complex disk $D(-\sqrt R, A)$ with $A={C_{1}\over T\sqrt{\vert R \vert} }$.  Moreover, one has
$\sqrt{\vert R\vert}\geq 2A$.\\
(c) If $\vert R\vert\leq R_{0}$ and $\vert R \vert T\leq M_{0}$, one has $\vert \mu_{j}\vert \leq C_{2}T^{-1/2}$ for all $j$.
  \end{lemma}
  \begin{proof} We first get rid of $\sigma$. The second equation  in \eqref{gl6}
  is of the form $Y=B_{0}+\sigma B_{1}+\sigma^3B_{2}$, thus by the first equation, we get
  \begin{equation*}\label{gl8}
  Y-B_{0}=\sigma(B_{1}+B_{2}(1+\mu^2-X))
  \end{equation*}
  and then the first line of \eqref{gl6} gives an equation for $\mu$
  \begin{equation}\label{gl9}
  (Y-B_{0})^2=(1+\mu^2-X)(B_{1}+B_{2}(1+\mu^2-X))^2\,,
  \end{equation}
where\begin{equation*}\label{gl10}
 \begin{aligned}
B_{0} & = 2\mu^3 H_{1}-2\mu^3/3+ 2H_2(1+\mu^2)^{-1/2}({T\over 2\sqrt{\rho+a\mu^2}}+\mu)\\
B_{1} &=-2\mu^2H_{1}-2H_{2}(1+\mu^2)^{-1/2}\\
B_{2} &=2/3\,.
\end{aligned}
\end{equation*} 
By \eqref{eq:11'} and \eqref{gl5} one gets through explicit computation the identity
\begin{equation*}\label{gl11}
\forall w, \ \ f_{0}(0,w)+f_{2}(0,w)=1/3  
  \end{equation*}
and this implies, (we use $a\mu^{2+k}=(a\mu^2)\mu^k$)
\begin{equation*}\label{gl12}
 \begin{aligned}
B_{0} & = -\mu^2T/6f_{3}+2/3\mu f_{4}+T/3f_{5}\\
B_{1}+B_{2}(1+\mu^2-X) &=-2X/3+f_{6}
\end{aligned}
\end{equation*}
with $f_{l}(0,0)=1$ for $l=3,4,5$ and $f_{6}(0,0)=0$. 
Let $D=-2X/3+f_{6}$. We may then rewrite \eqref{gl9} as
\begin{equation}\label{gl13}
P_{+}P_{-}=36(1-X){D^2\over T^2}  
  \end{equation} 
 with
\begin{equation*}\label{gl14}
P_{\pm}=(f_{*}\mu)^2- {2(f_{*}\mu)\over T}(2f_{4}\mp3D)/f_{*}-R', \ R'=R+f_{7}, \ f_{7}(0,0)=0  
  \end{equation*}
and $f_{*}=\sqrt f_{3}, f_{*}(0,0)=1$.

 By classical arguments on perturbations of polynomial equations,
 \eqref{gl13} implies that for $a$ and $a\vert\mu\vert^2$ small,
 the $\mu$ equation \eqref{gl9} admits at most $4$ complex solutions 
 (at most since we have the constraint 
 $a\vert\mu\vert^2$ small). 
 Set  $\mu^*=\mu f_{*}(a,a\mu^2)$. 
 Then $\mu\mapsto \mu^*$ is a holomorphic change of coordinates and
 $\vert{\partial \mu^*\over\partial \mu}-1\vert \leq Cte(a_{0}+\epsilon_{0})$.
With the notation $g_{\pm}=(2f_{4}\mp3D)/f_{*}$, the two roots $\mu^*_{\pm,\epsilon}$ of $P_{\epsilon}$
 satisfy the equation (recall that $g_{\epsilon}$ and $R'$ are functions of $(a,a\mu^2)$ ) 
 \begin{equation}\label{gl15}
  \mu^*_{\pm,\epsilon}=g_{\epsilon}/T \pm\sqrt{R'+(g_{\epsilon}/T)^2}, \ \ 
 \epsilon=\pm\,.
\end{equation}
Assume first $\vert R\vert\geq R_{0}$ with $R_{0}$ large. Then by \eqref{gl15}
and $f_{7}(0,0)=0$, there exists $C_{0}$ such that 
\begin{equation}\label{gl16}
\mu^*_{+,\epsilon}\in D(\sqrt R,A_{0}/2), \ \ \mu^*_{-,\epsilon}\in D(-\sqrt R,A_{0}/2), \ \ \ A_{0}=C_{0}(1/T+{a(1+\vert R\vert)\over \sqrt{\vert R\vert}})\,.
\end{equation}
 For $\sqrt{\vert R\vert} \geq 2A_{0}$, 
the two disks $D_{\pm}=D(\pm\sqrt R,A_{0})$ do not overlap and $dist(D_{+},D_{-})\geq \sqrt{\vert R\vert}$. Thus there exists $C_{3}$ such that  
$$
\vert P_{+}P_{-}(\mu)\vert \geq C_{3}\vert R\vert A_{0}^2\geq 
C_{3}C_{0}^2\vert R\vert/T^2, \ \ \forall \mu^*\in\C\setminus
(D_{+}\cup D_{-})\,,
$$
and this contradicts \eqref{gl13} for $\vert R\vert\geq R_{0}$ large enough, and proves (a).

For $\vert R\vert\leq R_{0}$, and $\vert R\vert T \geq M_{0}$, with $M_{0}$
large, \eqref{gl16} remains true and thus we get
\begin{equation*}\label{gl17}
\mu^*_{+,\epsilon}\in D(\sqrt R,A_{1}/2), \ \ \mu^*_{-,\epsilon}\in D(-\sqrt R,A_{1}/2), \ \ \ A_{1}=C_{0}(1/T+{a(1+R_{0})\over \sqrt{\vert R\vert}})\,.
\end{equation*}
Set $A_{2}={C\over T\sqrt{\vert R\vert}}$. Since $\vert R\vert\leq R_{0}$
and $T\leq a^{-1/2}$, for $C$ large enough 
one has $A_{2}\geq 2A_{1}$. The two disks
 $D_{\pm}=D(\pm\sqrt R, A_{2})$ do not overlap
and $dist(D_{+},D_{-})\geq \sqrt{\vert R\vert}$ for $M_{0}\geq 2C$.  Thus one has 
$$\vert P_{+}P_{-}(\mu)\vert \geq C_{4}A_{2}^2\vert R\vert= C_{4}C^2/T^2, \ \ \forall \mu^*\in\C\setminus (D_{+}\cup D_{-})\,,$$
and this contradicts \eqref{gl13} for  $C$
large enough, and proves (b) for $M_{0}$ large enough. Finally,    
for $\vert R\vert\leq R_{0}$, and $\vert R\vert T \leq M_{0}$, one has clearly
by \eqref{gl15} and $T\leq a^{-1/2}$, $\vert \mu_{\pm,\pm}\vert\leq C_{5}T^{-1/2}$, and thus for $c\geq 2C_{5}$
$$\forall \mu\in\C
 \ \text{such that} \  \vert\mu\vert \geq cT^{-1/2}, \  \vert P_{+}P_{-}(\mu)\vert \geq C_{6}c^4/T^2, \ \ $$
This contradicts \eqref{gl13} for  $c$
large enough. The proof of Lemma \ref{lemgl1} is complete. 
  \end{proof}
Let us now study the equation \eqref{gl7}, which provides $N$ . Since $B'(u)\in O(u^{-2})$,
 $z\geq \beta/2>0$, $N\leq C_{0}a^{-1/2}$, and $a\geq h^{4/7}$, one has
\begin{equation*}\label{gl18}
\vert NB'(\lambda z^{{3\over 2}})\vert \in O(N\lambda^{-2})=O(Nh^2/a^3)\in O(a^{-7/2}h^2)\in O(1)\,.
\end{equation*}
Let $\langle\mu\rangle=(1+\mu^2)^{1/2}$. From $\sigma^2-\mu^2=1-X$, we get for $X\in [-2,2]$ 
  $\vert\sigma-\mu\vert/{<\mu>}\in O(1)$. Therefore  \eqref{gl7} implies
\begin{equation}\label{gl19}
2N= T\Phi_{a}(\mu)+O(1), \ \ \Phi_{a}(\mu)=
{1\over 2 <\mu>\sqrt{\rho+a\mu^2}}\,.
\end{equation}
Let  $U=\{\mu\in \C, \vert\mu\vert \leq 0.5 \ \text{or} \
\vert Im(\mu)\vert\leq \vert Re(\mu)\vert /\sqrt 3\}$. Then $\Phi_{a}(\mu)$ 
is bounded on $U$  and 
 \begin{equation}\label{gl20}
\vert \Phi_{a}(\mu)-\Phi_{a}(\mu')\vert\leq {C\vert \mu^2-\mu'^2\vert 
\over \sup (<\vert\mu\vert>, <\vert\mu'\vert>)}
(a+{1\over <\vert\mu\vert><\vert\mu'\vert>}), \ \forall \mu,\mu' \in U\,.
\end{equation}
Observe that for $b\in \R$, and $\vert b\vert \geq 2r$, the complex disk
$D(b,r)$ is contain in $U$.

For a given point $(X,Y,T)\in [-2,2]\times \R\times [0,C_{0}a^{-1/2}]$, let us denote  by $\mathcal N(X,Y,T)$ the set of integers $N\geq 1$
 such that \eqref{gl4} admits at least one real solution $(\mu,\sigma,\lambda)$ with $a\vert\mu\vert^2\leq \epsilon_{0}$ and $\lambda\geq \lambda_{0}$ . We denote $\mathcal N^\C(X,Y,T)$ the set of 
 complex $N$
 such that \eqref{gl4} admits at least one complex solution $(\mu,\sigma)$ with
 $\mu\in U$ and  $a\vert\mu\vert^2\leq \epsilon_{0}$ and $\lambda\geq \lambda_{0}$.  Observe that $\mathcal N(X,Y,T)$ depends on $a$
 . For $E\subset \N$, $\vert E\vert$
 will denote the cardinal of $E$. Observe that \eqref{gl19} implies for an
 absolute constant $N_{0}$
 \begin{equation}\label{gl21}
\mathcal N(X,Y,T)\subset [1,T/2+N_{0}]\,.
\end{equation}
\begin{lemma}\label{lemgl2}
There exists a  constant $C_{0}$ such that the following holds true.
\begin{itemize}
\item[(a)] For all $(X,Y,T)\in [0,1]\times\R\times [0,a^{-1/2}]$, one has $\vert \mathcal N(X,Y,T)\vert \leq C_{0} $,
and $\mathcal N^\C(X,Y,T)$ is a subset of the union of $4$ disks of radius $C_{0}$.
\item[(b)]  For all  $(X,Y,T)\in [0,1]\times\R\times [0,a^{-1/2}] $, the subset of $\mathbb N$,  
$$\mathcal N_{1}(X,Y,T)=\bigcup_{\vert Y'-Y\vert+\vert T'-T\vert\leq 1, \vert X'-X\vert \leq 1}\mathcal N(X',Y',T')$$ satisfies 
$$ \vert \mathcal N_{1}(X,Y,T)\vert \leq C_{0}\,.$$
\end{itemize}
\end{lemma}  
\begin{proof}
We start with (a), which is a consequence of \eqref{gl19}, since by Lemma \ref{lemgl1}, for a given $(X,Y,T\geq T_{0})$, there are at most $4$
possible values of $\mu$ (for $T\leq T_{0}$ we  use \eqref{gl21}).

We proceed with (b). By \eqref{gl21}, we may assume $T\geq T_{1}$ with $T_{1}$ large.
Recall $R=2(1-3Y/T)$. Let $(X',Y',T')$ such that
$\vert Y'-Y\vert+\vert T'-T\vert\leq 1, X'\in [0,1]$. Set
$R'=2(1-3Y'/T')$. One has $\vert R-R'\vert \leq C(1+\vert R\vert)/T$.
Let us first assume $\vert R \vert \geq 2R_{0}$, with $R_{0}$ as in Lemma \ref{lemgl1}. Since 
$T$ is large, one has $\vert R' \vert \geq R_{0}$, and $\vert R' \vert\simeq
\vert R \vert$. Let $N'\in N(X',Y',T')$ and $\mu'$ such that \eqref{gl4} holds true.
By Lemma \ref{lemgl1} (a), one may assume $\mu'^*\in D(\sqrt R',A')$. Take
$\mu^*\in D(\sqrt R,A)$ associated to $(X,Y,T)$. Since $\mu'$ is real, one has $R'\geq R_{0}$, hence $R\geq 2R_{0}$,
and therefore $\mu \in U$. Let $N\in \mathcal N^\C(X,Y,T)$ associated to $\mu$.
 From $a^{1/2}\leq 1/T$ one gets
 \begin{equation*}\label{gl22}
\vert \mu-\mu'\vert \leq C \vert \mu^*-\mu'^*\vert\leq C(A+A'+\vert \sqrt R-\sqrt R'\vert)
\leq {C(1+\vert R\vert)\over T\sqrt{\vert R\vert}}\,.
\end{equation*} 
By \eqref{gl19} and \eqref{gl20}, this implies since 
$a\vert R\vert\simeq a\vert\mu'^2\vert\leq \epsilon_{0}$,
 \begin{multline}\label{gl23}
2\vert N-N'\vert \leq \vert T'-T\vert \Phi_{a}(\mu')+T\vert \Phi_{a}(\mu')-\Phi_{a}(\mu)\vert +O(1)\\
 \leq C(a+1/\vert R\vert){(1+\vert R\vert)\over \sqrt{\vert R\vert}}
+O(1)\in O(1)\,.
\end{multline}
Let us now assume  $\vert R \vert \leq 2R_{0}$ and $T\vert R \vert \geq M_{0}+8 $. 
From $ \vert RT-R'T'\vert = \vert 2(T-T')-6(Y-Y')\vert\leq 8$, we get  $\vert R'\vert T'\geq M_{0}$. We may thus apply Lemma \ref{lemgl1} (b). Let $N'\in N(X',Y',T')$ and $\mu'\in \R$ such that \eqref{gl4} holds true. Since $\mu'$ is real
one has $R'>0$, thus $R'T'>M_{0}$, and this implies $R>0$ (take $M_{0}$ large). Moreover one has $\vert R-R'\vert \leq C(1+\vert R \vert )/T$,
$\vert R' \vert \leq 3R_{0}$, and also $\vert R' \vert\simeq
\vert R \vert$. By the same argument as above, we get now $\vert \mu-\mu'\vert \leq {C(1+R_{0})\over T\sqrt R}$, and since $\vert \mu \vert+ \vert \mu' \vert \leq C\sqrt{\vert R\vert}$, one gets $\vert \mu^2-\mu'^2\vert \leq C(1+R_{0})/T$
and thus by \eqref{gl19} and \eqref{gl20}
 \begin{equation}\label{gl24}
2\vert N-N'\vert \leq CT(a+O(1))(1+R_{0})/T \in O(1)\,.
\end{equation}
Finally, for  $\vert R \vert \leq 2R_{0}$ and $T\vert R \vert \leq M_{0}+8 $,
one has $T'\vert R' \vert \leq M_{0}+16$, thus by part (c) of Lemma \ref{lemgl1},
one has $\vert \mu'_{j} \vert \leq CT^{1/2}$, $\vert \mu_{j} \vert \leq CT^{-1/2}$,
and thus we get in that case $\vert \mu^2-\mu'^2\vert \leq C/T$ thus \eqref{gl24}
holds also true in that case.
Since $N\in \mathcal N^\C(X,Y,T)$, \eqref{gl23}, \eqref{gl24},
and  part (a) of our lemma imply (b). The proof of our lemma
is complete.
\end{proof}
\subsubsection{Phase integrals}
We first recall the following lemma, for which we refer to \cite{stein3}.
\begin{lemma}\label{lemstat1}
Let $K\subset \R$ a compact set, and $a(\xi,\lambda)$  a classical symbol of 
degree $0$ in $\lambda\geq 1$ with $a(\xi,\lambda)=0$ for $\xi \notin K$.
Let $k\geq 2$, $c_{0}>0$ and $\Phi(\xi)$ a phase function such that
$$ \sum_{2\leq j\leq k}\vert\Phi^{(j)}(\xi)\vert\geq c_{0}, \quad \forall \xi\in K.$$ 
Then, there exists $C$ such that
$$ \vert \int e^{i\lambda\Phi(\xi)}a(\xi,\lambda)\vert \leq C\lambda^{-1/k}, \quad 
\forall \lambda\geq 1$$
Moreover, the constant $C$ depends only on $c_{0}$ and on a upper bound
of a finite number of derivatives of $\Phi^{(2)},a$ in a neighborhood of $K$. 
\end{lemma}
The next lemma will be of importance to us. As such, we have included its
proof for the sake of the reader while not claiming any novelty.

Let  $H(\xi)$ be a smooth function defined in a neighborhood of $(0,0)$ in $\R^2$, such that $H(0)=0$ and $\nabla H(0)=0$.
We assume that the Hessian $H''$ satisfies $\text{rank} (H''(0))= 1$ and $\nabla \det (H'')(0)\not= 0$.
Then the equation $\det (H'')(\xi)=0$ defines a smooth curve $\mathcal C$ near $0\in \R^2$ with $0\in \mathcal C$.
Let $s\rightarrow \xi(s)$ be a smooth parametrization of $\mathcal C$, with $\xi(0)=0$, and define the curve $X(s)$ in $\R^2$ by 
$$
X(s) =H'(\xi(s))\,.
$$ 
\begin{lemma}\label{lemstat2}
Let $K=\{\xi \in \R^2, \ \vert\xi\vert \leq r\}$ , and $a(\xi,\lambda)$  a classical symbol of 
degree $0$ in $\lambda\geq 1$ with $a(\xi,\lambda)=0$ for $\xi \notin K$.
Set for $x\in \R^2$ close to $0$
\begin{equation}\label{gls1}
I(x,\lambda)= \int e^{i\lambda (x.\xi-H(\xi))}a(\xi,\lambda)d\xi\,.
\end{equation}
Then for $r>0$ small enough, the following holds true:
\begin{itemize}
\item[(a)] If $X'(0)\not= 0$, there exists $C$ such that for all $x$ close to $0$
$$
\vert I(x,\lambda) \vert \leq C\lambda^{-5/6}\,.
$$
\item[(b)] If $X'(0)= 0$ and $X''(0)\not= 0$
there exists $C$ such that for all $x$ close to $0$
$$
\vert I(x,\lambda) \vert \leq C\lambda^{-3/4}\,.
$$
Moreover, if $a$ is elliptic at $\xi=0$, there exists $C'$ such that
$$
\vert I(0,\lambda) \vert \geq C'\lambda^{-3/4}\,.
$$
\end{itemize}
\end{lemma}
\begin{proof}
By a linear change of coordinates in $\xi$, we may assume $H(\xi)=\xi_{1}^2/2+ O(\xi^3)$.
Set $\Phi(x,\xi)=x.\xi-H(\xi)$. Then $\Phi'_{\xi_1}(x,\xi)=x-H'_{\xi_1}(\xi)$. Therefore, there exists
a unique non degenerate critical point $\xi_{1}^c(x,\xi_{2})$ in the variable $\xi_{1}$,
and the critical value $\Psi(x,\xi_{2})$ satisfies 
$$ 
G(x,\xi_{2})=\Psi'_{\xi_{2}}(x,\xi_{2})=
x_{2}-H'_{\xi_{2}}(\xi_{1}^c(x,\xi_{2}),\xi_{2})\,,
$$
and by stationary phase in $\xi_1$, one has
$$
 I(x,\lambda)=\lambda^{-1/2}\int e^{i\lambda
   \Psi(x,\xi_{2})}b(x,\xi_{2},\lambda)d\xi_{2}\,.
$$
By Lemma \ref{lemstat1}, it remains to prove:
\begin{itemize}
\item[(a')] If $X'(0)\not= 0$, there exists $c_{0}>0$ such that for all $(x,\xi_{2})$ close to $(0,0)$
$\vert \partial^2_{\xi_{2}}G\vert \geq c_{0}$.
\item[(b')]  If $X'(0)= 0$ and $X''(0)\not= 0$, there exists $c_{0}>0$ such that for all $(x,\xi_{2})$ close to $(0,0)$
$\vert \partial^3_{\xi_{2}}G\vert \geq c_{0}$, and in the case $a$ elliptic the lower bound at $x=0$.
\end{itemize}
Let us prove (a'). Since $G(x,\xi_{2})$ is smooth and $r$ small , it is sufficient to prove 
$\partial^2_{\xi_{2}}G(0,0)\not=0$. The Taylor expansion of $H$ at order $3$ reads as follows 
$$
H(\xi)=\xi_{1}^2/2+ a\xi_{1}^3 +
b\xi_{1}^2\xi_{2}+c\xi_{1}\xi_{2}^2+d\xi_{2}^3+O(\xi^4)\,.
$$
Thus one has $\xi_{1}^c(0,\xi_{2})=-c\xi_{2}^2 +O(\xi_{2}^3)$ and we get
$-G(0,\xi_{2})=3d\xi_{2}^2+O(\xi_{2}^3)$. Thus $\partial^2_{\xi_{2}}G(0,0)\not=0$ is equivalent
to $d\not=0$.\\
 On the other hand, one has $\det H''(\xi)=2c\xi_{1} + 6d\xi_{2}+ O(\xi^2)$,
and since by hypothesis $\nabla \det (H'')(0)\not= 0$, one has $(c,d)\not= (0,0)$.
Moreover, one has 
$$
X(s)=H'(\xi(s))=(\xi_{1}(s)+ O(s^2), O(s^2))
$$
 and therefore 
$X'(0)\not= 0$ is equivalent to $\xi'_{1}(0)\not=0$. This in turn is
equivalent to the fact that $\xi_{1}$
is a parameter on $\mathcal C$, which is equivalent to $d\not=0$.

Let us now prove (b'). Since $X'(0)= 0$, we get $d=0$ and therefore
$c\not= 0$. Now, $\xi_{2}$ is a parameter on $\mathcal C$, and we have $\xi_{1}\in O(\xi_{2}^2)$
on $\mathcal C$.
We will  use a Taylor expansion of $H$ at order $4$, but since $\xi_{1}^c(0,\xi_{2})$
is quadratic in $\xi_{2}$, and $\xi_{1}(s)$ quadratic in $s$, we will just need the $\xi_{2}^4$ term, i.e
$$
H(\xi)=\xi_{1}^2/2+ a\xi_{1}^3 +
b\xi_{1}^2\xi_{2}+c\xi_{1}\xi_{2}^2+e\xi_{2}^4+...+O(\xi^5)\,.
$$ 

Then we get $\det H''(\xi)=2c\xi_{1} + 4(3e-c^2)\xi_{2}^2+ O(\xi^3)$.
Therefore $\xi_{1}=2(c-3e/c)\xi_{2}^2+ O(\xi_{2}^3)$ is an equation for $\mathcal C$,
and we get that $X''(0)\not=0$ is equivalent to
$X_{1}''(0)\not=0$. this in turn is is equivalent to $c^2\not= 2e$.
On the other hand, we easily get $-G(0,\xi_{2})=
(4e-2c^2)\xi_{2}^3+O(\xi_{2}^4)$. Finally, for $\alpha\not=0$ and $b(\xi_{2},\lambda)$ a symbol of degree $0$ elliptic at $\xi_{2}=0$, and supported in $\vert \xi_{2}\vert \leq r$ with $r$ small enough, one has clearly
$$\vert \int e^{i\lambda (\alpha\xi_{2}^4+O(\xi_{2}^5))}b(\xi_{2},\lambda)d\xi_{2}\vert
\geq C'\lambda^{-1/4}$$ which completes the proof. 
\end{proof}

 \subsubsection{Proof of Proposition \ref{propL1}}
 
 Recall
   \begin{equation}\label{gl25}
w_{N,1}(T,X,\hbar)= \frac{(-i)^N\lambda}{2\pi} \int e^{i \lambda \varphi_\anl}
\chi_2(z)\frac{\chi_3(\sqrt{1+az})}{2\sqrt{1+az}} \chi_4(\sqrt a
\sigma) \sigma_0(\sqrt a T',\hbar)(1-\chi_{5})(z) \,dT'd\sigma dz
\end{equation}
where the phase $ \varphi_\anl$ is defined by (see \eqref{eq:3})
\begin{multline*}
  \varphi_\anl(T,X,T',\sigma,z)=\gamma_a(z)(T-T')+\tilde\psi_a(T')+\sigma(X-z)+\sigma^3/3\\
{}+N(-\frac
  4 3 z^\frac 3 2+\frac 1 \lambda B(\lambda z^\frac 3 2))\,.
\end{multline*}
For $\epsilon_{j}=\pm$, define
\begin{multline*}
  \label{gl27}
  \Phi_{N,\epsilon_{1},\epsilon_{2}}(T,X,z;a,\lambda)=\gamma_a(z)T+
  \frac 2 3 \epsilon_{1}(z-1)^{3/2}+ \frac 2 3 \epsilon_{2}(z-X)^{3/2}\\
{}-N(\frac
  4 3 z^\frac 3 2-\frac 1 \lambda B(\lambda z^\frac 3 2))\,.
\end{multline*}
\begin{lemma}\label{lemgl3}
The following identity holds true
   \begin{equation}\label{gl28}
w_{N,1}(T,X,\hbar)= \sum_{\epsilon_{1},\epsilon_{2}}\int e^{i\lambda\Phi_{N,\epsilon_{1},\epsilon_{2}}}\Theta_{\epsilon_{1},\epsilon_{2}}(z;a,\lambda)dz + R_{N,a}(T,X,\hbar)
\end{equation}
where $\Theta_{\epsilon_{j}}(z;a,\lambda)$ are smooth functions of $z$
with support in $[(1+z_{0})/2, (\zeta_{1}^2-1)/a]$ (remark that
$\zeta_{1}>1$ is a upper bound for the support of
$\chi_{3}$). Moreover,
$$
\vert z^l\partial_{z}^l\Theta_{\epsilon_j} \vert\leq C_{l}z^{-1/2}
\,\,\text{with}\,\,\, C_{l}\,\text{ independent of }\,a,\lambda\,.
$$
The remainder $ R_{N,a}(T,X,\hbar)$ is $O_{C^\infty}(\hbar^\infty)$ 
for $X\in [0,1], T\in [0,a^{-1/2}]$, uniformly
in $a,N$.  
\end{lemma}
\begin{proof}

The proof is a simple application of stationary phase in $(T',\sigma)$ in the integral
\eqref{gl25}. Recall $z\geq (1+z_{0})/2>1\geq X$ on the support of $(1-\chi_{5})(z)$, 
and $ z \leq (\zeta_{1}^2-1)/a$ on the support of $\chi_{3}(\sqrt{1+az})$. 
The $\sigma$ integral is equal to 
\begin{align*}\label{gl29}
J_{1} & = \int e^{i\lambda(\sigma^3/3-\sigma (z-X))}\chi_{4}(\sqrt a
\sigma )d\sigma\\
 & =
(z-X)^{1/2}\int e^{i\lambda (z-X)^{3/2}(s^3/3-s)}\chi_{4}(\sqrt a(z-X)^{1/2}s )ds\nonumber\,.
\end{align*}
One has $\sqrt a(z-X)^{1/2}\leq \sqrt{\zeta_{1}^2-1}$. Thus, by
stationary phase near the two critical points $s=\pm 1$ and
integration by part in $s$ elsewhere, we get
\begin{equation}\label{gl30}
J_{1}=\lambda^{-1/2}(z-X)^{-1/4}(e^{2/3 i\lambda (z-X)^{3/2}}b_{+}+e^{-2/3 i\lambda (z-X)^{3/2}}b_{-}) +O(\lambda^{-\infty}(z-X)^{-\infty})
\end{equation}
where $b_{\pm}(\sqrt a(z-X)^{1/2}, \lambda (z-X)^{3/2})$ are symbols of degree $0$
in the (large) parameter $\lambda (z-X)^{3/2}$.
Next, the $T'$ integral is equal to 
\begin{equation*}\label{gl31}
J_{2}=\int e^{i\lambda(-\gamma_{a}(z)T'+\tilde\psi_{a}(T'))}\sigma_{0}(\sqrt a T',\hbar)dT'\,.
\end{equation*}
Recall   that
$$
T'=-2\mu\sqrt{1+a+a\mu^2}, \
\  \partial_{T'}(-\gamma_{a}(z)T'+\tilde\psi_{a}(T'))= {\mu^2+1-z
  \over \sqrt{1+az}+\sqrt{1+a+a\mu^2}}\,.
$$ 
Thus we get two distinct critical points $T'_{\pm}=\mp 2\sqrt{z-1}\sqrt{1+az}$. The associated
critical values are   $-\gamma_{a}(z)T'_{\pm}+\tilde\psi_{a}(T'_{\pm})=\pm 2/3 (z-1)^{3/2}$.
As before for the $\sigma$ integral, we perform the change of variable $T'=s\sqrt{z-1}$,
in order to have the two critical points $s_{\pm}=\mp \sqrt{1+az}$ uniformly
at finite distance in $z$. One has $\sqrt a(z-1)^{1/2}\leq \sqrt{\zeta_{1}^2-1}$, 
and using \eqref{eq:2} and \eqref{eq:2bis} we get, again by stationary phase near the two critical points $s_{\pm}$ and integration by part in $s$ elsewhere,
\begin{equation}\label{gl32}
J_{2}=\lambda^{-1/2}(z-1)^{-1/4}(e^{2/3 i\lambda (z-1)^{3/2}}c_{+}+e^{-2/3 i\lambda (z-1)^{3/2}}c_{-}) +O(\lambda^{-\infty}(z-1)^{-\infty})
\end{equation}
where $c_{\pm}(\sqrt a(z-1)^{1/2}, \lambda (z-1)^{3/2})$ are symbols of degree $0$
in the large parameter $\lambda (z-1)^{3/2}$. By \eqref{gl30} and \eqref{gl32},
one gets that formula \eqref{gl28} holds true with symbols
\begin{equation}\label{gl32bis}
\Theta_{\epsilon_{1},\epsilon_{2}} (z;a,\lambda)={(-i)^N\chi_2(z)(1-\chi_{5})(z)\chi_3(\sqrt{1+az})c_{\epsilon_{1}}b_{\epsilon_{2}} \over 4\pi\sqrt{1+az}
(z-1)^{1/4}(z-X)^{1/4}}\,
\end{equation}
which completes the proof of Lemma \ref{lemgl3}.
\end{proof}
Let 
\begin{equation*}
\begin{aligned}
  \label{gl33}
  \gamma_{N,1,\epsilon_{1},\epsilon_{2}}(T,X,Y,h)= &\frac{\sqrt a}{(2\pi h)^2} \int e^{i {a^{3/2}\over h } \eta    Y} w_{N,1,\epsilon_{1},\epsilon_{2}}(T,X,\hbar) \eta \chi_0(\eta) d\eta
\\
w_{N,1,\epsilon_{1},\epsilon_{2}}(T,X,\hbar)=& \int e^{i\lambda\Phi_{N,\epsilon_{1},\epsilon_{2}}}\Theta_{\epsilon_{1},\epsilon_{2}}(z;a,\lambda)dz\,.
\end{aligned}
\end{equation*}
In order to prove Proposition \ref{propL1}, we are reduced to proving the following inequality:
\begin{equation}\label{gl34}
\vert \sum_{2\leq N\leq C_{0}/\sqrt a} \gamma_{N,1,\epsilon_{1},\epsilon_{2}}(T,X,Y,h)\vert \leq C(2\pi h)^{-2}h^{1/3}
\end{equation}
 with a constant $C$ independent of $h\in ]0,h_{0}]$, $a\in [h^\alpha,a_{0}]$, $X\in [0,1]$, $T\in [0,a^{-1/2}]$.

For convenience, we take $Z=z^{3/2}$ as a new variable of integration so that
\begin{equation}\label{gl35}
w_{N,1,\epsilon_{1},\epsilon_{2}}(T,X,\hbar)=\int e^{i\lambda\Phi_{N,\epsilon_{1},\epsilon_{2}}}\tilde\Theta_{\epsilon_{1},\epsilon_{2}}(Z;a,\lambda)dZ\,;
\end{equation} 
$\tilde\Theta_{\epsilon_{1},\epsilon_{2}}(Z;a,\lambda)$ are now smooth functions of $Z$
with support in $[((1+z_{0})/2)^{3/2}, ((\zeta_{1}^2-1)/a)^{3/2}]$. Since
$dz= 2 Z^{-1/3}dZ/3$, we get  
$\vert Z^l\partial_{Z}^l\tilde\Theta\vert\leq C_{l}Z^{-2/3}$ with $C_{l}$ independent of
$a,\lambda$. One has

\begin{equation}
\begin{aligned}
  \label{gl36}
\partial_{Z}\Phi_{N,\epsilon_{1},\epsilon_{2}}= &{2\over 3}\Big(H_{a,\epsilon_{1},\epsilon_{2}}(T,X;Z)-2N(1-{3\over 4}B'(\lambda Z))\Big) 
\\
H_{a,\epsilon_{1},\epsilon_{2}}= & Z^{-1/3} \Big({T\over 2}(1+aZ^{2/3})^{-1/2} 
+ \epsilon_{1}(Z^{2/3}-1)^{1/2}+ \epsilon_{2}(Z^{2/3}-X)^{1/2}\Big)\\
\partial_{Z}H_{a,\epsilon_{1},\epsilon_{2}}= & {1\over 3}Z^{-4/3} \Big(-{T\over 2}(1+aZ^{2/3})^{-3/2}
(1+2aZ^{2/3}) \\
 &\,\,\,\,\,\,\,\,\,\,\,\,\,\,\,\,\,\,\,\,\,\,\,\,\,\,\,\,\,\,\,\,\,\,\,\,\,\,\,\,\,\,\,\,\,\,\,\,{}+ \epsilon_{1}(Z^{2/3}-1)^{-1/2}+ \epsilon_{2}X(Z^{2/3}-X)^{-1/2}\Big)\,.
\end{aligned}
\end{equation}

We will first prove that \eqref{gl34} holds true in the case $(\epsilon_{1},\epsilon_{2})=(+,+)$. From \eqref{gl36}, we get that 
the equation $\partial_{Z}H_{a,+,+}(Z)=0$ admits an unique solution 
$Z_{q}=Z^+_{q}(T,X,a)>1$, such that
\begin{equation}
\begin{aligned}
  \label{gl37}
\lim_{T\rightarrow \infty}Z^+_{q}(T,X,a) = &1 \quad \text{uniformly in} \ X,a\\   
0>{9\over 2}Z_{q}^{5/3}\partial^2_{Z}H_{a,+,+}(Z_{q}) =& -{aT\over 2}(1+aZ_{q}^{2/3})^{-5/2}
(\frac 1 2-aZ_{q}^{2/3})\\ 
 &{}-\frac 1 2(Z_{q}^{2/3}-1)^{-3/2}-\frac 1 2 X(Z_{q}^{2/3}-X)^{-3/2}\,.
\end{aligned}
\end{equation}
Therefore, the function $H_{a,+,+}(Z)$ is strictly increasing on $[1, Z_{q}[$,
and strictly decreasing on $]Z_{q},\infty[$. Observe that 
\begin{equation}\label{gl38}
H_{a,+,+}(1)={T\over 2}(1+a)^{-1/2}+(1-X)^{1/2}, \quad \lim_{Z\rightarrow \infty}
H_{a,+,+}(Z)=2\,.
\end{equation}
For all $k$ one has
\begin{equation}\label{gl39}
\sup_{Z\geq 1}\vert \partial^k_{Z}(NB'(\lambda Z)\vert \leq C_{k}N\lambda^{-2}Z^{-(k+2)}\leq 
C'_{k}\hbar^2 a^{-7/2}\leq C''_{k}h^\nu, \ \ \nu=2-7\alpha/2>0\,.
\end{equation}
Let $T_{0}\gg 1$. We first prove that \eqref{gl34} holds true
for $T\in [0,T_{0}]$. Since $H_{a,+,+}(Z)\leq C (1+T)$, for $N\geq N(T_{0})=C(1+T_{0})$, one gets $\vert \partial_{Z}\Phi_{N,+,+}(Z)\vert \geq c_{0}N$
with $c_{0}>0$, and $\vert \partial^k_{Z}\partial_{Z}\Phi_{N,+,+}(Z)\vert \leq c_{k}NZ^{-k}$ for $k\geq 1$. Therefore, by integration by parts in $Z$ in \eqref{gl35} with the operator
$L(\Theta)=\lambda^{-1}\partial_{Z}((\partial_{Z}\Phi_{N,+,+})^{-1}\Theta )$, one gets an extra 
factor $(\lambda N Z)^{-1}$ at each iteration. Thus, we get
$w_{N,1,+,+}\in O(N^{-\infty}\lambda^{-\infty})$, and this implies 

\begin{equation*}\label{gl40}
\sup_{T\leq T_{0}, X\in [0,1], Y\in \R}\bigl \vert \sum_{N(T_{0})\leq N\leq C_{0}/\sqrt a} \gamma_{N,1,\epsilon_{1},\epsilon_{2}}(T,X,Y,h)\bigr\vert \in O(h^\infty)\,.
\end{equation*}

Next, for $T\in [0,T_{0}]$ and $2\leq N \leq N(T_{0})$, one may estimate the sum in \eqref{gl34}
by the $\sup$ of each term. But in that case, we know by \eqref{gl36},
\eqref{gl37}, and \eqref{gl39} that there exists at most a critical
point of order $2$ near $Z=Z_{q}$ for $\Phi_{N,+,+}$, and 
$$
\vert\partial_{Z}\Phi_{N,+,+}\vert+\vert\partial^2_{Z}\Phi_{N,+,+}\vert+\vert\partial^3_{Z}\Phi_{N,+,+}\vert
\geq c >0\,.
$$
 Moreover, by  the second item of \eqref{gl38},
and $N\geq 2$, one has  a positive lower bound for $\vert\partial_{Z}\Phi_{N,+,+}(Z)\vert$
for large values of $Z$; thus, large values of $Z$ yield $O(\lambda^{-\infty})$ contributions
to $w_{N,1,+,+}$, and eventually the worst contribution to
$w_{N,1,+,+}$ will be the critical
point of order $2$ near $Z=Z_{q}$. This provides 
$$
\vert w_{N,1,+,+}(T,X,\hbar)\vert \leq C \lambda^{-1/3}\,\,\text{
  with}\,\,\,C \,\,\text{independent of}\,\,T\in [0,T_{0}], X\in
[0,1]\,.
$$
Since $a^{1/2}\lambda^{-1/3}=\hbar^{1/3}=(h/\eta)^{1/3}$, we get 

\begin{equation*}\label{gl41}
\sup_{T\leq T_{0}, X\in [0,1], Y\in \R}\vert \sum_{2\leq N\leq N(T_{0})} \gamma_{N,1,\epsilon_{1},\epsilon_{2}}(T,X,Y,h)\vert \leq C(T_{0})(2\pi h)^{-2}h^{1/3}\,.
\end{equation*}

Next we prove that that \eqref{gl34} holds true
for $T\in [T_{0}, a^{-1/2}]$. As before, we may assume $N\leq C_{1}T$, with $C_{1}$ large,
the contribution of
the sum  $C_{1}T\leq N\leq C_{0}/\sqrt a$ being negligible. Recall that we have $Z\geq Z_{0}=((1+z_{0})/2)^{3/2}>1$
on the support of $\tilde\Theta_{+,+}$ in formula \eqref{gl35}. By the first item
of \eqref{gl37}, one may choose  $T_{0}$  large enough so that $Z_{q}^+(T,X,a)\leq (1+Z_{0})/2<Z_{0}$ for all $T\geq T_{0}$. By the last item of \eqref{gl36}, increasing $T_{0}$
if necessary, and using \eqref{gl39}, we may assume with a constant $c>0$
\begin{equation*}
\vert\partial^2_{Z}\Phi_{N,+,+}(Z)\vert \geq cTZ^{-4/3}, \quad \forall Z\geq Z_{0}, \ 
\forall T\geq T_{0}, \ \forall N\leq C_{0}a^{-1/2}
\end{equation*}
 Therefore, on the support of $\tilde\Theta_{+,+}$, the phase
$\Phi_{N,+,+}$ admits at most one  critical point $Z_{c}=Z_{c}(T,X,N,\lambda,a)$ and this critical point is non degenerate. Since $N\geq 2$, from the first two items of \eqref{gl36} we get
$Z_{c}^{1/3}\leq T$, and this implies $Z_{c}^{1/3}\simeq T/N$. 
If $T/N$ is bounded, $Z_{c}$ is bounded, and since $\partial_{Z}\Phi\leq -c<0$
for large $Z$, we get by stationary phase 
$$
\vert w_{N,1,+,+}(T,X,\hbar)\vert \leq
C\lambda^{-1/2}T^{-1/2}\,\,\text{ with}\,\,C \,\,\text{ independent
  of}\, N.
$$
 If $T/N$ is large, then we perform the change of variable $Z= s(T/N)^3$ in \eqref{gl35}; the unique critical point $s_{c}$
remains in a fixed compact interval of $]0,\infty[$, one has $\partial_{s}\Phi\leq -c(T/N)^3<0$
for $s$ large and also 
$$
\partial_{s}^k\tilde\Theta_{+,+}(s(T/N)^3,a,\lambda)\leq
C_{k}(N/T)^2s^{-2/3-k}.
$$
Thus, by stationary phase
\begin{equation}\label{gl42}
\forall T\in [T_{0},a^{-1/2}]\,,\quad \sup_{2\leq N\leq C_{1}T}\sup_{X\in [0,1]}\vert w_{N,1,+,+}(T,X,\hbar)\vert \leq C\lambda^{-1/2}T^{-1/2}\,.
\end{equation}
By Lemma \ref{lemgl2}, we know that for any given $M=(X,Y,T)$
there is at most $C_{0}$ values  of $N$ such that the projection of ${\bf\Lambda}_{a,N,h}$
intersect the ball of radius $1$ centered at $M$; therefore, we will prove that the previous arguments  imply
\begin{equation}\label{gl43}
\sup_{T\in[T_{0},a^{-1/2]}, X\in [0,1], Y\in \R}\vert \sum_{2\leq N\leq C_{0}/\sqrt a} \gamma_{N,1,+,+}(T,X,Y,h)\vert \leq C(T_{0})(2\pi h)^{-2}(a^{-1/4}h^{1/2})
\end{equation}
and since $a^{-1/4}h^{1/2}\leq h^{1/3}$, we will get that \eqref{gl34} holds true.
Let us now explain more precisely how one can estimate the sum in \eqref{gl43} by the supremum over $N$. \\
Let $G_{N}(T,X,\lambda,a)=\Phi_{N,+,+}(T,X,Z_{c}(T,X,N,\lambda,a);a,\lambda)$. The stationary phase at the critical point $Z_{c}=Z_{c}(T,X,N,\lambda,a)$ in \eqref{gl35} gives
\begin{equation*}\label{gl44}
w_{N,1,+,+}(T,X,\hbar)=\lambda^{-1/2}T^{-1/2}e^{i\lambda G_{N}(T,X,\lambda,a)}\psi_{N}(T,X;a,\lambda)\,.
\end{equation*}
By \eqref{gl30}, \eqref{gl32}, \eqref{gl32bis}, we know that the $\psi_{N}(T,X;a,\lambda)
$ are symbols of degree $0$ in $\lambda$, and $\partial_{\lambda}^k\psi_{N}\leq C_{k}\lambda^{-k}$ with $C_{k}$ independent of $2\leq N \leq C_{1}T$. This gives with 
$\tilde\lambda=a^{3/2}/h=\lambda/\eta$
\begin{multline}
  \label{gl45}
  \gamma_{N,1,+,+}(T,X,Y,h)= \frac{T^{-1/2}h^{1/2}a^{-1/4}}{(2\pi
    h)^2} \\
\times 
  \int e^{i \tilde\lambda \, \eta  (Y+ G_{N}(T,X,\tilde\lambda\eta,a))}
   \psi_{N}(T,X;a,\tilde\lambda\eta)\eta^{1/2} \chi_0(\eta) d\eta\,.
\end{multline}
This is an integral with large parameter $\tilde\lambda$ and phase 
$$
L_{N}(T,X,Y,\eta,\tilde\lambda)=\eta(Y+G_{N}(T,X,\tilde\lambda\eta)).
$$
 By construction, the equation 
$$
\partial_{\eta}L_{N}=Y+G_{N}(T,X,\lambda,a)+\lambda\partial_{\lambda}G_{N}(T,X,\lambda,a)=0
$$
implies that $(X,Y,T)$ belongs to the projection of ${\bf \Lambda}_{a,N,h}$ on $\R^3$. Let $T\in [T_{0},a^{-1/2}], X\in [0,1]$ and 
$Y\in \R$ be given.  For $N\notin \mathcal N_{1}(X,Y,T)$,
one has therefore $\partial_{\eta}L_{N}(T',X',Y',\eta,\tilde\lambda)\not=0$ for all
$\lambda$ and all $X'\in [0,1], \vert Y'-Y\vert+\vert T'-T\vert\leq 1$.
This implies, since $\partial_{\eta}L_{N}$ is linear in $Y$, 
$\vert \partial_{\eta}L_{N}(T,X,Y,\eta,\tilde\lambda)\vert\geq 1$. 
Moreover, one has, with $C_{k}$ independent of $N,T,X,\eta,\lambda,a$
\begin{equation}\label{gl46}
\vert \partial^k_{\eta}(\partial_{\eta}L_{N})\vert\leq C_{k}\,.
\end{equation}
To prove \eqref{gl46}, we just use that $\partial_{\lambda}Z_{c}$ satisfies
$$\partial_{\lambda}Z_{c}\partial_{Z}^2 \Phi_{N,+,+}(Z_{c})=-\partial_{\lambda}\partial_{Z}\Phi_{N,+,+}(Z_{c})= NZ_{c}B''(\lambda Z_{c})$$
and thus from \eqref{gl36} and \eqref{gl39}, we get for all $k\geq 1$,
$(\eta\partial_{\eta})^kZ_{c}=(\lambda\partial_{\lambda})^kZ_{c}\in O(h^\nu)$. Then \eqref{gl46} follows from 
 $$
\lambda\partial_{\lambda}G_{N}(T,X,\lambda)=\lambda(\partial_{\lambda}\Phi_{N,+,+})(T,X,Z_{c};a,\lambda)=
{N\over\lambda}(-B(\lambda Z_{c})+\lambda Z_{c}B'(\lambda Z_{c}))\,.
$$
Therefore, by integration by parts in $\eta$ in \eqref{gl45}, we get  
\begin{equation}\label{gl43bis}
\sup_{T\in[T_{0},a^{-1/2]}, X\in [0,1], Y\in \R}\vert \sum_{N\notin \mathcal N_{1}(X,
Y,T)} \gamma_{N,1,+,+}(T,X,Y,h)\vert \in O(h^\infty)\,.
\end{equation}
Finally, by Lemma \ref{lemgl2}, one has $\vert \mathcal N_{1}(X,Y,T)\vert \leq C_{0}$, and therefore, we get from \eqref{gl43bis} and \eqref{gl42} that \eqref{gl43} holds true.

Next, we show that \eqref{gl34} holds true for $(\epsilon_{1},\epsilon_{2})=(-,+)$.
In that case, from the last item of \eqref{gl36} and $X\in [0,1]$, one gets 
$\partial_{Z}H_{a,-,+}<0$ for $T>0$. Therefore the function $H_{a,-,+}(Z)$
decreases on $[1,\infty[$, from $H_{a,-,+}(1)={T(1+a)^{-1/2}\over 2}+ (1-X)^{1/2}$ to
 $H_{a,-,+}(\infty)=0$. The
equation $\partial_{Z}\Phi_{N,-,+}=0$ admits an unique solution  $Z_{c}$, and this critical
point is non degenerate. We can thus argue as we have done before for the $(+,+)$ case.

Finally, one sees that the case $(\epsilon_{1},\epsilon_{2})=(+,-)$ is similar
to the $(+,+)$ case, and the case $(\epsilon_{1},\epsilon_{2})=(-,-)$ is similar
to the $(-,+)$ case. We leave the details to the reader.

The proof of Proposition \ref{propL1} is complete. 
 \subsubsection{Proof of Proposition \ref{propL2}}
 Recall that
    \begin{multline}\label{gl50}
w_{N,2}(T,X,\hbar)= \frac{(-i)^N\lambda}{2\pi} \int e^{i \lambda \varphi_\anl}
\chi_2(z)\frac{\chi_3((1+az)^\frac 1 2)}{2(1+az)^\frac 1 2} \\
{}\times \chi_4(
a^\frac 1 2
\sigma) \sigma_0( a^\frac 1 2 T',\hbar)\chi_{5}(z) \,dT'd\sigma dz
\end{multline}
as well as $\partial_{T'}\varphi_\anl= {\mu^2+1-z\over \sqrt{1+az}+\sqrt{1+a+a\mu^2}}$
and $\partial_{\sigma}\varphi_\anl= X-z+\sigma^2$. \\Let 
$K=\{ T'=0, \sigma \in [-1,1],z=1\}$, let $\omega$ be a small neighborhood of $K$
and $\chi_{6}(T',\sigma,z)\in C_{0}^\infty(\omega)$ equal to $1$ near K. Since 
for $z$ in the support of the integral \eqref{gl50}, one has $z\in[\beta/2,z_{0}]$ with $z_{0}>1$ close to $1$ , decreasing $z_{0}>1$ if necessary, we get
by integration by parts in $T',\sigma$
    \begin{equation}\label{gl50bis}\begin{aligned}
w_{N,2}(T,X,\hbar)= & \frac{(-i)^N\lambda}{2\pi} \int e^{i \lambda \varphi_\anl}
\chi (T',\sigma,z; a,\hbar) \,dT'd\sigma dz + O(\lambda^{-\infty}) \\
\chi (T',\sigma,z; a,\hbar)= &\chi_2(z)\frac{\chi_3(\sqrt{1+az})}{2\sqrt{1+az}} \chi_4(\sqrt a
\sigma) \sigma_0(\sqrt a T',\hbar)\chi_{5}(z)\chi_{6}(T',\sigma,z)
\end{aligned}\end{equation}
with $O(\lambda^{-\infty})$ uniform in $T,X,N,a$. Moreover, $\chi(T',\sigma,z; a,\hbar)$
is a classical symbol of degree $0$ in $\hbar$, with support  $(T',\sigma,z)\in \omega$ and $a$
is just a harmless  parameter in $\chi$. 

  We first perform the integration with respect to $z$ in \eqref{gl50bis}. Recall
\begin{equation*}
  \label{gl51}
  \begin{aligned}
  \varphi_\anl(T,X,T',\sigma,z)=
  &\gamma_a(z)(T-T')+\tilde\psi_a(T')+\sigma(X-z)+\sigma^3/3\\ & {}+N(-\frac
  4 3 z^\frac 3 2+\frac 1 \lambda B(\lambda z^\frac 3 2))\\
 \partial_{z}\varphi_\anl(T,X,T',\sigma,z)= &{T-T'\over 2(1+az)^{1/2}} -\sigma
 -2Nz^{1/2}(1-{3\over 4}B'(\lambda z^{3/2})) \,.
\end{aligned}
\end{equation*}
Thus, $\varphi_\anl$ admits a unique critical point $z_{c}(T,T',\sigma,a,\lambda)$, and we are just
interested with values of the parameters such that $z_{c}$ is close to $1$.
With $u=({T-T'\over 2} -\sigma)/2N$, this means $u$ close to $1$. 
Since $\sigma$ is close to $[-1,1]$ and $N\geq 1$, we may thus assume
 $\tilde T=T/4N$ close to $[1/2,3/2]$, say $\tilde T\in [1/4,2]$.
We denote by $g(\tilde T,T',\sigma;a,N,\lambda)$ various functions
which are classical symbols of degree $0$ in $\lambda$ and with parameters $a,N$; in particular, 
with $w=(\tilde T,T',\sigma)$, for all $\alpha$, there exists $C_{\alpha}$ independent of 
$a,N,\lambda$ such that $\vert \partial_{w}^\alpha g \vert \leq C_{\alpha}$ for all
 $\tilde T \in [1/4,2]$ and all $(T',\sigma)$ close to $(0,0)$.
 
 We denote by $f_{k}(a,\tilde T, T'/N,\sigma/N)$ functions which are homogeneous 
 of degree $k$ in
 $(T'/N,\sigma/N)$. The notation $\mathcal O_{j}$ means any function of the form
 $f=\sum _{k\geq j}f_{k}$.
 We will use the following functions
 \begin{equation}\label{gl52bis}
 \begin{aligned}
 F_{0}= &{2\tilde T^2\over 1+\sqrt{1+4a\tilde T^2}} \\
 G_{0}^{-1}=& F_{0}^{1/2}(1+aF_{0})^{1/2}({1\over F_{0}}+{a\over 1+aF_{0}})\\
 H_{0}= &{1-F_{0}\over \sqrt{1+a}+\sqrt{1+aF_{0}}}\\
 F_{1}=&-{G_{0}\over N}(T'/2+\sigma (1+aF_{0})^{1/2})\,.
 \end{aligned}
 \end{equation}
\begin{lemma}\label{lemgl4} 
  \begin{itemize}
  \item[(a)] One has 
$$
z_{c}=F_{0}+F_{1}+ \mathcal O_{2}+g_{0}/\lambda^2\,.
$$
\item[(b)] The critical value $\Psi_{a,N,\lambda}(\tilde T,X,T',\sigma)=\varphi_\anl(T,X,T',\sigma,z_{c})$ is equal to
\begin{equation}\label{gl52}
\begin{aligned}
\Psi_{a,N,\lambda}= &(X-F_{0})\sigma+\sigma^3/3 + H_{0}T'+\tilde\psi_{a}(T')\\
&{}+{G_{0}\over 2N(1+aF_{0})^{1/2}}(\sigma (1+aF_{0})^{1/2} +T'/2)^2 \\
&{}-{1\over 12 N^2}(T'/2+\sigma)^3+ aN \mathcal O_{3} + g_{1}/\lambda^2\\&{}+ N(4\tilde T \gamma_{a}(F_{0})-{4\over 3}F_{0}^{3/2}+g_{2}(\tilde T,a,N,\lambda)/\lambda^2) \,.
\end{aligned}
\end{equation}
  \end{itemize}
\end{lemma}
\begin{proof}
(a) The equation for $z_{c}$ when $\lambda=\infty$ is 
$$
 z^{1/2}(1+az)^{1/2}= \tilde T -{1\over 2N}(T'/2+\sigma
(1+az)^{1/2})\,.
 $$
The solution of this equation is clearly of the form $z=\sum f_{k}(a,\tilde T, T'/N,\sigma/N)$
with $f_{0}=F_{0}$ solution of $F_{0}(1+aF_{0})= \tilde T^2$, and we get $F_{1}$
by a Taylor expansion at order $1$. Then (a) is a  consequence of the implicit function theorem 
applied to 
$$
 z^{1/2}(1+az)^{1/2} (1-{3\over 4}B'(\lambda z^{3/2}))= \tilde T
 -{1\over 2N}(T'/2+\sigma (1+az)^{1/2})\,.
 $$
 To prove part (b), one may of course 
insert the formula for $z_{c}$ into the definition of  $\varphi_\anl$. Another way is
to use
\begin{equation*}\label{gl53}
\begin{aligned}
\partial_{\sigma}(\Psi_{a,N,\lambda}-\sigma X-\sigma^3/3-\tilde\psi_{a})= & -z_{c}\\
\partial_{T'}(\Psi_{a,N,\lambda}-\sigma X-\sigma^3/3-\tilde\psi_{a})= & -\gamma_{a}(z_{c})\,.
\end{aligned}
\end{equation*}
Using part (a), this system is easily seen to be integrable and yields formula \eqref{gl52}
up to an integration constant which is easy to compute when $T'=\sigma=0$.
Moreover, when $a=0$ and  $\lambda=\infty$, one has $z_{c}^{1/2}= \tilde T-{1\over 2N}(T'/2+\sigma)$; therefore, one can easily compute the value of the critical value when 
$a=0,\lambda=\infty$, and this provides
the first two terms on the second line of \eqref{gl52}.   
The proof of our lemma is complete.\end{proof}
From $aT\in O(a^{1/2})$
and $N\lambda^{-2}\in O(h^\nu)$, we get
$\partial_{z}^2\varphi_\anl=-Nz^{-1/2}+O(a^{1/2}+h^\nu)$ for any $z$ close to $1$. 
Decreasing $a_{0}$ and $h_{0}$ if necessary, we get by
stationary phase, 
\begin{equation}\label{gl54}
 \int e^{i \lambda \varphi_\anl}
\chi (T',\sigma,z; a,\hbar) \,dT'd\sigma dz= \frac{2\pi}{\sqrt{\lambda N}} \int e^{i \lambda \Psi_\anl}\,.
\tilde \chi ( \tilde T,T',\sigma;1/N, a,\hbar) \,dT'd\sigma\,.
\end{equation}
Here, $\tilde \chi$ is a classical symbol of degree $0$ in $\hbar$, with harmless parameters
$a,1/N$.
Let us define $\tilde\gamma_{N,2}(T,X,Y,h)$ by the following formula,
where $\tilde\lambda=a^{3/2}/h=\lambda/\eta$:
\begin{equation}
  \label{gl55}
  \tilde\gamma_{N,2}(T,X,Y,h)=  
  \int e^{i \tilde\lambda \eta  (Y+ \Psi_{a,N,\eta\tilde\lambda})} 
   \tilde \chi \eta^{3/2} \chi_0(\eta)dT'd\sigma d\eta\,.
  \end{equation}
By \eqref{gl50bis} and \eqref{gl54}, Proposition \ref{propL2} is clearly a consequence of the following estimate:

\begin{equation}\label{gl56}
\sum_{4N/T \in [1/4,2]}{1 \over \sqrt N} 
\vert \tilde\gamma_{N,2}(T,X,Y,h) \vert \leq C \tilde\lambda^{-3/4}
\end{equation} 
with $C$ independent of $X\in [0,1]$, $T\in ]0,a^{-1/2}]$, $Y\in \R$,
 $a\in [h^\alpha, a_{0}]$   and
$\tilde\lambda\in [\tilde\lambda_{0},\infty[$ with $a_{0}$ small and $\tilde\lambda_{0}$ large.
\begin{lemma}\label{lemgl5} 
For all $k$ there exist $C_{k}$ independent of $X\in [0,1]$, $T\in ]0,a^{-1/2}]$, $Y\in \R$,
 $a\in [h^\alpha, a_{0}]$   and
$\tilde\lambda\in [\tilde\lambda_{0},\infty[$ such that
 \begin{equation*}\label{gl58}
\sum_{\substack{
            N\notin \mathcal N_{1}(X,Y,T)\\
            4N/T \in [1/4,2]}}
{1 \over \sqrt N} \vert \tilde\gamma_{N,2}(T,X,Y,h) \vert \leq C_{k} \tilde\lambda^{-k}\,.
\end{equation*} 
\end{lemma}
\begin{proof}

Recall that $T'=-2\mu\sqrt{1+a+a\mu^2}$. Let us define the functions (see \eqref{gl4})
 \begin{equation*}\label{gl59}
 \begin{aligned}
\mathcal X & = 1+\mu^2-\sigma^2\\
\mathcal Y &=2\mu^2(\mu-\sigma) H_1+\frac 2 3(\sigma^3-\mu^3)+4N(1-\frac
 3 4B'(\eta\tilde\lambda (1+\mu^2)^\frac 3 2)) H_2\\
\mathcal T & =2\sqrt{\rho+a\mu^2}\bigl(\sigma-\mu+2N\sqrt{1+\mu^2}(1-\frac 3 4B'(\eta\tilde\lambda (1+\mu^2)^\frac 3 2))\bigr)\,.
\end{aligned} 
\end{equation*}
There exists a universal constant $C_{0}$, with
$\Phi=\eta(Y+\phi_{a,N,\eta\tilde\lambda})$, such that
\begin{equation}\label{gl60}
\vert X- \mathcal X\vert+ \vert Y- \mathcal Y\vert+ \vert T- \mathcal T\vert
\leq C_{0}(\vert \Phi'_{\eta}\vert+\vert \Phi'_{T'}\vert+\vert \Phi'_{\sigma}\vert+\vert \Phi'_{z}\vert)\,.
\end{equation}
The left hand side of \eqref{gl60} does not depend of $z$. Since $G=\eta(Y+\phi_{a,N,\eta\tilde\lambda})$ is a critical value of $\Phi$
with respect to $z$, we get
\begin{equation*}\label{gl61}
\vert X- \mathcal X\vert+ \vert Y- \mathcal Y\vert+ \vert T- \mathcal T\vert
\leq C_{0}(\vert G'_{\eta}\vert+\vert G'_{T'}\vert+\vert G'_{\sigma}\vert)\,.
\end{equation*}
Consider any given $(X,Y,T)$. Then for $N\notin \mathcal N_{1}(X,Y,T)$, we have
$\vert x- \mathcal X\vert+ \vert y- \mathcal Y\vert+ \vert t- \mathcal T\vert\not = 0$
for all values of $\mu, \sigma, \eta,\tilde\lambda$, all $\vert x-X\vert\leq 1$, all $\vert y-Y\vert\leq 1$, all $\vert t-T\vert\leq 1$.
This implies $\vert X- \mathcal X\vert+ \vert Y- \mathcal Y\vert+ \vert T- \mathcal T\vert\geq 1$.
Therefore, for all values of $\mu, \sigma, \eta,\tilde\lambda$, and 
 $N\notin \mathcal N_{1}(X,Y,T)$ we get
$$
  \vert G'_{\eta}\vert+\vert G'_{T'}\vert+\vert G'_{\sigma}\vert \geq
  1/C_{0}\,.
$$
From \eqref{gl52}, using $\eta\partial_{\eta}=\lambda\partial_{\lambda}$, 
$N/\lambda^2 \in O(h^\nu)$ and the fact that any function of type $N\mathcal O_{3}$ is in $O(N^{-2})$ , we get that all the derivatives of $G$ with respect to $\eta, T', \sigma$
are uniformly bounded.  By integration by part in \eqref{gl55} we thus get 
$\vert \tilde\gamma_{N,2} \vert \leq C_{k}\tilde\lambda^{-k}$ for all $k$,
with  $C_{k}$ independent of $X\in [0,1]$, $T\in ]0,a^{-1/2}]$, $Y\in \R$,
   and $N\notin \mathcal N_{1}(X,Y,T)$
. The proof of Lemma \ref{lemgl5} is complete.
\end{proof}

From Lemma \ref{lemgl5}, and since $\vert \mathcal N_{1}(X,Y,T)\vert$ is uniformly bounded  by Lemma \ref{lemgl2}, we get that
 \eqref{gl56} will be a consequence of  
\begin{equation}\label{gl56bis}
\forall N  \,\text{ with } \,4N/T\in [1/4,2]\,,\,\,{1 \over \sqrt N} 
\vert \int e^{i\lambda \Psi_{a,N,\lambda}}\tilde \chi dT'd\sigma \vert \leq C \lambda^{-3/4}\,.
\end{equation} 
Here and after, we denote by $C$ any constant which is independent of $N\geq 1$, $X\in [0,1]$, $T\in ]0,a^{-1/2}]$, 
 $a\in [h^\alpha, a_{0}]$   and 
$\lambda\in [\lambda_{0},\infty[$ with $a_{0}$ small and $\lambda_{0}$ large.

Observe that we can now replace the phase $\Psi_{a,N,\lambda}$ by the phase
\begin{multline}\label{gl52bisbis}
\psi_{a,N,\lambda}=  (X-F_{0})\sigma+\sigma^3/3 + H_{0}T'+\tilde\psi_{a}(T')\\{}+{G_{0}\over 2N(1+aF_{0})^{1/2}}(\sigma (1+aF_{0})^{1/2} +T'/2)^2 -{1\over 12 N^2}(T'/2+\sigma)^3+ aN \mathcal O_{3} 
\end{multline}
since by \eqref{gl52} the difference $\Psi_{a,N,\lambda}-(\psi_{a,N,\lambda}+g_{1}/\lambda^2)$
does not depend on $T',\sigma$, and $e^{ig_{1}/\lambda}$ is a classical symbol of order $0$
in $\lambda$.
We set $\underline \chi = e^{ig_{1}/\lambda} \tilde \chi$ and recall $\hbar=a^{3/2}/\lambda$ . 
Then $\underline \chi(T',\sigma;\tilde T,a,N;\lambda)$ is a classical symbol
of degree $0$ in $\lambda\geq \lambda_{0}$, compactly supported in
$(T',\sigma)$ close to $\{0\}\times [-1,1]$. Moreover,
 for all $\alpha$, there exists $C_{\alpha}$ independent of $a,N$ 
and $\tilde T\in [1/4,2]$, such that $\sup_{(T',\sigma,\lambda)}\vert\partial^\alpha_{(T',\sigma)}\underline \chi\vert \leq C_{\alpha}$.\\

\begin{lemma}\label{lemgl6} There exists $C$ such that 
for all $N\geq \lambda^{1/3}$,
\begin{equation}\label{gl63}
{1 \over \sqrt N} 
\vert \int e^{i\lambda \psi_{a,N,\lambda}}\underline\chi dT' d\sigma \vert \leq C \lambda^{-5/6}\,.
\end{equation}
\end{lemma}

\begin{proof} It is sufficient  to prove that, for all $N\geq \lambda^{1/3}$,
\begin{equation}\label{gl64}
 \vert \int e^{i\lambda \psi_{a,N,\lambda}}\underline\chi dT' d\sigma \vert \leq C \lambda^{-2/3}\,.
\end{equation}
Set $X-F_{0}=-A\lambda^{-2/3}, H_{0}=-B\lambda^{-2/3}$ and perform a
change of variable in \eqref{gl64}: $T'=\lambda^{-1/3}x, \sigma=
\lambda^{-1/3}y$ . We are reduced to proving
\begin{equation}\label{gl65}
 \vert \int e^{i G}\underline \chi (\lambda^{-1/3}x, \lambda^{-1/3}y,...) dx dy \vert \leq C 
\end{equation}
with a phase G of the following form
\begin{multline}\label{gl52ter}
G= -Ay+y^3/3 - Bx+ \lambda\tilde\psi_{a}(\lambda^{-1/3}x)
+{G_{0}\lambda^{1/3}\over 2N(1+aF_{0})^{1/2}}(y (1+aF_{0})^{1/2} +x/2)^2 \\
{}+N^{-2}(x,y)^3 f(a,\tilde T,\lambda^{-1/3}x/N,\lambda^{-1/3}y/N) \,.
\end{multline}
Then \eqref{gl65} is an oscillatory integral  over
a domain of integration  of size $\lambda^{1/3}$. Parameters $F_{0},G_{0},\lambda^{1/3}/N$
are bounded, and the main point is to prove that the constant $C$ is uniform in 
$(A,B)=(r\cos\theta,r\sin\theta)$ with $r\leq c_{0}\lambda^{2/3}$. 
Recall  $\lambda\tilde\psi_{a}(\lambda^{-1/3}x)={x^3\over 24 (1+a)^{3/2}}(1+O(ax^2\lambda^{-2/3}))$.
One has
\begin{equation}\label{gl52quad}
\begin{aligned}
\partial_{x}G=&-B+\lambda^{2/3}\tilde\psi_{a}'(\lambda^{-1/3}x)+O((x,y))+O(1)\\
\partial_{y}G=&-A+y^2+O((x,y))+O(1)\,.
\end{aligned}
\end{equation}
Moreover, since $\underline\chi$ is compactly supported in $(T',\sigma)$
one has, with $C_{\alpha}$ independent of $\tilde T,a,N,\lambda$,
$$
\sup_{(x,y)}\vert\partial^\alpha_{(x,y)}\underline
\chi(\lambda^{-1/3}x, \lambda^{-1/3}y,...)\vert \leq
C_{\alpha}(1+\vert x\vert+\vert y
\vert)^{-\vert\alpha\vert}\,.
$$
 Therefore, for any $r_{0}$, the
oscillatory integral \eqref{gl65} is clearly bounded for $0\leq r \leq
r_{0}$ (integrate by part for large $(x,y)$).

For $r\in [r_{0},c_{0}\lambda^{2/3}]$, we rescale variables $(x,y)=r^{1/2}(x',y')$,
and we set $G=r^{3/2}G'$ and 
$\chi'(x',y',...)=\underline\chi (r^{1/2}\lambda^{-1/3}x', r^{1/2}\lambda^{-1/3}y',...)$.
Observe that since $r^{1/2}\lambda^{-1/3}$ is bounded, we still have decay estimates
$$\sup_{(x',y')}\vert\partial^\alpha_{(x',y')}\chi'\vert \leq
C_{\alpha}(1+\vert x'\vert+\vert y' \vert)^{-\vert\alpha\vert}\,.
$$
 We have to prove
\begin{equation}\label{gl65bis}
 r\vert \int e^{i r^{3/2}G'}\chi'  dx' dy' \vert \leq C \,.
\end{equation}
First, the critical points of $G'$ satisfy
\begin{equation}\label{gl52+}
\begin{aligned}
\partial_{x'}G'=&-\cos\theta+r^{-1}\lambda^{2/3}\tilde\psi_{a}'(r^{1/2}\lambda^{-1/3}x')+r^{-1/2}O((x',y'))+r^{-1}O(1)\\
\partial_{y}G=&-\sin\theta+y'^2+r^{-1/2}O((x',y'))+r^{-1}O(1)\,.
\end{aligned}
\end{equation} 
Let $F(u)$ denote a  smooth function near $u=0$, we then have
$$
r^{-1}\lambda^{2/3}\tilde\psi_{a}'(r^{1/2}\lambda^{-1/3}x')={x'^2\over 8(1+a)^{3/2}}
(1+ar\lambda^{-2/3}x'^2F(a^{1/2}r^{1/2}\lambda^{-1/3}x'))\,.
$$
By \eqref{gl52+}, the contribution of large $(x',y')$ to \eqref{gl65bis} 
is $O(r^{-\infty})$, and we may localize the integral on a compact set in $(x',y')$.
For $\vert \sin\theta\vert \geq 0.1$, and $r_{0}$ large, we  have two distinct 
non degenerate critical points $y'_{\pm}=\pm \vert\sin\theta\vert^{1/2}+O(r^{-1/2})$
in $y'$ with critical values $G'_{\pm}(x',...)$
and thus we get by stationary phase
\begin{equation}\label{gl65ter}
r \int e^{i r^{3/2}G'}\chi'  dx' dy'= r^{1/4}(\int e^{i r^{3/2}G_{+}'}\chi_{+}'  dx'+
\int e^{i r^{3/2}G_{-}'}\chi_{-}'  dx')\,.
\end{equation}
Moreover, one has  
$$
\partial_{x'}G'_{\pm}(x',...)=\partial_{x'}G'(x',y'_{\pm},...)
=-\cos\theta+r^{-1}\lambda^{2/3}\tilde\psi_{a}'(r^{1/2}\lambda^{-1/3}x')+O(r^{-1/2})\,
$$
and this imply $\vert \partial^3_{x'}G'_{\pm}\vert\geq c_{0}>0$. Thus, we get 
$$
\vert \int e^{i r^{3/2}G_{\pm}'}\chi_{\pm}'  dx'\vert \leq
C(r^{3/2})^{-1/3}=Cr^{-1/2}\,.
$$
Therefore, \eqref{gl65bis} holds true, and in fact, we have the
following better estimate in the range $r\geq 1$:
\begin{equation}\label{gl65terbis}
 r\vert \int e^{i r^{3/2}G'}\chi'  dx' dy' \vert \leq Cr^{-1/4}\,.
\end{equation}
If $\sin\theta$ is close to $0$, then we first perform the stationary phase in 
$x'$, and we use the same arguments. This completes the proof of Lemma \ref{lemgl6}. 
\end{proof}
Therefore, we can now assume that ${\lambda\over N^3}=\Lambda \geq 1$, and take
$\Lambda$ as our new large parameter. We set $X-F_{0}=-pN^{-2}, H_{0}=-qN^{-2}/2$ and 
we change variables in \eqref{gl56bis}: $T'=-2x/N, \sigma= -y/N$.
We will prove 
\begin{equation}\label{gl66}
 \vert \int e^{i\Lambda \mathcal G_{a}}\underline \chi (x/N, y/N,...) dx dy \vert \leq C \Lambda^{-3/4}
\end{equation}
with a phase $\mathcal G_{a}$ which takes the following form
\begin{multline}\label{gl52quadbis}
\mathcal G_{a}= py-y^3/3 +qx+ N^3\tilde\psi_{a}(-2x/N)
+{G_{0}\over 2(1+aF_{0})^{1/2}}(y (1+aF_{0})^{1/2}+x)^2 \\
{}+{1\over 12 N^2}(x+y)^3+aN^{-2}(x,y)^3 f(a,\tilde T, {x\over N^2},{y\over N^2})\,. 
\end{multline}
Observe that, for large $N$, \eqref{gl66} implies a better estimate
that \eqref{gl56}; more precisely, \eqref{gl66} is equivalent to
 \begin{equation}\label{gl66bis}
{1 \over \sqrt N} 
\vert \int e^{i\lambda \Psi_{a,N,\lambda}}\tilde \chi dT'd\sigma \vert \leq C N^{-1/4}\lambda^{-3/4}\,.
\end{equation} 
The above estimate is of course compatible with \eqref{gl63} for
$N\simeq \lambda^{-1/3}$.

Recall, see \eqref{eq:2}, that $\tilde \psi_{a}(T')={T'^3\over 24(1+a)^{3/2}}(1+O(aT'^2))$.
From \eqref{gl52ter} we get
\begin{equation}\label{gl67}
\begin{aligned}
\partial_{x}\mathcal G_{a}= & q-x^2 
+{G_{0}\over (1+aF_{0})^{1/2}}(y (1+aF_{0})^{1/2}+x) 
 +{1\over 4N^2}(x+y)^2+aO((x,y)^2)\\
\partial_{y}\mathcal G_{a}= & p-y^2 
+G_{0}(y (1+aF_{0})^{1/2}+x) 
 +{1\over 4N^2}(x+y)^2+aO((x,y)^2)\\ 
\end{aligned}
\end{equation}
and
\begin{equation}\label{gl68}
\begin{aligned}
\partial_{x,x}^2\mathcal G_{a}= & -2x 
+{G_{0}\over (1+aF_{0})^{1/2}}
 +{1\over 2N^2}(x+y)+aO((x,y))\\
\partial^2_{x,y}\mathcal G_{a}= & G_{0}+{1\over 2N^2}(x+y)+aO((x,y))\\
\partial_{y,y}^2\mathcal G_{a}= & -2y 
+G_{0}(1+aF_{0})^{1/2}
 +{1\over 2N^2}(x+y)+aO((x,y))\,.
\end{aligned}
\end{equation}

Thus we get the value of the hessian $\mathcal H_{N}(x,y;\tilde T,a)$
\begin{align}\label{gl69}
\mathcal H_{N}(x,y;\tilde T,a) & =\mathrm{det}\begin{pmatrix} \partial_{x,x}^2\mathcal G_{a} & \partial^2_{x,y}\mathcal G_{a}
  \\ \partial^2_{y,x}\mathcal G_{a} & \partial_{y,y}^2\mathcal
  G_{a}\end{pmatrix}\\
& =
 -2G_{0}(x+y)+4xy -{1\over N^2}(x+y)^2+aO((x,y))\nonumber\,.
\end{align}

\begin{lemma}\label{lemgl7}There exists $r_{0}$ and  $C$ such that 
for all $(p,q)$ with $\vert (p,q) \vert \geq r_{0}$  
\begin{equation}\label{gl70}
\vert \int e^{i\Lambda \mathcal G_{a}}\underline \chi (x/N, y/N,...) dx dy \vert\leq C \Lambda^{-5/6}\,.
\end{equation}
 \end{lemma}
\begin{proof} Set $(q,p)=(r\cos\theta,r\sin\theta)$ with $r\geq r_{0}$ large. 
Let $\chi\in C_{0}^\infty(\vert (x,y)\vert<c)$ with small $c$ and  $\chi=1$ near $0$. 
Then by \eqref{gl67}  we get for all $k$ by integration by part in $(x,y)$
$$
 \vert \int e^{i\Lambda \mathcal G_{a}}\chi(r^{-1/2}(x,y))
\underline \chi (x/N, y/N,...) dx dy \vert
\leq C_{k}r^{-k}\Lambda^{-k}\,.
$$
For the remaining term, we perform the change of variable $(x,y)=r^{1/2}(x',y')$ and we
 set $\mathcal G'_{a}=r^{-3/2}\mathcal G_{a}$. It remains to prove
\begin{equation}\label{gl71}
\vert r\int e^{ir^{3/2}\Lambda \mathcal G'_{a}}(1-\chi)(x',y')
\underline \chi (r^{1/2}x'/N, r^{1/2}y'/N,...) dx' dy'\vert \leq C \Lambda^{-5/6}\,.
\end{equation}
Observe that since $(1-\chi)(x',y')=0$ near $0$, $(1-\chi)(x',y')=1$ for 
$\vert (x',y')\vert\geq c$, and since $\underline \chi (u, v,...)$ is compactly supported
in $(u,v)$,  we still have 
$$
sup_{(x',y')}\vert\partial_{(x',y')}^\alpha (1-\chi)(x',y') \underline
\chi (r^{1/2}x'/N, r^{1/2}y'/N,...)\vert \leq C_{\alpha}(1+\vert
x'\vert + \vert y'\vert)^{-\vert\alpha\vert}\,.
$$
The phase $\mathcal G'_{a}$ is of the form 
$$
 \mathcal G'_{a}= \sin\theta y'-{y'^3\over 3} +\cos\theta x'
-{x'^3\over 3}++{1\over 12 N^2}(x'+y')^3 +r^{-1/2} O((x',y')^2)+a
O((x',y')^3)
$$
where $O((x',y')^k)$ means any function of the form $x'^jy'^l f(r^{1/2}x'/N,r^{1/2}y'/N,a,N)$
with $f$ smooth uniformly in $a,N$ and $j+l=k$. Thus for small $a$ and
large $r_{0}$, we may localize the integral \eqref{gl71}  to a compact
set in $(x',y')$ (integrate by part). 
The hessian of $G'_{a}$ is equal to $4x'y'-{1\over N^2}(x'+y')^2+ O(r^{-1/2}+a)$.
Thus for $N\geq 2$, $a$ small and $r_{0}$ large, the set on which the
hessian vanishes defines
a smooth curve $\Gamma$ outside $(x',y')=(0,0)$, which is close to the union of the two lines
 $c(x'+y')\pm (x'-y')=0$, $c^2={N^2-1\over N^2}\in [3/4,1]$.
Moreover, one has
\begin{equation} \label{gl72}
\begin{aligned}
\partial_{x'}\mathcal G'_{a}=& \cos\theta -x'^2 +{1\over 4N^2}(x'+y')^2 +O(r^{-1/2}+a)\\
\partial_{y'}\mathcal G'_{a}=& \sin\theta -y'^2 +{1\over 4N^2}(x'+y')^2 +O(r^{-1/2}+a)\,.
\end{aligned}
\end{equation}
The contribution of points $(x',y')$ outside $\Gamma$ to the left hand side of \eqref{gl71}
is estimated by $O(r^{-1/2}\Lambda^{-1})$ by the usual stationary phase theorem.
To estimate the contribution of points $(x',y')$ close to  $\Gamma$, we use lemma 
\ref{lemstat2}. 
For any value of $\theta$, one gets easily that the hypothesis of part (a) of Lemma \ref{lemstat2} holds true, and this yields the estimate 
$$\vert \int e^{ir^{3/2}\Lambda \mathcal G'_{a}}(1-\chi)(x',y')
\underline \chi (r^{1/2}x'/N, r^{1/2}y'/N,...) dx' dy'\vert \leq C(r^{3/2}\Lambda)^{-5/6}= r^{-5/4}\Lambda^{-5/6}$$
which provides the bound $Cr^{-1/4}\Lambda^{-5/6}$ on the right hand side of \eqref{gl71}. 
The proof of Lemma \ref{lemgl7} is complete.
\end{proof}
We can now assume $\vert (p,q) \vert \leq r_{0}$. There exists  
 $c>0$ independent of $N\geq 2$ such that 
\begin{equation}\label{gl72bis}
\forall (x,y)\in \R^2,\quad \vert x^2-{1\over 4N^2}(x+y)^2\vert+\vert y^2-{1\over 4N^2}(x+y)^2\vert
\geq c (x^2+y^2)\,.
\end{equation}
Thus by integration by part, \eqref{gl67} yields that
 large values of 
$(x,y)$ gives a contribution $O(\Lambda^{-\infty})$ to the integral \eqref{gl66}
. We can then replace $\underline \chi (x/N,y/N,...)$ by a symbol $\chi(x,y,...)$
compactly supported in the ball $B=\vert (x,y)\vert \leq R$ with $R$ large. 
We are left to prove
\begin{equation}\label{gl74}
\vert \int e^{i\Lambda \mathcal G_{a}} \chi  dx dy \vert \leq C \Lambda^{-3/4}\,.
\end{equation}
Uniformly in $N\geq 1$ and $\tilde T$
near $[1/2,3/2]$ and $(x,y)$ near $B$, we have
\begin{equation}\label{gl52quadter}
\begin{aligned}
\mathcal G_{a}= & \mathcal G_{0}+ O(a) \\
\mathcal G_{0}= &qx -x^3/3 + py-y^3/3 
+{\tilde T\over 2}(x+y)^2  +{1\over 12 N^2}(x+y)^3 \,.
\end{aligned}
\end{equation}
Note that the hessian of $\mathcal G_{a}$ is $\mathcal H_{a}=-2\tilde
T(x+y)+4xy-{1\over N^2}(x+y)^2+O(a)$. Therefore the set $\mathcal Z_{a}=\{(x,y); \mathcal H_{a}(x,y)=0\}$
is, for small $a$, a smooth curve in $B$ which is close to the parabola  $-2\tilde T(x+y)-(x-y)^2=0$ for $N=1$, 
and close to the hyperbola $-2\tilde T(x+y)+4xy-{1\over N^2}(x+y)^2=0$
for $N\geq 2$. Moreover, 
\begin{equation}\label{gl73}
\begin{aligned}
\partial_{x}\mathcal G_{a}= & q-\mathcal X(x,y,a); \ \mathcal X(x,y,a)=x^2 -\tilde T(x+y)-{1\over 4N^2}(x+y)^2 +O(a) \\
\partial_{y}\mathcal G_{a}= & p-\mathcal Y(x,y,a); \ \mathcal Y(x,y,a)=y^2 -\tilde T(x+y)-{1\over 4N^2}(x+y)^2 +O(a) \\
\end{aligned}
\end{equation}

It remains to use Lemma \ref{lemstat2} near any point $(q,p)$, $\vert (p,q) \vert \leq r_{0}$. If $(q,p)$
is not in the image of $\mathcal Z_{a}$ by the map $(\mathcal X,\mathcal Y)$,
then near $(q,p)$, the estimate \eqref{gl74} holds true with a factor $C\Lambda^{-1}$ on the
right hand side by the usual stationary phase theorem. If $(q,p)$ is  in the image of $\mathcal Z_{a}$ by the map $(\mathcal X,\mathcal Y)$, but $(q,p)\not= (0,0)$,
then one easily verifies that part (a) of Lemma \ref{lemstat2} applies, and this gives
near $(q,p)$ the estimate \eqref{gl74}  with a factor $C\Lambda^{-5/6}$ on the
right hand side. Finally, near $(q,p)=(0,0)$, one has $(x,y)$ near $(0,0)$, and 
one easily verifies that part (b) of Lemma \ref{lemstat2} applies, and therefore 
 \eqref{gl74} holds true.

This concludes the proof of Proposition \ref{propL2}.\cqfd

\begin{rem} Since the symbol
$\chi$ of degree $0$ is elliptic at $(x,y)=(0,0)$, the estimate \eqref{gl74} is optimal. To see this point, 
it is sufficient to apply part (b) of Lemma \ref{lemstat2} at $(p,q)=(0,0)$.
Observe that by \eqref{gl52bis}, $(p,q)=(0,0)$ is equivalent to $X=F_{0}=1$ and
$T=4N\sqrt{1+a}$, i.e equivalent to $x=a, t=4N\sqrt{a(1+a)}$ which are precisely the times
where a swallowtail occurs in the wave front set of the Green function. This, and \eqref{gl66bis},
proves Theorem \ref{disperoptimal}.
\end{rem}

 \subsubsection{Proof of Proposition \ref{propL3}}
 In order to prove Proposition \ref{propL3}, we use as before the splitting  
 $\gamma_{1}=\gamma_{1,1}+\gamma_{1,2}$. First, we prove 
 \begin{equation}\label{gl80}
 \vert \gamma_{1,1} (T,X,Y,h) \vert \leq C (2\pi h)^{-2}(({h\over a^{1/2}T})^{1/2} +h^{1/3})\,.
 \end{equation}
Going through the proof of Proposition \ref{propL1}, we notice  only
 one difference between the case $N=1$ and $N\geq 2$: namely for $T\in ]0,T_{0}]$,
 when we estimate
 $\int e^{i\lambda \Phi_{1,+,+}}\Theta _{+,+}(z;a\lambda)dz$, we may have a critical point
 associated to a very large value of $z$. Let $\chi(z)\in C_{0}^\infty(]z_{1},\infty[)$
 with $z_{1} $ large, and set
 \begin{equation}\label{gl81}
 J=\int e^{i\lambda \Phi_{1,+,+}}\Theta _{+,+}(z;a\lambda)\chi(z)dz\,.
  \end{equation} 
  We will prove
  $$
a^{1/2}\vert J \vert \leq C a^{1/2}\lambda^{-1/2}T^{-1/2}\,
$$
  which clearly gives the first term on the right of the inequality \eqref{gl80}. One has
  \begin{equation}\label{gl82}
  \begin{aligned}
  \partial_{z}\Phi_{1,+,+}=&{T\over 2}(1+az)^{-1/2}-{z^{-1/2}\over 2}(1+X)+ O(z^{-3/2})\\
  \partial^2_{z}\Phi_{1,+,+}=&{-Ta\over 4}(1+az)^{-3/2}+{z^{-3/2}\over 4}(1+X)+ O(z^{-5/2})\,.
  \end{aligned}\end{equation}
  Therefore, to get a large critical point $z_{c}$, $T$ must be small
  (recall $T\leq T_{0}$  means $t\leq a^{1/2}T_{0}$). One has then $z_{c}^{-1/2}\simeq T$
 and from \eqref{gl82} we get  $\partial^2_{z}\Phi_{1,+,+}(z_{c})\simeq T^3$. Recall
 that $\Theta_{+,+}(z,a,\lambda)$ is a classical symbol in $z$ of degree $-1/2$, thus
 $T^{-1}s^{1/2}\Theta_{+,+}(T^{-2}s,a,\lambda)$ is a symbol of degree $0$ 
 in $s\geq s_{0}>0$,  uniformly in $T\in ]0,T_{0}]$. Therefore,
 if we  perform the change of variable $z=T^{-2}s$ in \eqref{gl81}, we
 get $\vert J\vert \leq C\lambda^{-1/2}T^{-1/2}$ by stationary phase.

 It remains to prove that the inequality \eqref{gl1ter} holds true for $\gamma_{1,2}$.
 The only place where  $N\geq 2$ gets used in the proof of Proposition \ref{propL2}
 is  Lemma \ref{lemgl7} and inequality \eqref{gl72bis}. But for $N=1$, since $\underline \chi (x,y,...)$
 is compactly supported in $(x,y)$, we do not need the inequality 
 \eqref{gl72bis}. Moreover, we get from \eqref{gl67} that the phase
 $\mathcal G_{a}$ has no critical points on the support of $\underline \chi$
 for $\vert (p,q) \vert \geq r_{0}$ if $r_{0}$ is large, and this implies 
 for $\vert (p,q) \vert \geq r_{0}$
\begin{equation}\label{gl70bis}
\left \vert \int e^{i\Lambda \mathcal G_{a}}\underline \chi (x, y,...) dx dy \right\vert
\in O (\Lambda^{-\infty})\,.
\end{equation}
In fact, Lemma \ref{lemgl7} is telling us that the constant $C$ in the right of
\eqref{gl70} is uniform for large $N$. The proof of Proposition
\ref{propL3} is now complete.\cqfd

  \section{Parametrix for $0<a\leq h^{1/2}$}\label{secgal}
  We will write the initial
  data with the help of gallery modes, which we first describe in
  connection with the spectral analysis of our Laplace operator. We
   describe the corresponding solutions of the wave operator. We then estimate their $L^{\infty}(\Omega)$ norm for tangent initial directions by
  using Sobolev embedding, taking advantage of the size of the Fourier
  support. We deal with the non-tangent
  initial directions  by constructing a crude
  parametrix, relying partly on gallery modes and the asymptotics of
  the Airy function.

\subsection{Whispering gallery modes}\label{wgm}

Let $\Omega=\{(x,y)\in\mathbb{R}^{2}| x>0,y\in\mathbb{R}\}$ denote the half-space $\mathbb{R}^{2}_{+}$ with the Laplacian given by $\Delta_D=\partial^2_x+(1+x)\partial^2_y$ with Dirichlet boundary
condition on $\partial\Omega$. Taking the Fourier transform in the $y$-variable gives
\begin{equation}\label{deltaeta}
-\Delta_{D,\eta}=-\partial^{2}_{x}+(1+x)\eta^{2}.
\end{equation} 
For $\eta\neq 0$, $-\Delta_{D,\eta}$ is a self-adjoint,
positive operator on $L^{2}(\mathbb{R}_{+})$ with compact resolvent. Indeed, the potential $V(x,\eta)=(1+x)\eta^{2}$ is bounded from below, it is continuous and $\lim_{x\rightarrow\infty}V(x,\eta)=\infty$. Thus one can consider the form associated to $-\partial^{2}_{x}+V(x,\eta)$,
\[
Q(u)=\int_{x>0}|\partial_{x} v|^{2}+V(x,\eta)|v|^{2}dx,\quad D(Q)=H^{1}_{0}(\mathbb{R}_{+})\cap\{v\in L^{2}(\mathbb{R}_{+}), (1+x)^{1/2}v\in L^{2}(\mathbb{R}_{+}))\},
\]
which is clearly symmetric, closed and bounded from below by a positive constant $c$. If $c\gg 1$ is chosen such that $-\Delta_{D,\eta}+c$ is invertible, then $(-\Delta_{D,\eta}+c)^{-1}$ sends $L^{2}(\mathbb{R}_{+})$ in $D(Q)$ and we deduce that $(-\Delta_{D,\eta}+c)^{-1}$ is also a (self-adjoint) compact operator. The last assertion follows from the compact inclusion
\[
D(Q)=\{v| \partial_{x}v, (1+x)^{1/2}v\in L^{2}(\mathbb{R}_{+}), v(0)=0\}\hookrightarrow L^{2}(\mathbb{R}_{+}).
\]
We deduce that there exists a base of eigenfunctions $v_{k}$ of
$-\Delta_{D,\eta}$ associated to a sequence of eigenvalues
$\lambda_{k}(\eta)\rightarrow\infty$. From $-\Delta_{D,\eta}v=\lambda
v$ we obtain $\partial^{2}_{x}v=(\eta^{2}-\lambda+x\eta^{2})v$,
$v(0,\eta)=0$ and after a suitable change of variables we find that an
orthonormal basis of $L^2([0,\infty[)$ is given by eigenfunctions 
\begin{equation}\label{ek}
e_{k}(x,\eta)=f_k\frac{\eta^{1/3}}{k^{1/6}}Ai(\eta^{\frac{2}{3}}x-\omega_{k}),
\end{equation} 
where $(-\omega_{k})_{k}$ denote the zeros of Airy's
function in decreasing order and where $f_k$ are constants so that $\|e_k(.,\eta)\|_{L^2([0,\infty[)}=1$ for every $k\geq 1$, and all $f_{k}$'s remain in a fixed compact subset of $]0,\infty[$. 
The corresponding  eigenvalues are $\lambda_{k}(\eta)=\eta^{2}+\omega_{k}\eta^{\frac{4}{3}}$. 
\begin{rmq}\label{rmqappendixdirac}
Let $\delta_{x=a}$ denote the Dirac distribution on $\mathbb{R}_+$, $a>0$, then it reads as follows:
\begin{equation}
\delta_{x=a}=\sum_{k\geq 1}e_k(x,\eta)e_k(a,\eta).
\end{equation}
\end{rmq}
We define the gallery modes as follows:
\begin{dfn}\label{dfnWGM}
For $x>0$ let $E_{k}(\Omega)$ be the closure in $L^{2}(\Omega)$ of
\begin{equation}\label{ekabis}
\Big\{\frac{1}{2\pi}\int e^{iy\eta}
e_k(x,\eta)\psi_k(\eta)d\eta, \hat{\psi}_k\in \mathcal{S}(\mathbb{R})\Big\},
\end{equation}
where $\mathcal{S}(\mathbb{R})$ is the Schwartz space of rapidly decreasing functions,
\[\mathcal{S}(\mathbb{R})=\Big\{f\in C^{\infty}(\mathbb{R}), \|z^{\alpha} D^{\beta}f\|_{L^{\infty}(\mathbb{R})}<\infty\quad\forall\alpha,\beta\in \mathbb{N}\Big\}.
\]  
For fixed $k$, a function in $ E_{k}(\Omega)$ is called a whispering
gallery mode. 
\end{dfn}
We have the following result (see \cite{doi}):
\begin{thm}\label{thmWGM}
We have the orthogonal decomposition $(L^{2}(\Omega),\Delta_D)=\bigoplus_{\bot}E_{k}(\Omega)$, 
where $E_k(\Omega)$ denotes the space of gallery modes associated to the $k$th zero of the Airy function $Ai$ and where $\Delta_D=\partial^2_x+(1+x)\partial^2_y$ with Dirichlet boundary condition on $\partial\Omega$.
\end{thm}
\begin{proof}
Indeed, from \cite{doi}[Section 2.2] one can easily see that $(E_{k}(\Omega))_{k}$ are closed, orthogonal and that $\cup_{k}E_{k}(\Omega)$ is a total family (i.e. that the vector space spanned by $\cup_{k}E_{k}(\Omega)$ is dense in $L^{2}(\Omega)$).
\end{proof}
Let $\psi_{j}\in C^{\infty}_0(]0,\infty[)$ as in Remark \ref{remspectraly}. Using Remark \ref{rmqappendixdirac}, for $h\in (0,1]$, we write the initial data, localized at frequency $\frac 1h$, as follows
\begin{multline}\label{psidelta}
\psi_{2}(h\sqrt{-\Delta_g})\psi_{1}(hD_y)\delta_{x=a,y=0}=\\
\sum_{1\leq k}
{1 \over 2\pi h}\int e^{\frac ih y\eta}\psi_{2}(\eta\sqrt{1+\omega_{k}(h/\eta)^{2/3}})\psi_{1}(\eta)e_k(a,\eta/h)e_k(x,\eta/h)d\eta\,.
\end{multline}
Observe that in the sum over $k$, by Remark \ref{remspectraly} we may assume $k\leq \varepsilon h^{-1}$ with  $\varepsilon$ small. From \eqref{psidelta} we get 
\begin{multline}\label{psideltabis}
u_{a,h}(t,x,y)=e^{-it\sqrt{-\Delta_g}}\psi_{2}(h\sqrt{-\Delta_g})\psi_{1}(hD_y)\delta_{x=a,y=0}\\
=\sum_{1\leq k}
{1 \over 2\pi h}\int e^{\frac ih
  (y\eta-t\eta\sqrt{1+\omega_{k}(h/\eta)^{2/3}})}\psi_{2}(\eta\sqrt{1+\omega_{k}
(h/\eta)^{2/3}})\psi_{1}(\eta)e_k(a,\eta/h)e_k(x,\eta/h)d\eta\,.
\end{multline}
Our goal  is to prove the following proposition.
\begin{proposition}\label{propadisp}
There exists $C$ such that for every $h\in ]0,1]$, every $0<a\leq h^{1/2}$ and every
$t\in [-1,1]$, the following holds true
\begin{equation}\label{g3-1}
\Vert \mathds{1}_{x\leq a}u_{a,h}(t,x,y)\Vert_{L^\infty}\leq Ch^{-2}\min(1,h^{1/4}+(\frac {h}{\vert t \vert})^{1/3}) \,.
\end{equation}
\end{proposition}
This proposition will be proved  in the next two sections. Proposition \ref{propadisp}
clearly implies Theorem \ref{disper} for $a\leq h^{1/2}$. By time
symmetry, we may restrict ourselves to positive times $t\in
[0,1]$. Notice that the proof for the wave propagator
$\exp(+it\sqrt{-\Delta_g})$ is exactly the same as the sign plays no
role whatsoever.
\subsection{Tangential initial directions}\label{secta<}
In this section, we make use of the Sobolev embedding properties
related to the orthogonal basis $(e_{k})$.
\begin{lemma}\label{lemsob}
There exists $C_{0}$ such that for  $L\geq 1$ the following holds true
\begin{equation}\label{g3-2}
\sup_{b\in \R}\Big (\sum_{1\leq k\leq L}k^{-1/3}Ai^2(b-\omega_{k})\Big)\leq C_{0}L^{1/3}.
\end{equation}
\end{lemma}
\begin{proof}
From $\vert Ai(x)\vert \leq C(1+\vert x \vert)^{-1/4}$, we get
$$
J(b)=\sum_{1\leq k\leq L}k^{-1/3}Ai^2(b-\omega_{k})\lesssim \sum_{1\leq k\leq L}
k^{-1/3} {1\over 1+ \vert b-\omega_{k}\vert^{1/2}}\,.
 $$
From $\omega_{k}\simeq k^{2/3}$, we get easily with $C$ independent of $L$ and $D$ large enough
$$
 \sup_{b\leq 0}J(b) \leq CL^{1/3}, \quad  \sup_{b\geq DL^{2/3}}J(b)
 \leq CL^{1/3}\,.
$$
Thus we may assume $b=L^{2/3}b'$ with $b'\in [0,D]$. Since $\omega_{k}=k^{2/3}g(k)$
with $g$ being an elliptic symbol of degree $0$, we are left to prove that 
$$I(x)=L^{-1/3} \sum_{1\leq k\leq L}
{(k/L)}^{-1/3} {1\over 1+ L^{1/3}\vert x-({k/L})^{2/3}\vert^{1/2}} $$
satisfies $\sup_{x\in \R}I(x)\leq C_{0}L^{1/3}$. Since we can split $[0,1]$ into a finite
union of intervals on which the function ${t^{-1/3} \over 1+ L^{1/3}\vert x-t^{2/3}\vert^{1/2}}$
is monotone, and since each term in the sum is bounded by $1$, we get
$$ 
I(x) \lesssim C+L^{2/3}\int_{0}^1 {t^{-1/3} \over 1+ L^{1/3}\vert x-t^{2/3}\vert^{1/2}}dt
\leq Cte + L^{1/3}\int_{0}^1 {3\over 2\vert x-s\vert^{1/2}}ds\,,
$$
and the proof of Lemma \ref{lemsob} is complete.
\end{proof}

Let $u_{a,h,<L}$ be the function defined by \eqref{psideltabis} with the sum  restricted
to $k\leq L$. From \eqref{ek}, Lemma \ref{lemsob}, and Cauchy-Schwarz inequality, one gets
\begin{equation}\label{g3-3}
\Vert  u_{a,h,<L}(t,x,y)\Vert_{L^\infty}\leq C_{1}h^{-2}h^{1/3}L^{1/3}\,.
\end{equation}
 Taking $L=C/h$, we get that Proposition \ref{propadisp} holds true for $\vert t \vert \leq h$.
 With $L=h^{-1/4}$ or $L=1/\vert t \vert$, one sees also that \eqref{g3-1} holds true
 for $u_{a,h,<L}$. Thus we are reduced to proving that  \eqref{g3-1} holds true for
$u_{a,h,>L}$, which is defined by the sum over $k\geq L$ with $L \geq
D\max(h^{-1/4}, 1/\vert t \vert)$ with a large constant $D$, and where
$|t|>h$.

\subsection{Non-tangential initial directions }\label{sectntga<}
In this section, we denote by $v_{a,h}(t,x,y)$ the function defined 
for $h\leq t \leq 1$ by \eqref{psideltabis} 
with the sum restricted to $L\leq k \leq \varepsilon/h$, $L \geq D\max(h^{-1/4}, 1/t)$, $D>0$ large and $\varepsilon>0$ small. For each value of $k$, we set
\begin{equation}\label{g3-4}
\lambda=t\omega_{k}h^{-1/3}, \quad \mu= {ah^{-1/3}\over t \omega_{k}^{1/2}}\,.
\end{equation} 
From $\omega_{k}\simeq k^{2/3}$, $k\geq 1/t$, and $t\geq h$, one has,
for some $c>0$, $\lambda\geq c$; thus we will take $\lambda$ as our large parameter. However, the parameter
$\mu$ just satisfies
\begin{equation}\label{g3-5}
 0 \leq \mu \lesssim {h^{1/6}\over t}\min(t^{1/3},h^{1/12})
\end{equation} 
and thus may be small or arbitrary large. Observe that for $D$ large enough,
for $k\geq Dh^{-1/4}$ and $0\leq x \leq a\leq h^{1/2}$ one has
$$\omega_{k}-xh^{-2/3}\eta^{2/3}\geq \omega_{k}/2$$
for all $\eta$ in the support of $\psi_{1}$. Therefore we can use the asymptotic expansion
of the Airy function from section \ref{appairy}, with $\omega=e^{i\pi/4}$,

$$
Ai(\zeta)=\sum_{\pm}\omega^{\pm1}e^{\mp\frac 23 i
  (-\zeta)^{3/2}}(-\zeta)^{-1/4}\Psi_{\pm}(-\zeta)
$$
which is valid for $-\zeta
>1$, with $-\zeta=\omega_k-\eta^{2/3}h^{-2/3}x\geq
\omega_k/2\geq\frac{\omega_1}{2}>1$ since $\omega_1\approx 2.33$.
We thus get from \eqref{psideltabis}
\begin{equation}\label{g3-6}
v_{a,h}=\sum_{L\leq k\leq \varepsilon/h}w_{k}, \quad 
w_k= {1\over 2\pi h}\sum_{\pm,\pm}\int e^{i\lambda \Phi^{\pm,\pm}_{k}}\sigma^{\pm,\pm}_{k}d\eta
\end{equation}
The phases $\Phi^{\pm,\pm}_{k}$ and the symbols $\sigma^{\pm,\pm}_{k}$ of $w_k$ read as follows,
with the notation $z=h^{2/3}\omega_{k}\eta^{-2/3}\geq 2a$:
\begin{equation}\label{Phikpmpm}
h\lambda\Phi^{\pm,\pm}_{k}(t,x,y,\eta,a)=\eta \Big(y-t\sqrt{1+z}\mp\frac 23
(z-x)^{3/2}\mp\frac 23(z-a)^{3/2}\Big),
\end{equation}

\begin{multline}\label{sigmakpmpm}
\sigma^{\pm}_{k,\pm}(x,\eta,a,h)=h^{-1/3}\eta \psi_{1}(\eta)\psi_{2}(\eta\sqrt{1+z})\frac{f^2_k}{k^{1/3}}
\omega^{\pm}\omega^{\pm}\\
\times (z-x)^{-1/4}(z-a)^{-1/4}\Psi_{\pm}(\eta^{2/3}h^{-2/3}(z-x))\Psi_{\pm}(\eta^{2/3}h^{-2/3}(z-a)).
\end{multline}

One has $3\eta\partial_{\eta}= -2 z\partial_{z}$ and for $0\leq x \leq a \leq 2z$,
$$
\vert (z\partial_{z})^j((z-x)^{-1/4})\vert \leq C_{j}z^{-1/4}\leq
C'_{j}(hk)^{-1/6}\,;
$$
moreover, $\Psi_{\pm}$ are classical symbols of degree $0$ at infinity and 
$$
\vert \eta^{2/3}h^{-2/3}(z-x)\vert \geq \omega_{k}/2\geq Ch^{-1/6}\,,
$$
 since $k\geq L\geq D h^{-1/4}$. Therefore we get
from \eqref{sigmakpmpm} that for all $j$, there exists $C_{j}$ independent of $h,k,a,x, \eta$
such that
\begin{equation}\label{g3-7}
\vert \partial_{\eta}^j \sigma^{\pm,\pm}_{k}(x,\eta,a,h) \vert \leq C_{j} (hk)^{-2/3}\,.
\end{equation}
\begin{prop}\label{propdispumntg<}
For $\varepsilon$ small, there exists $C$ independent of $a\in (0,h^{1/2}]$, $t\in [h,1]$, $ x\in[0,a]$, $y\in \R$
and $k\in [L, \varepsilon/h]$ such that the following holds true
\begin{equation}\label{g3-8}
\vert \int e^{i\lambda \Phi^{\pm,\pm}_{k}}\sigma^{\pm,\pm}_{k}d\eta\vert 
\leq C (hk)^{-2/3}\lambda^{-1/3}\,.
\end{equation}
\end{prop}

 Observe that from 
 \eqref{g3-6} and the definition  \eqref{g3-4} of $\lambda$, \eqref{g3-8} implies
 
 \begin{equation}
 \begin{aligned}
 \Vert \mathds{1}_{x\leq a}v_{a,h}(t,x,y)\Vert_{L^\infty}& \leq Ch^{-1}
 \sum_{k\leq \varepsilon/h}(hk)^{-2/3}t^{-1/3}h^{1/9}k^{-2/9}\\
 &=Ch^{-2}({h\over t})^{1/3} h^{1/9}(\sum_{k\leq \varepsilon/h}k^{-8/9}) \leq C'h^{-2}({h\over t})^{1/3}
 \end{aligned}\end{equation}
 and therefore Proposition \ref{propadisp} holds true for $v_{a,h}$.
\begin{proof}
Since from \eqref{g3-7} the $(hk)^{2/3}\sigma^{\pm,\pm}_{k}$ are classical symbols of degree
$0$ compactly supported in $\eta$, we intend to apply the stationary phase to an integral of the form
\begin{equation}\label{g3-9}
J=\int e^{i\lambda \Phi^{\pm,\pm}_{k}} gd\eta
\end{equation}
with $g$ a classical symbol of degree
$0$ compactly supported in $\eta$. We have to prove uniformly with respect to the parameters
the inequality
\begin{equation}\label{g3-10}
\vert J \vert \leq C\lambda^{-1/3}.
\end{equation} 
Differentiating the phase with
respect to $\eta$ yields
\[
h\lambda \partial_{\eta}\phi^{\pm,\pm}_{k}= y-t\frac{1+\frac 23 z}{\sqrt{1+z}}\pm\frac 23
 x (z-x)^{1/2}\pm\frac 23 a (z-a)^{1/2},
\]
where the two $\pm$ signs are independent from each other (thus, we
have $4$ cases to consider). Let  $\delta={x\over a} \in [0,1]$,  $\alpha={a \over h^{2/3}\omega_{k}}$ and $s=\eta^{-2/3}\in [s_{0},s_{1}]$. Since $D$ is large, one has $\alpha\in [0,c_{0}]$
with $c_{0}$ such that $\eta^{-2/3}=s\geq s_{0}\geq 2c_{0}$ on the support of  $\psi_{1}(\eta)$.
Let $X={y-t\over t\omega_{k}h^{2/3}}$ and define the function $g(z)$ by 

$$ \frac{1+\frac 23 z}{\sqrt{1+z}}=1+zg(z), \quad g(z)=\frac 16 + {z\over 24}+ O(z^2).$$
Then the derivative of the phase is equal to
\begin{equation}
\partial_{\eta}\phi^{\pm,\pm}_{k}=X-sg(h^{2/3}\omega_{k}s)+\frac 23 \mu \theta^{\pm,\pm}, \quad
\theta^{\pm,\pm}=\pm \delta (s-\delta \alpha)^{1/2}\pm (s-\alpha)^{1/2}.
\end{equation}
 We now study critical points. We take $s=\eta^{-2/3}$ as variable and we get
\begin{equation}\label{g3-11}
\begin{aligned}
\partial_{s}\partial_{\eta}\phi^{\pm,\pm}_{k}&= -(g(z)+zg'(z))+{\mu\over 3}\Big(\pm \delta (s-\delta \alpha)^{-1/2}\pm (s-\alpha)^{-1/2}\Big) \\
\partial_{s}^2\partial_{\eta}\phi^{\pm,\pm}_{k}&=-h^{2/3}\omega_{k}(2g'(z)+zg''(z))-{\mu\over 6}\Big(\pm \delta (s-\delta \alpha)^{-3/2}\pm (s-\alpha)^{-3/2}\Big).
\end{aligned}
\end{equation} 
\begin{lemma}\label{lemphi3+}
For $\varepsilon$ small enough, there exists $c>0$ independent of $k\leq \varepsilon/h$ such that
$$ \vert \partial_{s}\partial_{\eta}\phi^{\pm,\pm}_{k}\vert+ \vert \partial_{s}^2\partial_{\eta}\phi^{\pm,\pm}_{k}\vert \geq c.$$
\end{lemma} 
\begin{proof}
One has $(s-\alpha)^{-1/2}\geq \delta(s-\delta\alpha)^{-1/2}$; for $\varepsilon$ small,
$z=h^{2/3}\omega_{k}s$ is small and thus $g(z)+zg'(z)$ is close to
$\frac 16$. Thus we get
$$
\vert \partial_{s}\partial_{\eta}\phi^{\pm,-}_{k}\vert \geq 1/10\,.
$$
The derivative
$\partial_{s}\partial_{\eta}\phi^{\pm,+}_{k}$ may vanish but in case
$\vert \partial_{s}\partial_{\eta}\phi^{\pm,+}_{k}\vert \leq 1/100$, the first line of
\eqref{g3-11} implies
$$
{\mu\over 3}(s-\alpha)^{-1/2}\geq 0.05\,.
$$
The second line of
\eqref{g3-11} then gives a positive lower bound on 
$\vert \partial_{s}^2\partial_{\eta}\phi^{+,+}_{k}\vert$. It remains to study 
$\phi^{-,+}_{k}$. For any function $f$, one has
\begin{equation}\label{g3-12}
f(s-\alpha)-\delta f(s-\delta\alpha)=(1-\delta)f(s-\delta\alpha)-
\int_{0}^{\alpha(1-\delta)}f'(s-\delta\alpha -t)dt\,.
\end{equation}
Taking $f(t)=t^{-1/2}$, we thus find that 
$$
\vert \partial_{s}\partial_{\eta}\phi^{-,+}_{k}\vert \leq
1/100\Longrightarrow \mu(1-\delta)\geq c>0\,.
$$
 If one applies \eqref{g3-12} with 
$f(t)=t^{-3/2}$, we then find that for $\varepsilon$ small, the second line of
\ref{g3-11} implies $\vert \partial_{s}^2\partial_{\eta}\phi^{-,+}_{k}\vert\geq  c/2$.
The proof of Lemma \ref{lemphi3+} is complete. 
\end{proof}
From Lemma \ref{lemphi3+} and \ref{lemstat1}, we get that Proposition \ref{propdispumntg<}
holds true in the case where the parameter $\mu$ is bounded, since in that case
all the derivatives of order $\geq 2$ of the phase $\phi^{\pm,\pm}_{k}$ are bounded.
It remains to study the case where $\mu$ is large. 

In cases $(+,+)$ or $(-,-)$, and $\mu$ large, we can take as large parameter 
$\Lambda=\lambda\mu$. Since $(s-\alpha)^{-1/2}+\delta (s-\delta\alpha)^{-1/2}\geq c>0$,
we get in that case that \eqref{g3-8} holds true with a better factor $(hk)^{-2/3}\Lambda^{-1/2}$ on the right hand side.

It remains to study the cases $(+,-)$ and  $(-,+)$ for $\mu$ large. But in these cases,
we can use \eqref{g3-12}: therefore, if $\mu(1-\delta)$ is bounded, all the derivatives of order $\geq 2$ of the phase $\phi^{\pm,\mp}_{k}$ are bounded, and therefore from Lemmas \ref{lemphi3+} and \ref{lemstat1}, we get that Proposition \ref{propdispumntg<}
holds true.

Finally, in the cases $(+,-)$ and  $(-,+)$ and $\mu (1-\delta)$ large,
we can take as large parameter 
$\Lambda'=\lambda\mu(1-\delta)$, and since by \eqref{g3-12} one has 
$$
\vert (s-\alpha)^{-1/2}-\delta (s-\delta\alpha)^{-1/2}\vert \geq
c(1-\delta)\,,
$$
 with $c>0$,
we get in that case that \eqref{g3-8} holds true with a better factor
$(hk)^{-2/3}\Lambda'^{-1/2}$ on the right hand side.

The proof of Proposition \ref{propdispumntg<} is now complete.
\end{proof}
This concludes the proof of Proposition \ref{propadisp}.
\section{Dimension $d\geq 3$}\label{dim3}

Let $d\geq 3$ and $\Omega_d=\{(x,y)\in\mathbb{R}_+\times\mathbb{R}^{d-1}\}$ with Laplace operator $\Delta_d=\partial^2_x+(1+x)\triangle_{y}$. The normal variable is still denoted $x>0$, and the boundary is still defined by the condition $x=0$.  Proofs of Theorems
\ref{disper} and \ref{disperoptimal} follow exactly along the same
line as in the $2d$ case,  for both  $a\lesssim h^{1/2}$ and $a\gg h^{4/7}$.

\subsection{Parametrix for $a\gg h^{4/7}$}
In higher dimensions the parametrix construction is identical to the one in the two dimensional case. We set $\hbar=h/\vert \eta \vert$, and we define $v(t,x,y,h)$ with $y\in \mathbb R^{d-1}$ by 
\begin{equation}\label{g4-2}
\begin{aligned}
 v(t,x,y,h)=&\sum_{0\leq N\leq C_0/\sqrt a} v_N(t,x,y,h)\\
  v_N(t,x,y,h)=&\frac 1 {(2\pi h)^d} \int e^{i\frac \eta h y}
    u_N(t,x,h/\vert\eta\vert) \vert\eta\vert \chi_0(\vert \eta \vert) \,d\eta,
\end{aligned}
\end{equation}
with the same $ u_N(t,x,h/\vert\eta\vert)$ as before. We take polar
coordinates in $\eta\in \mathbb R^{d-1}$, $\eta=\vert\eta\vert \omega$.
We thus get
\begin{equation}\label{g4-3}
  v_N(t,x,y,h)=\frac 1 {(2\pi h)^d}\int (\int e^{i{\vert \eta \vert \over h} (y.\omega)}d\omega)
    u_N(t,x,h/\vert\eta\vert) \vert\eta\vert^d \chi_0(\vert \eta \vert) \,d\vert\eta\vert,
\end{equation}
In the above formula, apart from the harmless factor
$\vert\eta\vert^d$ instead of $\vert\eta\vert$, we have a
superposition with respect to $\omega\in S^{d-2}$ of functions of the
same type as before, which are evaluated at $z=y.\omega$. We shall use the
following lemma.
\begin{lemma}\label{lemmed}
Let $\psi_{j}\in C_{0}^\infty(]0,\infty[)$. There exists $c_{0}>0$ such that for
every $a\in ]0,1]$ and every $t\in [h,1]$ 
\begin{equation}\label{g4-1}
h^d\vert \psi_{1}(h\sqrt{-\Delta_{d}})\psi_{2}(h\vert D_{y} \vert) e^{\pm it \sqrt{-\Delta_{d}}}
\delta_{x=a,y=0}\vert_{L^\infty (x\leq a,\vert y \vert \leq c_{0}t)} \in O(({h\over t})^\infty)
\end{equation}
\end{lemma}
\begin{proof} We may and will assume $a\leq 2t$; In fact, for $t\leq a/2$, by finite speed
of propagation, the singular support of $e^{\pm it \sqrt{-\Delta_{d}}}
\delta_{x=a,y=0}$ has not reached the boundary $x=0$, and then \eqref{g4-1} is a simple
consequence of propagation of singularities in the interior (see the argument below).
Let $T\in [h,1]$ be given; perform the change of variable $t=Ts, x=TX, y=TY$,
and set $f_{T}(s,X,Y)=f(Ts,TX,TY)$. Then one has

$$ (\Delta_{d}f)_{T}=T^{-2}P_{T}f_{T}, \ P_{T}=\partial_{X}^2+ (1+TX)\triangle_{Y}$$

Set $\hbar= h/T\leq 1$. One has for any $\psi$ the identity
$(\psi(hD_{t,x,y})f)_{T}=\psi(\hbar D_{s,X,Y})f_{T}$, and therefore \eqref{g4-1}
is equivalent to the estimate at time $s=1$
\begin{equation}\label{g4-2bis}
\vert \psi_{1}(\hbar\sqrt{P_{T}})\psi_{2}(\hbar\vert D_{Y} \vert) e^{\pm i \sqrt{-P_{T}}}
\delta_{X=a/T,y=0}\vert_{L^\infty (X\leq a/T,\vert Y \vert \leq c_{0})} 
\in O(\hbar^\infty)\,.
\end{equation}
Observe that $b=a/T\leq 2$  is bounded. Since $\psi_{2}(\hbar\vert
D_{Y} \vert)$ commutes with the flow $e^{\pm i \sqrt{-P_{T}}}$, using
the Melrose-Sj\"ostrand theorem on propagation of singularities at the
boundary \cite{mesj78},
we just need to verify the following : There exists $c_{0}>0$ such that 
for any $T\in [0,1]$ and any optical ray $s\rightarrow \rho(s)$ associated to the symbol 
$\xi^2+(1+TX)\eta^2$ starting at $t=0$ from $\rho(0)=(X=b,Y=0; \xi_{0}, \eta_{0})$ with
$\xi_{0}^2+(1+Tb)\eta_{0}^2=1$ and $\vert \eta_{0} \vert \geq c_{1}>0$, one has 
$\vert Y(\rho(1))\vert \geq 4c_{0}$. But on the generalized bicharacteristic flow,
one has $\partial_{s}\eta=0$ and $\partial_{s}Y= 2\eta (1+TX(s))$ and therefore
$Y(s)=\eta_{0}(g(s))$ with $g(s)\geq 2s$ and the result is obvious. Observe
that the cutoff by $\psi_{2}(h\vert D_{y} \vert)$ is essential to get the lower bound
on $\vert \eta_{0} \vert$. The proof of Lemma \ref{lemmed} is complete.
\end{proof}

In order to prove our dispersive estimates, we may assume $h\leq t\leq 1$, and therefore
by Lemma \ref{lemmed}, we may also assume $\vert y \vert \geq c_{0}t \geq c_{0}h$. Classical 
stationary phase in $\omega\in S^{d-2}$ gives

\begin{equation}\label{g4-4}
\int e^{i{\vert \eta \vert \over h} (y.\omega)}d\omega= 
({h\over \vert \eta \vert \vert y\vert})^{{d-2\over 2}}(e^{{i\vert \eta \vert \vert y\vert\over h} } \sigma_{+}({h\over \vert \eta \vert \vert y\vert})+
e^{{-i\vert \eta \vert \vert y\vert\over h}} \sigma_{-}({h\over \vert \eta \vert \vert y\vert}))
\end{equation}
where $\sigma_{\pm}$ are classical symbols of degree $0$ in the small parameter
${h\over \vert \eta \vert \vert y\vert}$. Inserting \eqref{g4-4} in \eqref{g4-3}, 
and since for $\vert y \vert \geq c_{0}t$ and $\vert \eta\vert \in [\frac 18, 8]$
one has 
$({h\over \vert \eta \vert \vert y\vert})^{{d-2\over 2}} \leq C (h/t)^{{d-2\over 2}}$, we easily 
see that the proof of theorem \ref{disper} and \ref{disperoptimal} follows exactly like in the $2d$ case.

\subsection{Case $a\lesssim h^{1/2}$}
Indeed, the dispersive estimates  follow once we notice that
definition \ref{dfnWGM} and theorem \ref{thmWGM} extend to the $d$
dimensional domain $\Omega_d$. It is enough to define
for $x>0$, $E_{k}(\Omega_d)$ to be the closure in $L^{2}(\Omega_d)$ of
\begin{equation}\label{eka}
\Big\{\frac{1}{(2\pi)^{d-1}}\int e^{i<y,\eta>}
Ai(|\eta|^{\frac{2}{3}}x-\omega_{k})\hat{\varphi}(\eta)d\eta, \varphi\in \mathcal{S}(\mathbb{R}^{d-1})\Big\},
\end{equation}
where $\mathcal{S}(\mathbb{R}^{d-1})$ is the Schwartz space of rapidly decreasing functions,
\[\mathcal{S}(\mathbb{R}^{d-1})=\Big\{f\in C^{\infty}(\mathbb{R}^{d-1})| \|z^{\alpha} D^{\beta}f\|_{L^{\infty}(\mathbb{R}^{d-1})}<\infty\quad\forall\alpha,\beta\in \mathbb{N}^{d-1}\Big\}.
\]  

\begin{thm}
We have the orthogonal decomposition $(L^{2}(\Omega_d),\Delta_d)=\bigoplus_{\bot}E_{k}(\Omega_d)$, 
where $E_k(\Omega_d)$ denotes the space of gallery modes associated to the $k$th zero of the Airy function $Ai$ and where $\Delta_d=\partial^2_x+(1+x)\triangle_{y}$ with Dirichlet boundary condition on $\partial\Omega_d$.
\end{thm}

Therefore, by Lemma \ref{lemmed} and \eqref{g4-4}, the proof of our main theorems follows exactly like in the 
$2d$ case.

\appendix
\section{The energy critical nonlinear wave equation}\label{nonlin}
We consider the equation
\begin{equation}
  \label{eq:nlw}
  \Box_g u+|u|^{\frac 4 {d-2}} u=0
\end{equation}
with data $(u_0,u_1)\in H^1_0(\Omega_d)\times L^2(\Omega_d)$, with
$3\leq d\leq 6$. When the
domain is $\R^d$, there is a long line of seminal works regarding this
model, which may be one of the simplest model of a critical wave
equation. To our knowledge, the fist work to address the energy
setting (as opposed to $C^\infty$) is \cite{Shatah-Struwe-IMRN},
where low dimensions are dealt with, using only the oldest Strichartz
estimates (time and space exponents are equal). Higher dimensions
($d\geq 7$) have their own set of difficulties, mostly related to
the low power nonlinearity ($1+4/(d-2)<2$) and the subsequent failure of
its derivative with respect to $u$ to be Lipschitz. All these
technical annoyances may be solved one way or another, but are out of
the scope of the present paper.

Hence we contend ourselves with the low dimensions. There are essentially two things to be checked:
\begin{itemize}
\item we have a ``good'' local Cauchy theory, providing energy
  class solutions ;
\item this local Cauchy theory may be tweaked as to insert (a small
  power of) the
  potential energy of the solution in the nonlinear estimates, so that
  we can then perform the non concentration argument from \cite{BLP}
  and extend our solutions globally in time. Remark that the potential
  energy is $\|u\|_{L^{2d/(d-2)}_x}^{2d/(d-2)}$ which corresponds to the
  critical nature of the equation, as by Sobolev embedding, $H^1_0\hookrightarrow L^{2d/(d-2)}$.
\end{itemize}
We refer to \cite{ip08}, \cite{ip12} for
details on how to deal with fractional derivatives, Besov spaces on
domains and product-type estimates (alternatively, one may proceed
with interpolation as in \cite{BLP}).
\begin{rem}
  Note that in proving Theorem \ref{thStri} from Theorem \ref{disper},
  one needs, for $p>2$, an embedding $\dot B^{0,2}_p\subset L^p$ on
  domains, which may be proved directly or follows from a Mikhlin-H\"ormander multiplier theorem from Alexopoulos (see
  \cite{ip08} and references therein).
\end{rem}
 Having these tools at hand, we
may proceed exactly as in $\R^d$, provided we have the right set of
exponents.
\begin{itemize}
\item Case $d=3$: Theorem \ref{thStri} allows for the Strichartz triplet
  $(q=4,r=12, \beta=1)$ and one may proceed like in the $\R^3$ case. This was
  already observed in \cite{blsmso08} and allows for a streamlined argument
  when compared to \cite{BLP}.
\item Case $d=4$: Theorem \ref{thStri} allows for the Strichartz triplet
  $(q=11/5,r=22/3, \beta=1)$. As by Sobolev embedding we have
  $H^1_0 \hookrightarrow L^4_x$, we may write 
$$
|u|^2 u \leq  |u|^{4/5} |u|^{11/5} \in L^\infty_t L^5_x \times L^1_t
L^{5/3}_x \subset L^1_t L^2_x\,,
$$
and we may proceed as in $\R^4$.
\item Case $d=5$: Theorem \ref{thStri} allows for the Strichartz triplet
  $(q=2,r=5, \beta=1)$. As by Sobolev embedding we have
  $H_0^1\hookrightarrow L_x^{10/3}$, we may write 
$$
|u|^{4/3} u \leq  |u|^{2} |u|^{1/3} \in  L^1_t L^{5/2}_x \times
L^\infty_t L^{10}_x \subset L^1_t L^2_x\,,
$$
and we may proceed as in $\R^5$.
\item Case $d= 6$: Theorem \ref{thStri} allows for the Strichartz
  triplet $(q=2,r=18/5,\beta=1)$. By Sobolev embedding $\dot
  B^{1/3,2}_{18/5}\hookrightarrow L^4_x$, we get $u\in
  L^2_t L^4_x$ which provides a local Cauchy theory but without the
  potential energy factor. However we may estimate
$$
|u|u \in  L^2_t L^{4}_x \times L^\infty_t \dot B^{2/3,2}_{36/17}
\subset L^2_t \dot B^{2/3,2}_{18/13}\,,
$$
which is the dual endpoint Strichartz space. As we may estimate the $
\dot B^{2/3,2}_{36/17}$ norm of $u$ in term of $H^1_0$ and $L^3_x$
  norms, we now have a good local Cauchy theory, suitable to
  globalization in time.
\end{itemize}

\def\cprime{$'$} \def\cprime{$'$}

\end{document}